\documentclass[oneside]{amsart}

\usepackage{placeins}

\usepackage{hyphenat}
\usepackage{amssymb, amsmath, amsthm, amsfonts, mathtools}
\usepackage{mathabx}

\allowdisplaybreaks

\usepackage{hyperref}
\usepackage[nameinlink]{cleveref}

\usepackage{autonum} 

\usepackage{xcolor}

\usepackage{tensor}

\usepackage{wasysym}

\usepackage[normalem]{ulem}

\usepackage{pdfpages}

\usepackage{graphicx}

\usepackage{marginnote}

\usepackage{xfrac}
\usepackage{nicefrac}
\usepackage{faktor}

\usepackage{cancel}

\usepackage{slashed}

\usepackage{tikz}
\usetikzlibrary{matrix,calc,arrows.meta}
\usepackage{tikz-cd}
\usetikzlibrary{backgrounds}

\tikzcdset{
	boxedcd/.style={
		every matrix/.append style={
			draw=red,
			thick,
			fill=yellow!50!white,
			rounded corners,
			#1
		},
	},
}

\usepackage{quiver}

\usepackage{comment}

\usepackage{esvect}

\usepackage{braket}

\usepackage{enumitem}

\usepackage{nicematrix}

\usepackage{thmtools}

\usepackage{pdfpages}

\usepackage{leftindex}

\usepackage{upgreek}

\usepackage{ stmaryrd }

\theoremstyle{plain}
\newtheorem{theorem}{Theorem}[section]
\crefname{theorem}{Theorem}{Theorems}
\newtheorem{proposition}[theorem]{Proposition}
\crefname{proposition}{Proposition}{Propositions}
\newtheorem{lemma}[theorem]{Lemma}
\crefname{lemma}{Lemma}{Lemmas}
\newtheorem{corollary}[theorem]{Corollary}
\crefname{corollary}{Corollary}{Corollaries}
\theoremstyle{definition}
\newtheorem{definition}[theorem]{Definition}
\crefname{definition}{Definition}{Definitions}
\newtheorem{example}[theorem]{Example}
\crefname{example}{Example}{Examples}
\newtheorem{def-prop}[theorem]{Definition-Proposition}
\crefname{def-prop}{Definition-Proposition}{}
\newtheorem{def-lem}[theorem]{Definition-Lemma}
\crefname{def-lem}{Definition-Lemma}{}

\crefname{assumption}{Assumption}{Assumptions}
\theoremstyle{remark}
\newtheorem{remark}[theorem]{Remark}
\crefname{remark}{Remark}{Remarks}

\crefname{claim}{Claim}{Claims}

\crefname{notation}{Notation}{}
\newtheorem{question}[theorem]{Question}
\crefname{question}{Question}{Questions}

\crefname{todo}{To do}{To dos}

\crefname{problem}{Problem}{Problems}

\crefname{chapter}{Chapter}{Chapters}
\crefname{section}{Section}{Sections}
\crefname{subsection}{Subsection}{Subsections}
\crefname{figure}{Figure}{Figures}

\usepackage[figurename=Figure,tablename=Table]{caption}

\newcommand*\colvec[3][]{
	\begin{pmatrix}\ifx\relax#1\relax\else#1\\\fi#2\\#3\end{pmatrix}
}

\usepackage{mathbbol} 
\DeclareSymbolFontAlphabet{\mathbbm}{bbold}%
\DeclareSymbolFontAlphabet{\mathbb}{AMSb}%

\newcommand{\setfrac}[2]{#1/#2}

\def\black{\color{black}}

\def\B{\mathcal{B}}

\newcommand{\cl}[1]{\mathcal{#1}}
\newcommand{\bb}[1]{\mathbb{#1}}

\begin{document}

\title[Operator systems and positive extensions]
{Operator systems and positive extensions over discrete groups}


\author[E.T.A. Kakariadis]{Evgenios T.A. Kakariadis}
\address{Department of Mathematics\\ National and Kapodistrian University of Athens\\ Athens\\ 1578 84\\ Greece}
\email{evkakariadis@math.uoa.gr}

\author[M. Leimbach]{Malte Leimbach}
\address{Max Planck Institute for Mathematics, Vivatsgasse 7, 53111 Bonn, Germany.}
\email{leimbach@mpim-bonn.mpg.de}

\author[I.G. Todorov]{Ivan G. Todorov}
\address{School of Mathematical Sciences\\ University of Delaware\\ 501 Ewing Hall\\ Newark\\ DE 19716\\ USA} 
\email{todorov@udel.edu}

\author[W.D. van Suijlekom]{Walter D. van Suijlekom}
\address{Institute for Mathematics, Astrophysics and Particle Physics, Radboud	University Nijmegen, Heyendaalseweg 135, 65254 AJ Nijmegen, The Netherlands.}
\email{waltervs@math.ru.nl}

\date{March 31, 2026}
	
\black


\begin{abstract}
The extension problem asks whether positive semi-definite functions on a symmetric unital subset of a discrete group can be extended to positive semi-definite functions on the whole group.
It has been known at least since the work of Rudin in the 1960s that this is closely related to the problem of finding sums of squares factorisations of positive elements in the group C*-algebra.
We give an operator system perspective at these two problems explaining their equivalence:
the extension property is characterised by a certain quotient map on the Fourier--Stieltjes algebra, and the factorisation property by a certain complete order embedding into the group C*-algebra.
These properties are linked to the duality of the operator systems which have recently emerged from spectral and Fourier truncations in noncommutative geometry.
We exemplify how one can relate certain extension problems to operator system techniques such as nuclearity and the C*-envelope.
\end{abstract}

\maketitle

\addtocontents{toc}{\protect\setcounter{tocdepth}{1}}

\tableofcontents

\section{Introduction}

The study of positivity of functions via sums of squares of other functions has its roots in the work of Hilbert \cite{Hil88}, with a corresponding factorisation question formulated by him 
and becoming widely known as Hilbert's 17th problem
\cite{Hilbert-problems}. 
In the context of classical Fourier analysis, affirmative first results and counterexamples for positive semi-definite functions defined on subsets of abelian groups were obtained by Rudin \cite{Rud63}. 
The study of the sums-of-squares problem revealed deep and intricate connections with harmonic analysis, first in the case of abelian groups \cite{Gab98, BN00a, CP52, Rud63} and then in the non-abelian case; see e.g. \cite{BT07, BT-CMJ}.

Despite the negative general results, the failure of the characterisation of positivity as an expression of sums of squares led to fruitful semi-definite optimisation hierarchies. 
In the commutative case, these were first studied by Lasserre \cite{lasserre} and have had lasting impact in optimisation theory, while in the context of quantum information theory, sums-of-squares approximations for the problem of determining the quantum commuting value of a non-local game have been proposed with great success by Navascu\'es, Pironio and Ac\'in \cite{NPA}.
At the heart of these optimisation hierarchies lies a constrained resolution approach, where the test set (usually convex) for a given optimisation problem is represented as an increasing union (or the closure thereof) of subsets that lead to computationally amenable estimations of the cost function. 

A constrained resolution approach to noncommutative geometry has independently generated recent interest; 
here, the $\mathrm{C}^*$-algebras in the spectral triple formalism are replaced by operator systems, as a means to describe physical systems with incomplete information \cite{CvS21}.
Specifically, in the group setting, one considers the state space of spectral truncations, which leads to the problem of extending partially positive semi-definite functions to globally positive semi-definite ones. 
The goal of this paper is to present a unifying approach to the extension and factorisation problems in general (not necessarily abelian) discrete groups, using the formalism of operator systems, motivated by the main examples appearing in the context of noncommutative geometry \cite{CvS21}. 

To be more precise, in the latter context, operator systems arise in two main (and complementary) ways: 
on the one hand as compressions $PAP$ of a $\mathrm{C}^*$-algebra $A \subseteq \mathcal{B}(H)$ by a projection $P \in \mathcal{B}(H)$ (e.g.\ a spectral projection for a Dirac operator on $H$) and, on the other hand, as certain operator subsystems of $A$.
In the case of the $\mathrm{C}^*$-algebra $A = \mathrm{C}(\mathbb{T})$ of continuous complex-valued functions on the circle, the operator systems of interest are the \emph{Toeplitz system} 
\begin{align} 
\mathrm{C}(\mathbb{T})^{(N)} := \{ (T_{ij})_{i,j=0}^{N-1} \in \mathrm{M}_N(\mathbb{C}) \mid 
T_{ij} = T_{(i+1)(j+1)} \text{ for all } i,j=0, \dots, N-1 \}
\end{align}
(the addition is mod $N$), consisting of the $N \times N$ Toeplitz matrices, and the \emph{Fourier system} 
\begin{align}
\mathrm{C}(\mathbb{T})_{(N)} := \{f \in \mathbb{T} \mid \widehat{f}_n = 0 \text{ when } n \geq N\},
\end{align}
that is, the operator subsystem of $A \cong \mathrm{C}^*(\mathbb{Z})$ spanned by finitely supported sequences $(\dots,0,a_{-N+1},\dots,a_0,\dots,a_{N-1},0,\dots)$.
Using the matrix-valued Fej\'{e}r--Riesz lemma it was shown that the Toeplitz and the Fourier system are unitally completely order isomorphic to each other's operator system duals \cite{CvS21,Far21}.
As we will explain shortly, it seems appropriate to consider this duality of operator systems as a property of the group $\mathbb{Z}$ and the subset $\{1,\dots,N\} \subseteq \mathbb{Z}$;
it is then natural to ask, for which discrete groups $\Gamma$ and finite subsets $\Sigma \subseteq \Gamma$ analogous results hold.
	
To this end, for a discrete group $\Gamma$ and a finite subset $\Sigma \subseteq \Gamma$, we denote by $\mathrm{C}^*(\Gamma)^{(\Sigma)}$ the \emph{Toeplitz system}, i.e.\ the operator system of Toeplitz matrices 
\begin{align} 
(T_{st^{-1}})_{s,t\in\Sigma} \in \mathrm{M}_\Sigma \cong \mathcal{B}(\ell^2(\Sigma)),
\end{align}
and by $\mathrm{C}^*(\Gamma)_{(\Sigma\Sigma^{-1})}$ the \emph{Fourier system}, i.e.\ the operator subsystem of the group C*-algebra $\mathrm{C}^*(\Gamma)$ spanned by $\delta_{st^{-1}}$, $s,t\in\Sigma$, where $\delta_s$, $s\in\Gamma$, are the canonical unitary generators of $\mathrm{C}^*(\Gamma)$.
The Toeplitz system $\mathrm{C}^*(\Gamma)^{(\Sigma)}$ arises both as the compression of the (reduced) group $\mathrm{C}^*$-algebra by the orthogonal projection $P : \ell^2(\Gamma) \rightarrow \ell^2(\Sigma)$, and as the image of the restriction map $\rho : \mathrm{B}(\Gamma) \rightarrow \mathrm{M}_\Sigma$ of functions in the Fourier--Stieltjes algebra $\mathrm{B}(\Gamma)$ to the set $\Sigma\Sigma^{-1}$.
We mainly adopt the latter point of view.
Recalling that the dual of the group $\mathrm{C}^*$-algebra $\mathrm{C}^*(\Gamma)$ is the Fourier--Stieltjes algebra $\mathrm{B}(\Gamma)$, it is natural to expect that the Toeplitz and Fourier systems are dual to each other if and only if the restriction map $\rho$ is a quotient map.
We make this precise in terms of the involved matrix order structures by showing that the Toeplitz and the Fourier system are dual to each other if and only if every positive semi-definite Toeplitz matrix $(T_{st^{-1}})_{s,t\in\Sigma}$ with matrix entries extends to a positive semi-definite function on $\Gamma$, \autoref{thm:Characterizations-CPEP}.
	
The duality of Toeplitz matrices and sums of squares was already exploited by Rudin \cite{Rud63} in relating the extension problem for partially defined positive semi-definite functions to a Fej\'{e}r--Riesz type factorisation property
(see \cite{Sas06,Sas94,BW11} for extensive treatments, and \cite[Theorem 7.1]{BT07} for an example of the explicit use of operator system duality which is similar to our \autoref{thm:Characterizations-CPEP}). 
We capture the sums-of-squares concept by introducing the \emph{sums-of-squares operator system} $\mathrm{SOS}(\Sigma)$ and we show that the Toeplitz system $\mathrm{C}^*(\Gamma)^{(\Sigma)}$ is its operator system dual.
We infer that the Toeplitz system $\mathrm{C}^*(\Gamma)^{(\Sigma)}$ and the Fourier system $\mathrm{C}^*(\Gamma)_{(\Sigma\Sigma^{-1})}$ are each other's operator system duals if and only if the Fourier system is completely order isomorphic to the sums-of-squares system (see \autoref{thm:Characterizations-CPEP}).

Note that if the group $\Gamma$ is amenable, its group $\mathrm{C}^*$-algebra $\mathrm{C}^*(\Gamma) \cong \mathrm{C}^*_\mathrm{r}(\Gamma)$ is a coamenable compact quantum group in the sense of Woronowicz \cite{Wor87} (with comultiplication given by $\mathrm{C}^*_\mathrm{r}(\Gamma) \ni \lambda(s) \mapsto \lambda(s) \otimes \lambda(s) \in \mathrm{C}^*_\mathrm{r}(\Gamma) \otimes \mathrm{C}^*_\mathrm{r}(\Gamma)$), and the Toeplitz and Fourier systems are the operator systems arising respectively as Peter--Weyl \cite{Lei24} and Fourier truncations \cite{Rie23}.
In this way, we provide the connection between these two types of truncations.
As a byproduct we obtain that the operator system of Toeplitz matrices is not (min,max)-nuclear, thereby simplifying the corresponding argument in \cite{Far21} (see also \cite{FKPT14}).
While the non-nuclearity of the Toeplitz system is a special instance of a general characterisation of nuclear finite-dimensional operator systems \cite{Kav14}, it is in fact witnessed by the Toeplitz system itself \cite{FKPT14,Far21}. 
    
More generally, the problem of extending positive semi-definite (operator-valued) functions is usually phrased for \emph{positivity domains}, i.e.\ subsets $\Delta$ of a discrete group $\Gamma$ such that $e\in \Delta$ and $\Delta^{-1}=\Delta$.
An operator-valued function $u : \Delta \rightarrow \mathcal{B}(H)$ is called \emph{positive semi-definite} if the Toeplitz matrix 
\begin{align} 
(u(st^{-1}))_{s,t\in\Sigma} \in \mathrm{M}_\Sigma(\mathcal{B}(H)) \cong \mathcal{B}(\ell^2(\Sigma,H))
\end{align}
is positive semi-definite, for all finite subsets $\Sigma \subseteq \Gamma$ such that $\Sigma\Sigma^{-1} \subseteq \Delta$. 
A positivity domain $\Delta$ is said to have the \emph{complete extension property} if, for every $n\in\mathbb{N}$, every positive semi-definite function $u : \Delta \rightarrow \mathrm{M}_n$ admits an extension to a positive semi-definite function on $\Gamma$ (see \cite{llt}).
We observe that positive semi-definite functions on a positivity domain span a matrix ordered vector space $\mathrm{B}(\Delta)$ which we identify as the matrix ordered vector space dual of an operator system $\mathrm{UOS}(\Delta)$. 
The latter can be built up from sums-of-squares operator systems; 
in fact, we show that the operator system $\mathrm{UOS}(\Delta)$ is a colimit of a suitable diagram of sums-of-squares operator systems, \autoref{prop:Colimit}.
	
Similarly to the case of Toeplitz matrices and sums of squares, we relate the extension property for positive semi-definite functions on a positivity domain to the extension property for bounded functions on a positivity domain and a factorisation property, \autoref{th_main}.
The colimit description of the operator system $\mathrm{UOS}(\Delta)$ provides new techniques to (dis-)prove factorisation, or, equivalently, extension properties;
for instance, we discuss the positivity domain 
\begin{align}\label{eqn:Delta-five-points}
\Delta = \{(0,0), (0,\pm 1), (\pm 1,0)\} \subseteq \mathbb{Z}^2
\end{align} 
(\autoref{fig:Positivity-domain-5pts}) and show that the operator system $\mathrm{UOS}(\Delta)$ is the coproduct of two copies of the sums-of-squares systems generated by the unitaries $\delta_0, \delta_1 \in \mathbb{C}[\mathbb{Z}]$.
This allows us to compute the minimal $\mathrm{C}^*$-cover of the operator system $\mathrm{UOS}(\Delta)$ which turns out to be $\mathrm{C}^*(\mathbb{F}_2)$.
On the other hand, the minimal $\mathrm{C}^*$-cover of the Fourier system $\mathrm{C}^*(\mathbb{Z}^2)_{(\Delta)} \subseteq \mathrm{C}^*(\mathbb{Z}^2)$  generated by the same elements as the operator system $\mathrm{UOS}(\Delta)$ is $\mathrm{C}^*(\mathbb{Z}^2)$;
we conclude that the operator systems $\mathrm{C}^*(\mathbb{Z}^2)_{(\Delta)}$ and $\mathrm{UOS}(\Delta)$ are not completely order isomorphic, so the positivity domain $\Delta$ from (\ref{eqn:Delta-five-points}) does not have the complete factorisation and the complete extension property.

\begin{figure}[h]
		\centering
		\begin{tikzcd}
			{\mathrm{B}(\Gamma)} &  \mathrm{B}(\Gamma)|_\Delta & \mathrm{B}(\Delta) & {\mathrm{C}^*(\Gamma)^{(\Sigma)}}\\
			\mathrm{C}^*(\Gamma) & \mathrm{C}^*(\Gamma)_{(\Delta)}  & \mathrm{UOS}(\Delta) &  {\mathrm{SOS}(\Sigma)}
			\arrow[two heads, from=1-1, to=1-2]
			\arrow[from=1-2, to=1-3]
			\arrow[from=1-3, to=1-4]
			\arrow[from=2-4, to=2-3]
			\arrow[from=2-3, to=2-2]
			\arrow[hook', from=2-2, to=2-1]
		\end{tikzcd}
		\caption{The operator systems (in the second line) and their respective matrix ordered vector space duals (in the first line), together with the canonical ucp maps discussed in this note;
		here $\Gamma$ is a discrete group, $\Delta \subseteq \Gamma$ a positivity domain and $\Sigma \subseteq \Gamma$ a finite subset with $\Sigma\Sigma^{-1} \subseteq \Delta$.}
		\label{fig:CP-maps}
\end{figure}

\smallskip

\section*{Acknowledgments}
EK acknowledges that this research work was supported within the framework of the National Recovery and Resilience Plan Greece 2.0, funded by the European Union - NextGenerationEU (Implementation Body: HFRI. Project name: Noncommutative Analysis: Operator Systems and Nonlocality. HFRI Project Number: 015825).
ML was supported by NWO under grant OCENW.KLEIN.376, and by Max Planck Institute for Mathematics in Bonn. 
This article is based upon work from COST Action CaLISTA CA21109 supported by COST (European Cooperation in Science and Technology). 
IT was supported by NSF grants 
CCF-2115071 and DMS-2154459. 

The authors are grateful to BIRS for hospitality during the BIRS Workshop \lq\lq Operator systems and applications'' in February 2025, during which part of this work was completed. 
This material is based upon work supported by the Swedish Research Council under grant no.\ 2021-06594 while the authors were in residence at Institut Mittag-Leffler in Djursholm, Sweden, during the 2026 program on Operator Algebras and Quantum Information.

We are grateful to Matthijs Muis for valuable discussions.
Many of the results of this project have been included in ML's PhD thesis \cite{Lei25thesis}.

\section{Operator system preliminaries}\label{sec:Prelims}

\subsection{Order unit spaces} 

Let $V$ be a complex $^*$-vector space and $V_\mathrm{h}$ be the real vector space of all its hermitian elements.
A \emph{convex cone} in $V$ is a non-empty subset $\mathcal{C} \subseteq V_\mathrm{h}$ such that $[0,\infty)\cdot\mathcal{C}+\mathcal{C}\subseteq\mathcal{C}$.
A convex cone $\mathcal{C}$ in $V$ is called \emph{spanning} if $\mathcal{C}-\mathcal{C} = V_\mathrm{h}$ and \emph{proper} if $\mathcal{C} \cap (-\mathcal{C}) = \{0\}$.
A complex $^*$-vector space $V$, equipped with a proper spanning convex cone $\mathcal{C}$ is called an \emph{ordered vector space}.
We often write $V^+ := \mathcal{C}$ and call it the \emph{positive cone} of (the ordered vector space) $V$.
The positive cone $V^+$ induces a partial order on $V$ by defining $v \leq w$ if $w-v \in V^+$.
	
Given a complex $^*$-vector space $V$ with a proper convex cone $\mathcal{C}$, we say that an element $e \in \mathcal{C}$ is an \emph{order unit} if, for every hermitian element $v \in V_\mathrm{h}$, there exists a real number $r > 0$ such that $re-v \in \mathcal{C}$.
Note that if $e$ is an order unit then, for every $v \in V_\mathrm{h}$, there exists a real number $r > 0$ such that 
\begin{align} 
v = re - (re-v) \in \mathcal{C} - \mathcal{C},
\end{align}
so the cone $\mathcal{C}$ is spanning.
A $^*$-vector space $V$, equipped with a proper convex cone $V^+$ and an order unit $e_V$ is called an \emph{order unit space}.
	
For an order unit space $(V,V^+,e_V)$, let 
\begin{align}
p_V(v) := \inf\{r > 0 \mid -re_V \leq v \leq re_V\}, \ \ \ v\in V_\mathrm{h},
\end{align}
and note that $p_V$ is a semi-norm on $V_\mathrm{h}$, referred to as the \emph{order seminorm}. 
We note that any extension of $p_V$ to $V$ induces the same topology on $V$, referred to as the \emph{order topology}
\cite{PT09}. 
	
An order unit space $(V,V^+,e_V)$ (as well as its order unit $e_V$) is called \emph{Archimedean} if, for all $v \in V$, we have that 
\begin{align} 
re_V+v \geq 0 \mbox{ for all } r > 0 
\ \ \Longrightarrow \ \ 
v \geq 0.
\end{align}
If $V$ is an Archimedean order unit space, the order seminorm $p_V$ is a norm and the positive cone $V^+$ is closed in $V$ for the order topology \cite{PT09}.

If $V$ and $W$ are ordered vector spaces, a linear map $\phi : V \rightarrow W$ is called \emph{self-adjoint} if $\phi(v^*) = \phi(v)^*$, $v \in V$. 
If, furthermore, $\phi(V^+) \subseteq W^+$, the map $\phi$ is called \emph{positive}.
We denote by $\mathcal{P}(V,W)$ the cone of all positive maps $V \to W$. 
If $V$ and $W$ are order unit spaces, $\phi$ is called \emph{unital} if $\phi(e_V) = e_W$.
In this case, $\phi(V^+) \subseteq W^+$ implies that $\phi$ is self-adjoint.
If $V$ is an order unit space, unital positive linear functionals on $V$ are called \emph{states}; the set $\mathcal{S}(V)$ of all states is a convex set called the \emph{state space} of $V$. 

For ordered vector spaces $V$ and $W$, 
a positive map $\phi : V \rightarrow W$ is called an \emph{order isomorphism} if it is bijective and $\phi^{-1} : W \rightarrow V$ is positive.
An injective positive map $\phi$ is an \emph{order embedding} if $\phi$ is an order isomorphism onto its image.

For an ordered vector space $V$, let $(V^{\mathrm{d}})^+ := \mathcal{P}(V,\mathbb{C})$ be the set of positive linear functionals $V \rightarrow \mathbb{C}$ and by $V^{\mathrm{d}}$ its complex linear span.
Then $(V^{\mathrm{d}},(V^{\mathrm{d}})^+)$ is an ordered vector space;
the state space $S(V)$ is a subset of $V^{\mathrm{d}}$ which, equipped with the weak$^*$-topology, is a compact (convex) space by the Banach--Alaoglu theorem.

\subsection{Matrix order unit spaces and operator systems} 
	
For a vector space $V$, we let $\mathrm{M}_{m,n}(V)$ be the vector space of all $m$ by $n$ matrices with entries in $V$, and write $\mathrm{M}_{n}(V) = \mathrm{M}_{n,n}(V)$.
If $V$ is a $^*$-vector space then $\mathrm{M}_{n}(V)$ is also a *-vector space when equipped with the involution $(x_{i,j})_{i,j}^* := (x_{j,i}^*)_{i,j}$.
A \emph{matrix ordering} on a complex $^*$-vector space $V$ is a family $(\mathcal{C}_n)_{n\in\mathbb{N}}$ of cones such that $(\mathrm{M}_n(V),\mathcal{C}_n)$ is an ordered vector space for every $n\in\mathbb{N}$, that is \emph{compatible}
in that 
\begin{align} 
\alpha\mathcal{C}_m \alpha^* \subseteq \mathcal{C}_n \mbox{ for all }
\alpha \in \mathrm{M}_{n,m} \mbox{ and all } m,n\in\mathbb{N}.
\end{align}
We usually write $\mathrm{M}_n(V)^+$ in the place of $\cl C_n$. 
Given a matrix ordered vector space $(V,(\mathrm{M}_n(V)^+)_{n\in\mathbb{N}})$, an element $e \in V$ is called a \emph{matrix order unit} if the element
\begin{align}
e_n := \begin{pmatrix}
e &&\\
&\ddots&\\
&&e
\end{pmatrix}
\in \mathrm{M}_n(V)
\end{align}
is an order unit for the ordered vector space $(\mathrm{M}_n(V),\mathrm{M}_n(V)^+)$ for all $n\in\mathbb{N}$.
A \emph{matrix order unit space} is a matrix ordered vector space, equipped with a matrix order unit.
A matrix order unit $e$ is called \emph{Archimedean} if $e_n$ is Archimedean for all $n\in\mathbb{N}$.
An \emph{operator system} is a matrix ordered vector space, equipped with a distinguished Archimedean matrix order unit.
	
For a linear map $\phi : X \rightarrow Y$ between matrix ordered vector spaces $X$ and $Y$, we define its matrix amplification $\phi^{(n)} : \mathrm{M}_n(X) \rightarrow \mathrm{M}_n(Y)$ by letting 
\begin{align}
\phi^{(n)}((x_{ij})_{i,j}) := (\phi(x_{ij}))_{i,j}, \ \ \ (x_{ij})_{i,j} \in \mathrm{M}_n(X).
\end{align}
The map $\phi$ is called \emph{completely positive} if $\phi^{(n)} : \mathrm{M}_n(X) \rightarrow \mathrm{M}_n(Y)$ is a positive map for all $n\in\mathbb{N}$.
We denote the cone of all completely positive maps $X \rightarrow Y$ by $\mathcal{CP}(X,Y)$.
We sometimes use the abbreviations ``cp'' and ``ucp'' for ``completely positive'' and ``unital completely positive'' respectively.
We denote the categories of matrix ordered vector spaces with cp maps by $\mathbf{MVS}$, of matrix order unit spaces with ucp maps by $\mathbf{MOU}$ and of operator systems with ucp maps by $\mathbf{OSy}$.
	
Let $X$ and $Y$ be matrix ordered $^*$-vector spaces.
A completely positive map $\phi \in \mathcal{CP}(X,Y)$ is called a \emph{complete order isomorphism} if $\phi^{(n)} \in \mathcal{P}(\mathrm{M}_n(X),\mathrm{M}_n(Y))$ is an order isomorphism for every $n\in\mathbb{N}$.
We say that $X$ and $Y$ are \emph{completely order isomorphic} if there is a complete order isomorphism $X \rightarrow Y$, and write $X \cong Y$.
A completely positive map $\phi$ is a \emph{complete order embedding} if $\phi$ is a complete order isomorphism onto its image.
	
Every unital self-adjoint subspace $X \subseteq \mathcal{B}(H)$ is an operator system with positive cones defined by letting 
\begin{align}
\mathrm{M}_n(X)^+ := \mathrm{M}_n(X) \cap \mathcal{B}(H^n)^+,
\end{align}
and a matrix order unit defined by letting $e_X := \mathbf{1}_{\mathcal{B}(H)}$. 
By the Choi--Effros theorem \cite{CE77}, every operator system is unitally completely order isomorphic to such a unital self-adjoint subspace of $\mathcal{B}(H)$, for some complex Hilbert space $H$.
In particular, if $A$ is a unital $\mathrm{C}^*$-algebra then every self-adjoint subspace of $A$ that contains the unit is an operator system in a canonical fashion.
If $X \subseteq A$ is an operator system with $1_A$ as its matrix order unit, then every positive linear functional $\phi : X \rightarrow \mathbb{C}$ can be extended to a positive linear functional on $A$ by Krein's theorem. 
More generally, every cp map $\Phi : X \rightarrow \mathcal{B}(H)$ can be extended to a cp map $A \rightarrow \mathcal{B}(H)$ by Arveson's extension theorem \cite{Arv69}.
	
For a matrix ordered vector space $X$ with ordered vector space dual $X^{\mathrm{d}}$, set 
\begin{align} 
\mathrm{M}_n(X^{\mathrm{d}})^+ := \mathcal{CP}(X,\mathrm{M}_n), \ \ \ n\in\mathbb{N}.
\end{align}
If $X$ is an operator system, by \cite[Theorem 6.1]{Pau02} we have the identification
\begin{align}
\mathrm{M}_n(X^{\mathrm{d}})^+ = \mathcal{CP}(X,\mathrm{M}_n)\cong \mathcal{P}(\mathrm{M}_n(X),\mathbb{C}) = (\mathrm{M}_n(X)^{\mathrm{d}})^+.
\end{align}
If $X$ is a finite dimensional operator system, by \cite[Corollary 4.5]{CE77}, one can (non-canonically) choose an Archimedean matrix order unit for the matrix ordered vector space $X^{\mathrm{d}}$, giving it the structure of an operator system.
For (matrix) ordered vector spaces $V,W$, the adjoint $\phi^{\mathrm{d}} : W^{\mathrm{d}} \rightarrow V^{\mathrm{d}}$ of every (completely) positive map $\phi : V \rightarrow W$ is (completely) positive.

\subsection{Archimedeanisation}

Matrix order unit spaces can be turned into operator systems as follows \cite{PTT11}.
Let $X$ be a matrix order unit space.
Set 
\begin{align} 
N := \bigcap_{\phi \in \mathcal{S}(X)} \ker(\phi).
\end{align}
Then the identication 
\begin{align} 
\mathrm{M}_n(\setfrac{X}{N}) \cong \setfrac{\mathrm{M}_n(X)}{\mathrm{M}_n(N)}
\end{align}
induces a matrix order unit space structure on the quotient $\setfrac{X}{N}$. 
Define
\begin{align}
\mathcal{C}_n^\mathrm{arch}
:= 
\{&(x_{ij})_{i,j}+\mathrm{M}_n(N) \in \setfrac{\mathrm{M}_n(X)}{\mathrm{M}_n(N)} \mid \\
&r(e_X)_n+(x_{ij})_{i,j} + \mathrm{M}_n(N) \in \mathrm{M}_n(X)^+ + \mathrm{M}_n(N), \text{ for all } r>0\},
\end{align}
for all $n\in\mathbb{N}$.
We sometimes write 
\begin{align} 
\mathrm{Arch}(\mathrm{M}_n(X)^+) := \mathcal{C}_n^\mathrm{arch}
\end{align}
if explicit reference to $X$ is required.
Then the triple $(\setfrac{X}{N},(\mathcal{C}_n^\mathrm{arch})_{n\in\mathbb{N}},e_X+N)$ is an operator system, called the \emph{Archimedeanisation} of the matrix order unit space $X$; 
we denote it by $\mathrm{Arch}(X)$.
Ucp maps from the matrix order unit space $X$ to operator systems $Y$ uniquely factor through the Archimedeanisation;
more precisely, the following universal property holds:
There is a unique surjective ucp map $q : X \rightarrow \mathrm{Arch}(X)$ such that, for every ucp map $\phi : X \rightarrow Y$ with $Y$ an operator system, there exists a unique ucp map $\tilde{\phi} : \mathrm{Arch}(X) \rightarrow Y$ such that $\tilde{\phi} \circ q = \phi$.
\begin{figure}[h]
	\centering
	\begin{tikzcd}
		X & Y \\
		\mathrm{Arch}(X) &
		\arrow["\phi", from=1-1, to=1-2]
		\arrow["q"', from=1-1, to=2-1]
		\arrow["\tilde{\phi}"', from=2-1, to=1-2]
	\end{tikzcd}
	\caption{The universal property of the Archimedeanisation.}
\end{figure}

If the subspace $N$ is $\{0\}$, the cone $\mathcal{C}_n^\mathrm{arch}$ is the closure of the cone $\mathrm{M}_n(X)^+$ in the order topology.
The universal property of the Archimedeanisation implies functoriality in the following sense.

\begin{lemma}\label{lem:Fctr-Arch}
The Archimedeanisation is a functor from $\mathbf{MOU}$ to $\mathbf{OSy}$.
\end{lemma}
    
\begin{proof}
Let $X$ and $Y$ be matrix order unit spaces and let $\phi : X \rightarrow Y$ be a ucp map in $\mathbf{MOU}$.
Then there are canonical ucp maps $q_X : X \rightarrow \mathrm{Arch}(X)$ and $q_Y : Y \rightarrow \mathrm{Arch}(Y)$.
By the universal property of the Archimedeanisation there is a unique map $\tilde{\phi} : \mathrm{Arch}(X) \rightarrow \mathrm{Arch}(Y)$ such that $\tilde{\phi} \circ q_X = q_Y \circ \phi$, since the right-hand side is a ucp map from $X$ to the operator system $\mathrm{Arch}(Y)$.
\end{proof}

The Archimedeanisation functor from Lemma \ref{lem:Fctr-Arch} will be denoted by $\mathrm{Arch}$.

\subsection{Quotients}\label{subsec:Quotients}
        
We describe the construction of quotients in $\mathbf{MVS}$, and refer to \cite[Section 3]{KPTT13} for further details.
To this end, let $X$ be a matrix ordered vector space and let $J \subseteq X$ be a $^*$-invariant subspace.
Assume moreover that $J$ is a \emph{complete order ideal}, i.e.\ if $x \in \mathrm M_n(J)$ and $0 \leq y \leq x$ then $y \in \mathrm{M}_n(J)$.
Set 
\begin{align}\label{eqn:Matrix-cones-quotient}
\mathrm{M}_n(\setfrac{X}{J})^+ := \{(x_{ij}+J)_{i,j} \mid \exists y \in \mathrm{M}_n(J) \text{ such that } x+y \in \mathrm{M}_n(X)^+\}.
\end{align}
One readily checks that $\mathrm{M}_n(\setfrac{X}{J})^+$ is a convex spanning cone for the $^*$-vector space $\mathrm{M}_n(\setfrac{X}{J})$ and that 
$(\mathrm{M}_n(\setfrac{X}{J})^+)_{n\in \bb{N}}$
is a compatible family of cones.
For properness, as in the proof of \cite[Proposition 2.41]{PT09}, any element of the intersection $\mathrm{M}_n(\setfrac{X}{J})^+ \cap (-\mathrm{M}_n(\setfrac{X}{J})^+)$ can be written as $x + \mathrm{M}_n(J) = -y + \mathrm{M}_n(J)$ with $x, y \in \mathrm{M}_n(X)^+$.
Then there is an element $j \in \mathrm{M}_n(J)$ such that $j = x +y \in \mathrm{M}_n(X)^+$;
hence $0 \leq x \leq j$ which implies that $x \in \mathrm{M}_n(J)$.
So $x + \mathrm{M}_n(J) = 0 + \mathrm{M}_n(J)$, i.e.\ the intersection $\mathrm{M}_n(\setfrac{X}{J})^+ \cap (-\mathrm{M}_n(\setfrac{X}{J})^+)$ is trivial.
This shows that $(\setfrac{X}{J},(\mathrm{M}_n(\setfrac{X}{J})^+)_{n\in\mathbb{N}})$ is a matrix ordered vector space again, and we say it is a \emph{quotient} of $X$.

Note that the canonical quotient map $q : X \rightarrow \setfrac{X}{J}$ is cp.
If $X$ and $Y$ are matrix ordered vector spaces and $\phi : X \rightarrow Y$ is a cp map, then $\ker(\phi)$ is clearly $^*$-invariant; moreover, it is a complete order-ideal, since $x \in \ker(\phi^{(n)}) = M_n(\ker(\phi))$ and $0 \leq y \leq x$ imply 
\begin{align}
0 \leq \phi^{(n)}(y) \leq \phi^{(n)}(x) = 0.
\end{align}
We say that the cp map $\phi : X \rightarrow Y$ is an \emph{$\mathbf{MVS}$-quotient map} if it is surjective and the induced map $\tilde{\phi} : \setfrac{V}{\ker(\phi)} \rightarrow Y$ is a complete order isomorphism.
	
Moving to the category $\mathbf{OSy}$ \cite{KPTT13}, let $X$ be an operator system.
A subspace $J \subseteq X$ is called a \emph{kernel} if there exist an operator system $Y$ and a ucp map $\phi : X \rightarrow Y$ such that $J = \ker(\phi)$ 
(equivalently, $J$ is a kernel if and only if there is a family of states $\{\phi_\iota\}_{\iota \in I} \subseteq \mathcal{S}(X)$ such that $J = \bigcap_{\iota \in I} \ker(\phi_\iota)$).
Defining $\mathrm{M}_n(\setfrac{X}{J})^+$ as in (\ref{eqn:Matrix-cones-quotient}) and $e_{\setfrac{X}{J}} := e_X+J$, the quotient $\setfrac{X}{J}$ inherits the structure of a matrix order unit space from $X$.
The operator system $\mathrm{Arch}(\setfrac{X}{J})$ is referred to as an \emph{operator system quotient}; 
it possesses the universal property that every ucp map from $X$ to an operator system factors uniquely and canonically through it.
A kernel $J \subseteq X$ is called \emph{completely proximinal} if 
\begin{align}
\mathrm{M}_n(\setfrac{X}{J})^+ 
= 
\mathrm{Arch}(\mathrm{M}_n(\setfrac{X}{J})^+), \ \ \ n\in\mathbb{N};
\end{align}
in particular, in this case $\setfrac{X}{J} = \mathrm{Arch}(\setfrac{X}{J})$.

\subsection{Tensor products}

We recall some elements of the theory of tensor products for operator systems from \cite{KPTT11}.
Let $X$ and $Y$ be operator systems and denote by $X \otimes Y$ the corresponding algebraic tensor product.
An \emph{operator system structure} on $X \otimes Y$ is a family $\tau := (\mathcal{C}_n)_{n\in\mathbb{N}}$ which satisfies the following properties:
\begin{enumerate}
\item[(T1)] The triple $(X \otimes Y, (\mathcal{C}_n)_{n\in\mathbb{N}}, e_X \otimes e_Y)$ is an operator system, denoted $X \otimes_\tau Y$;
\item[(T2)] We have $\mathrm{M}_k(X)^+ \otimes \mathrm{M}_\ell(Y)^+ \subseteq \mathcal{C}_{k\ell}$, for all $k,\ell \in\mathbb{N}$;
\item[(T3)] For all ucp maps $\phi : X \rightarrow \mathrm{M}_k$, $\psi : Y \rightarrow \mathrm{M}_\ell$, the map $\phi \otimes \psi : X \otimes_\tau Y \rightarrow \mathrm{M}_{k\ell}$ is ucp.
\end{enumerate}
An operator system structure $\tau$ is \emph{functorial} if it satisfies:
\begin{enumerate}
\item[(T4)] For all operator systems $S, T$ and ucp maps $\phi : X \rightarrow S$, $\psi : Y \rightarrow T$, the map $\phi \otimes \psi$ is a ucp map $X \otimes_\tau Y \rightarrow S \otimes_\tau T$.
\end{enumerate}

The class of compatible families of cones over $X \otimes Y$ is non-empty; 
indeed, one such family is formed by the cones
\begin{align}\label{eqn:D-max-cones}
\mathcal{D}_n^\mathrm{max}(X,Y)
:= \{A(x \otimes y)A^* \mid x \in \mathrm{M}_k(X)^+, y \in \mathrm{M}_\ell(Y)^+, A \in \mathrm{M}_{n,k\ell}\}, 
\ \ \ n \in \mathbb{N}.
\end{align}
Its Archimedeanisation defines a functorial operator system structure  on $X \otimes Y$, called the \emph{maximal tensor product} and denoted by \black $\mathrm{max}$.
It is maximal in the sense that, for every compatible family $(\mathcal{C}_n)_{n\in\mathbb{N}}$ of positive cones over $X \otimes Y$, we have 
\begin{align}
\mathcal{D}_n^\mathrm{max}(X,Y) \subseteq \mathcal{C}_n, \ \ \ n\in\mathbb{N}.
\end{align}
In particular, if $\tau$ is any operator system structure on $X \otimes Y$ then the canonical identity map $X\otimes_{\max} Y\to X\otimes_{\tau} Y$ is completely positive.

Assume that $X$ and $Y$ are only matrix ordered vector spaces (and not necessarily operator systems).
The cones $\mathcal{D}_n^\mathrm{max}(X,Y)$ can still be defined in the same way. 
Recall, further, that a bilinear map $\phi : X\times Y\to Z$ (where $Z$ is another matrix ordered vector space) is called \emph{jointly completely positive} if 
\begin{align} 
(\phi(x_{i,j},y_{p,q}))_{i,j,p,q}\in \mathrm{M}_{nm}(Z)^+
\end{align}
whenever $(x_{i,j})_{i,j}\in \mathrm{M}_n(X)^+$ and $(y_{k,l})_{k,l}\in \mathrm{M}_m(Y)^+$. 
It follows from the arguments in \cite{KPTT11} that a bilinear map $\phi : X\times Y\to Z$ is jointly completely positive precisely when its linearisation $\tilde{\phi} : X\otimes Y\to Z$ satisfies the properties
\begin{align}
\tilde{\phi}^{(n)}(\mathcal{D}_n^\mathrm{max}(X,Y))\subseteq 
\mathrm{M}_n(Z)^+, \ \ \ n\in \bb{N}.
\end{align}
The same universal property, in the case where $X$, $Y$ and $Z$ 
are operator systems, characterises the operator system
tensor product $X\otimes_{\max} Y$. 
	
The \emph{minimal tensor product} of two operator systems $X$ and $Y$, which are by the Choi--Effros theorem unitally completely order embedded in some C$^*$-algebras $A$ and $B$ respectively, is the operator system structure inherited from the inclusion of $X \otimes Y$ in the spatial tensor product $A \otimes_\mathrm{min} B$.
This tensor product is independent of the choice of inclusions $X \subseteq A$ and $Y \subseteq B$.
Moreover it is a functorial tensor product and it is \black minimal in the sense that, for every compatible family $(\mathcal{C}_n)_{n\in\mathbb{N}}$ of positive cones over $X \otimes Y$, we have that 
\begin{align} 
\mathcal{C}_n \subseteq \mathrm{M}_n(X \otimes_\mathrm{min} Y)^+, \ \ \ n\in\mathbb{N}.
\end{align}
If $X, Y$ are finite dimensional operator systems, we have
canonical complete order isomorphisms
\begin{align}
(X \otimes_\mathrm{max} Y)^\mathrm{d} 
\cong 
X^\mathrm{d} \otimes_\mathrm{min} Y^\mathrm{d}
\ \ \mbox{ and } \ \
(X \otimes_\mathrm{min} Y)^\mathrm{d} 
\cong 
X^\mathrm{d} \otimes_\mathrm{max} Y^\mathrm{d}
\end{align}
(see \cite{KPTT11} and the references therein).

For two operator systems $X, Y$ and two operator system structures $\sigma, \tau$ on $X \otimes Y$, we say that the pair $(X,Y)$ is \emph{$(\sigma,\tau)$-nuclear} if $X \otimes_\sigma Y \cong X \otimes_\tau Y$ canonically.
We say that an operator system $X$ is $(\sigma,\tau)$-nuclear if the pair $(X,Y)$ is $(\sigma,\tau)$-nuclear, for every operator system $Y$.

\subsection{$\mathrm{C}^*$-covers}

A \emph{$\mathrm{C}^*$-cover} of an operator system $X$ is a unital C$^*$-algebra $A$ together with a unital complete order embedding $\iota : X \rightarrow A$ such that $A$ is generated by $\iota(X)$.
Every operator system $X$ admits a maximal and a minimal $\mathrm{C}^*$-cover, denoted respectively by $(\mathrm{C}^*_\mathrm{max}(X),\iota_\mathrm{max})$ and $(\mathrm{C}^*_\mathrm{min}(X),\iota_\mathrm{min})$;
they are characterised by the fact that, for any $\mathrm{C}^*$-cover $(A,\iota)$ there are unique $^*$-homomorphisms $\mathrm{C}^*_\mathrm{max}(X) \rightarrow A \rightarrow \mathrm{C}^*_\mathrm{min}(X)$ such that the diagram in \autoref{fig:C*-max-min} commutes.
\begin{figure}[h]
	\centering
	\begin{tikzcd}
		& {\mathrm{C}^*_\mathrm{max}(X)}\\
		{X} & {A}\\
		& {\mathrm{C}^*_\mathrm{min}(X)}
		\arrow["{\iota_\mathrm{max}}", from=2-1, to=1-2]
		\arrow[from=1-2, to=2-2]
		\arrow["{\iota}", from=2-1, to=2-2]
		\arrow[from=2-2, to=3-2]
		\arrow["{\iota_\mathrm{min}}"', from=2-1, to=3-2]
	\end{tikzcd}
	\caption{The universal properties of the maximal and minimal $\mathrm{C}^*$-cover.}
	\label{fig:C*-max-min}
\end{figure}
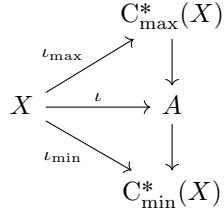
The maximal $\mathrm{C}^*$-cover is sometimes called the \emph{universal $\mathrm{C}^*$-cover}, denoted by $\mathrm{C}^*_\mathrm{u}(X)$, and its existence was shown in \cite[Proposition 8]{KW98}.
The minimal $\mathrm{C}^*$-cover is often called the \emph{$\mathrm{C}^*$-envelope} and denoted by $\mathrm{C}^*_\mathrm{e}(X)$;
its existence was first shown in \cite{Ham79} through the existence of the injective envelope; see also \cite{Pau02}.
A dilation theoretic proof was later obtained in \cite{DMcC05} through the existence of maximal dilations.

\subsection{Coproducts and amalgamated direct sums}\label{ss_ads}

Let $X$ and $Y$ be operator systems and let $V$ be an operator system, together with unital complete order embeddings $X \overset{\iota_X}{\hookleftarrow} V \overset{\iota_Y}{\hookrightarrow} Y$.
The \emph{amalgamated direct sum (of $X$ and $Y$ over $V$)} \cite{KL09} is the unique operator system $X \oplus_V Y$ which comes with ucp maps $X \overset{\phi_X}{\rightarrow} X \oplus_V Y \overset{\phi_Y}{\leftarrow} Y$ and which satisfies the following universal property:
For every operator system $Z$ and ucp maps 
\begin{align} 
X \overset{\psi_X}{\rightarrow} Z \overset{\psi_Y}{\leftarrow} Y
\text{ with } 
\psi_X \circ \iota_X = \psi_Y \circ \iota_Y,
\end{align}
there is a unique ucp map $\psi : X \oplus_V Y \rightarrow Z$ such that 
\begin{align}
\psi \circ \phi_X = \psi_X
\ \ \mbox{ and } \ \
\psi \circ \phi_Y = \psi_Y.
\end{align}
In other words, the amalgamated direct sum is the pushout in the category $\mathbf{OSy}$ of the diagram $X \overset{\iota_X}{\hookleftarrow} V \overset{\iota_Y}{\hookrightarrow} Y$, \emph{cf.}\ \autoref{fig:Amalgamated-direct-sum}.
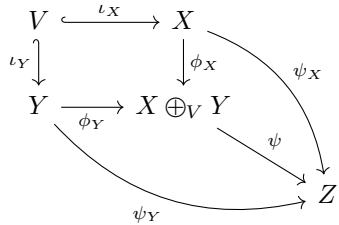
\begin{figure}[h]
	\centering
	\begin{tikzcd}
		V & X & \\
		Y & X \oplus_V Y & \\
		& & Z
		\arrow["\iota_X", hook, from=1-1, to=1-2]
		\arrow["\iota_Y"', hook', from=1-1, to=2-1]
		\arrow["\phi_X", from=1-2, to=2-2]
		\arrow["\phi_Y"', from=2-1, to=2-2]
		\arrow["\psi", from=2-2, to=3-3]
		\arrow["\psi_X", bend left=30, from=1-2, to=3-3]
		\arrow["\psi_Y"', bend right=30, from=2-1, to=3-3]
	\end{tikzcd}
	\caption{The universal property of the amalgamated direct sum.}
	\label{fig:Amalgamated-direct-sum}
\end{figure}
The existence of the amalgamated direct sum follows from the existence of the amalgamated free product of $\mathrm{C}^*$-algebras and by realizing $X \oplus_V Y$ as an operator subsystem of $\mathrm{C}^*_\mathrm{max}(X) \ast_{\mathrm{C}^*_\mathrm{max}(V)} \mathrm{C}^*_\mathrm{max}(Y)$.
From the universal property of the amalgamated direct sum one obtains the fact that the ucp maps $\phi_X$ and $\phi_Y$ are complete order embeddings, and using the universal property of the amalgamated free product of $\mathrm{C}^*$-algebras, one readily checks by diagram chasing that 
\begin{align}
\mathrm{C}^*_\mathrm{max}(X \oplus_V Y) 
\cong 
\mathrm{C}^*_\mathrm{max}(X) \ast_{\mathrm{C}^*_\mathrm{max}(V)} \mathrm{C}^*_\mathrm{max}(Y),
\end{align}
\emph{cf.}\ \cite[Proposition 2.6]{KL09}, see also \cite[Theorem 4.24]{MT18}.

More concretely, if $V$ is finite dimensional, we have 
\begin{align}
X \oplus_V Y = \frac{X \oplus Y}{J}
\text{ for }
J = \{(\iota_X(v),-\iota_Y(v)) \mid v \in V\}.
\end{align}
In fact, in this case $J$ is the kernel of a ucp map onto an operator system and it is furthermore completely proximinal \cite[Proposition 2.4]{Kav14}, so that 
\begin{align}
\mathrm{Arch}\left(\setfrac{X \oplus Y}{J}\right) 
= 
\setfrac{X \oplus Y}{J}.
\end{align}
One readily checks the universal property of the amalgamated direct sum for $\setfrac{X \oplus Y}{J}$.
If we do not specify the operator subsystem $V$, we usually assume $V \cong \mathbb{C}e_X \cong \mathbb{C}e_Y$ and just speak of the \emph{amalgamated direct sum} or \emph{coproduct} of $X$ and $Y$, which we then denote by $X \oplus_\mathbf{1} Y$, \emph{cf.}\ \cite{Kav14,Fri14,Cha23}.
	
Recall that an operator system $X$ is \emph{hyperrigid} inside its minimal $\mathrm{C}^*$-cover if for every non-degenerate representation $\pi$ of $\mathrm{C}^*_\mathrm{min}(X)$ on a Hilbert space $H$, the only ucp extension of the map $\pi|_X : X \rightarrow \mathcal{B}(H)$ is $\pi$ itself \cite{Arv11}.
If the operator systems $X$ and $Y$ are hyperrigid \cite{Arv08} in their respective minimal $\mathrm{C}^*$-covers, we have 
\begin{align}
\mathrm{C}^*_\mathrm{min}(X \oplus_\mathbf{1} Y) 
\cong 
\mathrm{C}^*_\mathrm{min}(X) \ast_\mathbf{1} \mathrm{C}^*_\mathrm{min}(Y),
\end{align}
see \cite[Theorem 4.11]{Cha23}.
In particular, this is the case if $X$ and $Y$ contain \emph{enough unitaries} in their respective minimal $\mathrm{C}^*$-covers in the sense of \cite{Kav14}; 
in this case the isomorphism $\mathrm{C}^*_\mathrm{min}(X \oplus_\mathbf{1} Y) \cong \mathrm{C}^*_\mathrm{min}(X) \ast_\mathbf{1} \mathrm{C}^*_\mathrm{min}(Y)$ is shown in \cite{FKPT18}.

\subsection{Colimits}
    
The amalgamated direct sum discussed in Subsection \ref{ss_ads} may be seen as a special instance of a \emph{colimit} in the category $\mathbf{OSy}$ (\emph{cf.}\ the universal properties of amalgamated direct sums and colimits illustrated in the diagrams \autoref{fig:Amalgamated-direct-sum} and \autoref{fig:Colimit}).
We recall this general category theoretic notion.
Let $(A,\leq_A)$ be a partially ordered set and consider a set of objects $X_\alpha$, $\alpha \in A$, and morphisms $\phi_{\beta,\alpha} : X_\alpha \rightarrow X_\beta$, for all $\alpha, \beta \in A$ with $\alpha \leq_A \beta$.
We say that the pair $((X_\alpha)_{\alpha \in A}, (\phi_{\beta,\alpha})_{\alpha\leq_A\beta})$ is a \emph{diagram} if $\phi_{\gamma,\beta} \circ \phi_{\beta,\alpha} = \phi_{\gamma,\alpha}$, for all $\alpha \leq_A \beta \leq_A \gamma$.
A \emph{colimit} of the diagram $((X_\alpha)_{\alpha \in A}, (\phi_{\beta,\alpha})_{\alpha\leq_A\beta})$ in $\mathbf{C}$ is an object $X$ of $\mathbf{C}$, equipped with morphisms $\phi_{\infty,\alpha} : X_\alpha \rightarrow X$, such that $\phi_{\infty,\beta} \circ \phi_{\beta,\alpha} = \phi_{\infty,\alpha}$, for all $\alpha \leq_A \beta$, satisfying the following universal property:
For every object $Y$ of $\mathbf{C}$ and morphisms $\psi_{\infty,\alpha} : X_\alpha \rightarrow Y$, such that $\psi_{\infty,\beta} \circ \phi_{\beta,\alpha} = \psi_{\infty,\alpha}$ for all $\alpha \leq_A \beta$, there exists a morphism $\psi : X \rightarrow Y$ such that $\psi \circ \phi_{\infty,\alpha} = \psi_{\infty,\alpha}$ for all $\alpha \in A$;
in this case, we write $\underset{\longrightarrow}{\lim}_\mathbf{C} X_\alpha := X$.
Using the universal property, one shows that colimits are unique if they exist.
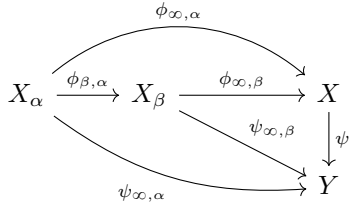
\begin{figure}[h]
	\centering
	\begin{tikzcd}
		X_\alpha & X_\beta && X \\
		&&& Y
		\arrow["\phi_{\beta,\alpha}", from=1-1, to=1-2]		
        \arrow["\phi_{\infty,\alpha}", bend left=40, from=1-1, to=1-4]
		\arrow["\phi_{\infty,\beta}", from=1-2, to=1-4]
		\arrow["\psi_{\infty,\alpha}"', bend right=20, from=1-1, to=2-4]
		\arrow["\psi_{\infty,\beta}", from=1-2, to=2-4]
		\arrow["\psi", from=1-4, to=2-4]
	\end{tikzcd}
	\caption{The universal property of the colimit.}
	\label{fig:Colimit}
\end{figure}

If the index set $A$ is the set of natural numbers or more generally a directed set, diagrams are often called \emph{inductive} or \emph{directed systems}, and their colimits (if they exist) are called \emph{inductive} or \emph{direct limits}.

In the categories $\mathbf{C^*Alg}$ of unital $\mathrm{C}^*$-algebras with unital $^*$-homomorphisms and $\mathbf{OSy}$, direct limits exist \cite[II.8.2]{Bla06}, \cite[Theorem 4.11]{MT18}.
As explained in \cite{MT18}, direct limits in $\mathbf{OSy}$ are constructed by taking the direct limit in the matrix order unit space category $\mathbf{MOU}$ and Archimedeanizing.
The direct limit in $\mathbf{MOU}$ is in turn constructed as the matrix level-wise direct limit of order unit spaces, the existence of which is also shown in \cite{MT18}.

In \autoref{sec:Psd-fct-pos-domain} we will encounter a diagram in $\mathbf{OSy}$ which is not a directed system;
however, we have a natural candidate for its colimit in $\mathbf{MOU}$, and the following lemma ensures that it admits a colimit in $\mathbf{OSy}$ as well.
First note that, due to the functoriality of the Archimedeanisation \autoref{lem:Fctr-Arch}, 
if $((X_\alpha)_{\alpha \in A}, (\phi_{\beta,\alpha})_{\alpha\leq_A\beta})$ is a diagram in $\mathbf{MOU}$ then
$((\mathrm{Arch}(X_\alpha))_{\alpha \in A}, (\mathrm{Arch}(\phi_{\beta,\alpha}))_{\alpha\leq_A\beta})$ is a diagram in $\mathbf{OSy}$.

\begin{lemma}\label{lem:Arch-colimit-commute}
Let $(A,\leq_A)$ be a partially ordered set and $((X_\alpha)_{\alpha \in A}, (\phi_{\beta,\alpha})_{\alpha\leq_A\beta})$ be a diagram in $\mathbf{MOU}$ admitting a colimit in $\mathbf{MOU}$.
Then the diagram $((\mathrm{Arch}(X_\alpha))_{\alpha \in A},$ $(\mathrm{Arch}(\phi_{\beta,\alpha}))_{\alpha\leq_A\beta})$ admits a colimit in $\mathbf{OSy}$ and 
\begin{align}
	\underset{\longrightarrow}{\lim}_\mathbf{OSy} \mathrm{Arch}(X_\alpha)
	= \mathrm{Arch}\left(\underset{\longrightarrow}{\lim}_\mathbf{MOU} X_\alpha\right).
\end{align}
\end{lemma}
	
\begin{figure}[h]
	\centering
	\begin{tikzcd}
		X_\alpha && X_\beta &&& X & \\
		\mathrm{Arch}(X_\alpha) && \mathrm{Arch}(X_\beta) &&& \mathrm{Arch}(X) & Z
		\arrow["\phi_{\beta,\alpha}", from=1-1, to=1-3]
		\arrow["\phi_{\infty,\alpha}", bend left=20, from=1-1, to=1-6]
		\arrow["\phi_{\infty,\beta}", from=1-3, to=1-6]
		\arrow["q_\alpha", from=1-1, to=2-1]
		\arrow["q_\beta", from=1-3, to=2-3]
		\arrow["\mathrm{Arch}(\phi_{\beta,\alpha})", from=2-1, to=2-3]
		\arrow["\mathrm{Arch}(\phi_{\infty,\alpha})"', bend right=20, from=2-1, to=2-6]
		\arrow["\mathrm{Arch}(\phi_{\infty,\beta})", from=2-3, to=2-6]
		\arrow["q_X", from=1-6, to=2-6]
		\arrow["\psi_{\infty,\alpha}"', bend right=40, from=2-1, to=2-7]
		\arrow["\psi_{\infty,\beta}"', bend right=20, from=2-3, to=2-7]
		\arrow["\psi_0", from=1-6, to=2-7]
		\arrow["\psi", from=2-6, to=2-7]
	\end{tikzcd}
	\caption{The diagram for the proof of \autoref{lem:Arch-colimit-commute}.}
\end{figure}
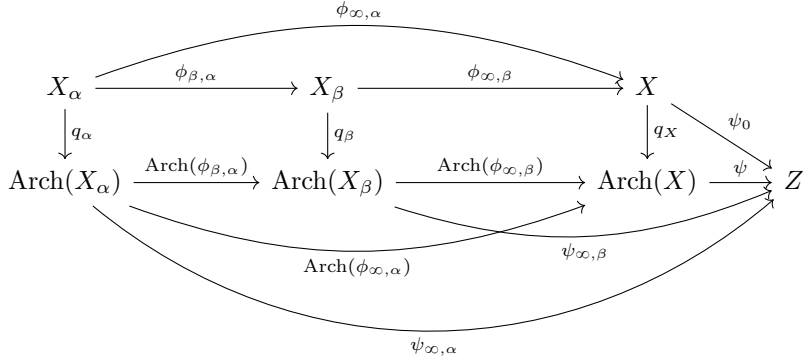

\begin{proof}
Let $X := \underset{\longrightarrow}{\lim}_\mathbf{MOU}X_\alpha$.
Assume that $Z$ is an operator system, equipped with ucp maps $\psi_{\infty,\alpha} : \mathrm{Arch}(X_\alpha) \rightarrow Z$ such that $\psi_{\infty,\beta} \circ \mathrm{Arch}(\phi_{\beta,\alpha}) = \psi_{\infty,\alpha}$, for all $\alpha\leq_A\beta$.
Since $Z$ is in particular a matrix order unit space, there exists a unique ucp map $\psi_0 : X \rightarrow Z$ such that $\psi_0 \circ \phi_{\infty,\alpha} = \psi_{\infty,\alpha} \circ q_\alpha$.
By the universal property of the Archimedeanisation, there exists a unique ucp map $\psi : \mathrm{Arch}(X) \rightarrow Z$ such that $\psi \circ q = \psi_0$.
By the functoriality of the Archimedeanisation (\autoref{lem:Fctr-Arch}), we obtain ucp maps $\mathrm{Arch}(\phi_{\infty,\alpha}) : \mathrm{Arch}(X_\alpha) \rightarrow \mathrm{Arch}(X)$ which satisfy $\psi \circ \mathrm{Arch}(\phi_{\infty,\alpha}) = \psi_{\infty,\alpha}$ for all $\alpha\in A$.
This shows that $\mathrm{Arch}(X) = \underset{\longrightarrow}{\lim}_\mathbf{OSy} \mathrm{Arch}(X_\alpha)$ as claimed.
\end{proof}

\subsection{Operator spaces}\label{subsec:OSp}

We now recall some basic operator space notions, and refer the reader to \cite{Pau02,Pis03,BlM04,ER20} for further details.
An \emph{operator space} is a complex vector space $X$ equipped with a family of matrix norms $\lVert\cdot\rVert_{n}$ on $\mathrm{M}_{n}(X)$, $n\in\mathbb{N}$, which is compatible in the sense that:
\begin{enumerate}
\item $\lVert A x B \rVert_{m} \leq \lVert A \rVert \lVert x \rVert_{n} \lVert B \rVert$, for all elements $x \in \mathrm{M}_{n}(X)$, and scalar matrices $A\in \mathrm{M}_{m,n}$
and $B\in \mathrm{M}_{n,m}$;
\item $\lVert x \oplus y \rVert_{m+n} = \max \{\lVert x \rVert_{n}, \lVert y \rVert_{m}\}$, for all $x \in \mathrm{M}_{n}(X), y \in \mathrm{M}_{m}(X)$.
\end{enumerate}
A linear map $\phi : X \rightarrow Y$ between operator spaces is called completely bounded (cb, for short) if 
\begin{align} 
\lVert \phi \rVert_\mathrm{cb} 
:= 
\sup_{n\in\mathbb{N}} \sup_{\lVert x \rVert_n \leq 1} \lVert \phi^{(n)}(x) \rVert_n < \infty;
\end{align}
in this case, $\lVert \phi \rVert_\mathrm{cb}$ is referred to as the \emph{cb-norm} of $\phi$.
We denote the set of all cb maps $\phi : X \rightarrow Y$ by $\mathcal{CB}(X,Y)$.
A cb map $\phi \in \mathcal{CB}(X,Y)$ is called completely contractive (cc, for short) if $\lVert \phi \rVert_\mathrm{cb} \leq 1$ and completely isometric if $\lVert \phi^{(n)}(x) \rVert_n = \lVert x \rVert_n$, for all $n\in\mathbb{N}$, $x \in \mathrm{M}_n(X)$.
We denote the category of operator spaces, with cc maps as morphisms, by $\mathbf{OSp}$.
	
Let $X \subseteq \mathcal{B}(H)$ be a closed subspace.
Then $X$ is an operator space by isometrically identifying the space of matrices $\mathrm{M}_n(X)$ as a subspace of $\mathcal{B}(H^n)$.
Every cb map $\phi: X \rightarrow \mathcal{B}(K)$ extends to a cb map $\tilde{\phi} : \mathcal{B}(H) \rightarrow \mathcal{B}(K)$ with $\lVert \tilde{\phi} \rVert_\mathrm{cb} = \lVert \phi \rVert_\mathrm{cb}$ by \emph{Wittstock's extension theorem} \cite{Wit84}.

By \emph{Ruan's theorem} \cite{Rua88}, every operator space is completely isometrically isomorphic to a subspace of $\mathcal{B}(H)$, for some Hilbert space $H$, implying that
every operator system is canonically (unitally completely order isomorphic) to an operator space. 
For a cp map $\phi$ we have $\phi(\mathbf{1}) = \lVert \phi \rVert_\mathrm{cb}$.
A unital linear map is cp if and only if it is cc, and it is a complete order embedding if and only if it is completely isometric.
By Wittstock's decomposition theorem \cite{Wit81}, see also \cite[Theorem 8.5]{Pau02}, every cb map on an operator system 
with values in $\cl B(H)$, for some Hilbert space $H$,
is a linear combination of four cp maps.
Given an operator space $X$, its dual $X^{\mathrm{d}}$ is an operator space when equipped with the matrix norms inherited by declaring the identification $\mathrm{M}_n(X^{\mathrm{d}}) \cong \mathcal{CB}(X,\mathrm{M}_n)$ completely isometric \cite[1.2.20]{BlM04}.
Moreover, the matrix ordered vector space and operator space structures on $X^{\mathrm{d}}$ are compatible in the sense that the proper convex cones $\mathcal{CP}(X,\mathrm{M}_n)$ span $\mathcal{CB}(X,\mathrm{M}_n)$ and are closed (for the cb-norm);
this shows that $X^{\mathrm{d}}$ is a \emph{(positively generated) matrix ordered operator space}; see \cite{HKM23} for further details. 
		
Recall from \cite[1.2.14]{BlM04} that if $X$ is an operator space and $Y \subseteq X$ a closed subspace the quotient $\setfrac{X}{Y}$ is an operator space, for the matrix-norms given by 
\begin{align}
\lVert (x_{ij}+Y)_{i,j} \rVert_n 
:= 
\inf\{\lVert x+y \rVert_n \mid y \in \mathrm{M}_n(Y)\}.
\end{align}
For operator spaces $X$ and $Z$, a cc map $\phi : X \rightarrow Z$ is called an $\mathbf{OSp}$-quotient map if it is onto and the induced map $\tilde{\phi} : \setfrac{X}{\ker(\phi)} \rightarrow Z$ is a complete isometry (i.e.\ the inverse map $\tilde{\phi}^{-1} : Z \rightarrow \setfrac{X}{\ker(\phi)}$ is cc).
	
A map $\phi : X \rightarrow Y$ between operator spaces is cb if and only if $\phi \otimes \mathbf{I}^Z : X \otimes Z \rightarrow Y \otimes Z$ is cb, for every operator space $Z$ \cite{DCH85}, see also \cite[Proposition 2.1.1]{Pis03}.
It follows that a map $\phi : X \rightarrow Y$ between operator systems is cp if and only if $\phi \otimes \mathbf{I}^Z : X \otimes Z \rightarrow Y \otimes Z$ is cp, for every operator system $Z$.

\subsection{Complete order embeddings and quotient maps}
		
The following proposition is central to \autoref{thm:Characterizations-CPEP} and \autoref{th_main} below.
We give a full proof of the equivalence of (1) and (3) for which we could not find a reference in the infinite dimensional case.
Note that this is in line with \cite[Proposition 1.8]{FP12} which implies the equivalence of (1) and (3) in the case that the operator systems $X$ and $Y$ are finite-dimensional.

\begin{proposition}\label{prop:Equiv-embeddings-quotients}
Let $X$ and $Y$ be operator systems and let $\phi : X \rightarrow Y$ be a ucp map.
The following statements are equivalent:
\begin{enumerate}
\item[(i)] $\phi : X \rightarrow Y$ is a complete order embedding;
\item[(ii)] $\phi : X \rightarrow Y$ is a complete isometry;
\item[(iii)] $\phi^{\mathrm{d}} : Y^{\mathrm{d}} \rightarrow X^{\mathrm{d}}$ is an $\mathbf{MVS}$-quotient map;
\item[(iv)] $\phi^{\mathrm{d}} : Y^{\mathrm{d}} \rightarrow X^{\mathrm{d}}$ is an $\mathbf{OSp}$-quotient map.
\end{enumerate}
\end{proposition}
	
\begin{proof}
The equivalence [(i)$\Leftrightarrow$(ii)] is well-known, see e.g.\ \cite[1.3.3]{BlM04}.
For the equivalence [(ii)$\Leftrightarrow$(iv)] see \cite[1.4.3]{BlM04};
we note that the implication [(ii)$\Rightarrow$(iv)] requires (the matrix-valued version of) Wittstock's extension theorem to extend cb maps $X \rightarrow \mathrm{M}_n$ to cb maps $Y \rightarrow \mathrm{M}_n$ (if $X$ is completely isometrically identified with its image $\phi(X) \subseteq Y$).
The converse implication uses the completely isometric inclusion of an operator space in its bidual.

\smallskip

\noindent
[(i)$\Rightarrow$(iii)].
Assume that $\phi : X \rightarrow Y$ is a complete order embedding.
In particular, $\phi$ is one-to-one, so $\phi^{\mathrm{d}}$ is cp and surjective.
Denote by 
\begin{align}
q : Y^{\mathrm{d}} \rightarrow \setfrac{Y^{\mathrm{d}}}{\ker(\phi^{\mathrm{d}})}
\end{align}
the canonical cp quotient map and by 
\begin{align}
\widetilde{\phi^{\mathrm{d}}} : \setfrac{Y^{\mathrm{d}}}{\ker(\phi^{\mathrm{d}})} \rightarrow X^{\mathrm{d}}
\end{align}
the induced cp map with the property  $\widetilde{\phi^{\mathrm{d}}} \circ q = \phi^{\mathrm{d}}$.
In order to prove statement (iii), we need to show that the inverse $(\widetilde{\phi^{\mathrm{d}}})^{-1}$ is cp.
To this end, let $x^{\mathrm{d}} \in \mathrm{M}_n(X^{\mathrm{d}})^+$.
We view $x^{\mathrm{d}}$ as a cp map (denoted in the same way) $x^{\mathrm{d}} : X \rightarrow \mathrm{M}_n$.
By assumption, $\phi^{-1} : \phi(X) \rightarrow X$ is a well-defined ucp map; hence 
\begin{align}
x^{\mathrm{d}} \circ \phi^{-1} \in \mathcal{CP}(\phi(X),\mathrm{M}_n).
\end{align}
By Arveson's extension theorem there exists a cp map $y^{\mathrm{d}} \in \mathcal{CP}(Y,\mathrm{M}_n)$ that extends $x^{\mathrm{d}} \circ \phi^{-1}$, i.e.\
\begin{align}\label{eqn:Pullback-y*-x*}
y^{\mathrm{d}}(\phi(x)) =  x^{\mathrm{d}} \circ \phi^{-1}(\phi(x)) = x^{\mathrm{d}}(x), \ \ \ x \in X.
\end{align} 
Viewing $x^{\mathrm{d}}$ and $y^{\mathrm{d}}$ as elements of $\mathrm{M}_n(X^{\mathrm{d}})^+$ and $\mathrm{M}_n(Y^{\mathrm{d}})^+$, respectively, equation (\ref{eqn:Pullback-y*-x*}) reads $(\phi^{\mathrm{d}})^{(n)}(y^{\mathrm{d}}) = x^{\mathrm{d}}$.
It follows that
\begin{align}
((\widetilde{\phi^{\mathrm{d}}})^{-1})^{(n)} (x^{\mathrm{d}})
&= ((\widetilde{\phi^{\mathrm{d}}})^{-1})^{(n)} \left((\phi^{\mathrm{d}})^{(n)}(y^{\mathrm{d}})\right) \\
&= ((\widetilde{\phi^{\mathrm{d}}})^{-1})^{(n)} \circ (\widetilde{\phi^{\mathrm{d}}})^{(n)} \circ q^{(n)}(y^{\mathrm{d}}) \\
&= q^{(n)}(y^{\mathrm{d}}) 
\in \mathrm{M}_n(\setfrac{Y^{\mathrm{d}}}{\ker(\phi^{\mathrm{d}})})^+.
\end{align} 
This shows the map $(\widetilde{\phi^{\mathrm{d}}})^{-1}$ is cp, as claimed.

\smallskip

\noindent
[(iii)$\Rightarrow$(i)].
Assume that $\phi^{\mathrm{d}} : Y^{\mathrm{d}} \rightarrow X^{\mathrm{d}}$ is an $\mathbf{MVS}$-quotient map.
In particular, $\phi^{\mathrm{d}}$ is surjective, so $\phi$ is one-to-one with an inverse map $\phi^{-1} : \phi(X) \rightarrow X$.
Identifying the dual $\phi(X)^{\mathrm{d}}$ of the image $\phi(X) \subseteq Y$ with the quotient $\setfrac{Y^{\mathrm{d}}}{\phi(X)^\perp} \cong \setfrac{Y^{\mathrm{d}}}{\ker(\phi^{\mathrm{d}})}$ we may view the dual map $(\phi^{-1})^{\mathrm{d}} : X^{\mathrm{d}} \rightarrow \phi(X)^{\mathrm{d}}$ of the inverse of $\phi$ as the inverse of the dual map of $\phi$ (modulo $\phi(X)^\perp$), i.e.\ 
\begin{align}
(\phi^{-1})^{\mathrm{d}}
= 
(\widetilde{\phi^{\mathrm{d}}})^{-1} : X^{\mathrm{d}} \rightarrow \setfrac{Y^{\mathrm{d}}}{\ker(\phi^{\mathrm{d}})}.
\end{align}
Note that the map $(\widetilde{\phi^{\mathrm{d}}})^{-1}$ is cp by the assumption that $\phi^*$ is an $\mathbf{MVS}$-quotient map.
Hence the dual map 
\begin{align}
((\widetilde{\phi^{\mathrm{d}}})^{-1})^{\mathrm{d}} : (\setfrac{Y^{\mathrm{d}}}{\ker(\phi^{\mathrm{d}})})^{\mathrm{d}} \cong \phi(X)^{\mathrm{dd}} \rightarrow X^{\mathrm{dd}}
\end{align}
is cp.
Recall from \cite[Proposition 6.2]{KPTT13} that an operator system is canonically unitally completely order embedded into its bidual.
We thus consider the restriction of $((\widetilde{\phi^{\mathrm{d}}})^{-1})^{\mathrm{d}}$ to the operator subsystem $\phi(X) \subseteq \phi(X)^{\mathrm{dd}}$; 
we claim that 
\begin{align}
((\widetilde{\phi^{\mathrm{d}}})^{-1})^{\mathrm{d}}|_{\phi(X)} = \phi^{-1}.
\end{align}
Indeed, for $x \in X$, viewing $\phi(x)$ as an element of $\phi(X)^{\mathrm{dd}}$, we have
\begin{align}
((\widetilde{\phi^{\mathrm{d}}})^{-1})^{\mathrm{d}}((\phi(x)))
= 
\phi(x) \circ (\widetilde{\phi^{\mathrm{d}}})^{-1}
= 
\phi(x) \circ (\phi^{-1})^{\mathrm{d}}
= 
\phi^{-1} \circ \phi(x)
= x.
\end{align}
It follows that the map
\begin{align}
\phi^{-1} = ((\widetilde{\phi^{\mathrm{d}}})^{-1})^{\mathrm{d}}|_{\phi(X)} : \phi(X) \rightarrow X
\end{align}
is cp, and hence $\phi$ is a complete order embedding as claimed.
\end{proof}

\begin{remark}
We point out that the statements in \autoref{prop:Equiv-embeddings-quotients} are also equivalent to $\phi^{\mathrm{d}}$ being a quotient map in the category $\mathbf{MOS}$ of matrix ordered operator spaces with cc cp maps as morphisms \cite{HKM23}.
In fact, this is implied by (3) and (4) together, and the fact that a $\mathbf{MOS}$-quotient map is in particular an $\mathbf{OSp}$-quotient map.
\end{remark}

\section{Positive semi-definite functions and the Fourier operator system}\label{ss_opval}
	
Let $\Gamma$ be a discrete group and $H$ a Hilbert space.
If $\Sigma \subseteq \Gamma$ is a finite subset and $X \subseteq \mathcal{B}(H)$ is an operator space, we denote by $\mathrm{M}_\Sigma(X)$ the set of matrices over $X$ indexed by $\Sigma \times \Sigma$, identified with the space $\B(\ell^2(\Sigma)) \otimes_{\min} X \subseteq \mathcal{B}(\ell^2(\Sigma)\otimes H)$.
We set $\mathrm{M}_\Sigma := \mathrm{M}_\Sigma(\mathbb{C})$.
We denote by $\delta_s$, $s \in \Gamma$, the unitary generators of the universal group C$^*$-algebra $\mathrm{C}^*(\Gamma)$.
	
The set ${\rm F}(\Gamma,\cl B(H))$ of all $\cl B(H)$-valued functions $u : \Gamma \to \cl B(H)$ 
is a $^*$-vector space under the operation $u^*(s) := u(s^{-1})^*$, $s\in \Gamma$. 
We recall that a function $u : \Gamma \to \cl B(H)$ is called \emph{positive semi-definite} if 
\begin{align}
(u(st^{-1}))_{s,t\in\Sigma} \in \mathrm{M}_\Sigma(\cl B(H))^+,
\text{ for all finite subsets } 
\Sigma \subseteq \Gamma.
\end{align}
Observe that for the finite set $\{e, s\}$ with $s \in \Gamma$, positive semi-definiteness implies that
\begin{align}
u(s)^* = u(s^{-1})
\ \ \mbox{ and } \ \
u(s)^* u(s) \leq \lVert u(e) \rVert u(e);
\end{align}
therefore $u(e)$ is selfadjoint and $H' := \overline{u(e)H}$ is reducing for $u$ with
\begin{align}
u(s) = \left[\begin{matrix} u(s)|_{H'} & 0 \\ 0 & 0 \end{matrix}\right], \text{ for all } s \in \Gamma.
\end{align}
We denote by $\mathrm{B}(\Gamma,H)^+$ the cone of all positive semi-definite $\cl B(H)$-valued functions, and by $\mathrm{B}(\Gamma,H)$ the linear span of $\mathrm{B}(\Gamma,H)^+$ inside ${\rm F}(\Gamma,\cl B(H))$.
We call $\mathrm{B}(\Gamma,H)$ the \emph{$\mathcal{B}(H)$-valued Fourier--Stieltjes algebra} of $\Gamma$.
Clearly, if $H=\mathbb{C}$ we obtain the usual \emph{Fourier--Stieltjes algebra}, denoted by $\mathrm{B}(\Gamma)$.
	
The following proposition is part of folklore; we include its proof for the convenience of the reader. 

\begin{proposition}\label{prop:cpmaps}
Let $\Gamma$ be a discrete group. 
\begin{itemize}
\item[(i)] If $u\in \mathrm{B}(\Gamma,H)^+$, then there exists a (necessarily) unique completely positive map $\phi_u : \mathrm{C}^*(\Gamma)\to \cl B(H)$ such that $\phi_u(\delta_s) = u(s)$, for all $s \in \Gamma$. 
\item[(ii)] If $\phi : \mathrm{C}^*(\Gamma)\to \cl B(H)$ is a completely positive map, then there exists a (necessarily) unique positive semi-definite function $u_\phi\in \mathrm{B}(\Gamma,H)^+$ such that $\phi = \phi_{u_\phi}$. 
\end{itemize}
\end{proposition}

\begin{proof}
(i) 
Since $\mathrm{C}^*(\Gamma)$ is densely spanned by $\{\delta_s\}_{s \in \Gamma}$, any completely positive extension $\phi_u$ of $u$ is necessarily unique.
Hence we need to show only existence of the completely positive extension map $\phi_u$.
We proceed in cases.

\smallskip
\noindent
{\em Case 1}. Suppose first that $u(e) = \mathbf{1}_{\mathcal{B}(H)}$.
By \cite[Theorem 7.1]{SFBK10}, there exists a Hilbert space $K \supseteq H$ and a *-representation $\pi : \mathrm{C}^*(\Gamma)\to \cl B(K)$ such that
\begin{align}
u(s) = P_H \pi(\delta_s)|_H, \text{ for all } s \in \Gamma.
\end{align}
Then the completely positive map $\phi_u := P_H \pi(\cdot) |_H$ is the required extension.

\smallskip
\noindent
{\em Case 2.} Suppose that $u(e)$ is an invertible element of $\mathcal{B}(H)$.
Let $u' \in {\rm F}(\Gamma,\cl B(H))$ be the $\mathcal{B}(H)$-valued function given by
\begin{align}
u'(s) := u(e)^{-1/2} u(s) u(e)^{-1/2}, \text{ for all } s \in \Gamma.
\end{align}
Then for all finite subsets $\Sigma \subseteq \Gamma$, setting $n := \lvert \Sigma \rvert$, we have
\begin{align}
(u'(st^{-1}))_{s,t \in \Sigma}
=
(u(e)^{-1/2} \otimes \mathbf{1}_n) (u(st^{-1}))_{s,t\in\Sigma} (u(e)^{-1/2} \otimes \mathbf{1}_n).
\end{align}
Since $u(e)$ is a positive operator we get that $u' \in \mathrm{B}(\Gamma,H)^+$ with $u'(e) = \mathbf{1}_{\mathcal{B}(H)}$.
By Case 1 we obtain a unital completely positive map $\phi_{u'}$ such that 
\begin{align}
\phi_{u'}(\delta_s) = u'(s), \text{ for all } s \in \Gamma.
\end{align}
Therefore the map 
\begin{align}
\phi_u(\cdot) := u(e)^{1/2} \phi_{u'}(\cdot) u(e)^{1/2}
\end{align}
is completely positive and satisfies
\begin{align}
\phi_{u}(\delta_s) = u(e)^{1/2} u'(s) u(e)^{1/2} = u(s),
\end{align}
for all $s \in \Gamma$.

\smallskip
\noindent
{\em Case 3.} Suppose that $H$ is finite dimensional and that $u(e) \neq \mathbf{1}_{\mathcal{B}(H)}$.
Set $H' := \overline{u(e) H} = u(e) H$ and observe that $u(e)|_{H'}$ is positive and invertible.
Moreover $H'$ is reducing for every $u(s)$ with
\begin{align}
u(s) = \begin{bmatrix} u(s)|_{H'} & 0 \\ 0 & 0 \end{bmatrix}, \text{ for all } s \in \Gamma.
\end{align}
Therefore the $\mathcal{B}(H')$-valued function $u'$ given by
\begin{align}
u'(s) := u(s) |_{H'}, \text{ for all } s \in \Gamma,
\end{align}
is positive semi-definite with $u'(e)$ invertible in $\mathcal{B}(H')$.
By Case 2, we obtain a completely positive map $\phi_{u'} : \mathrm{C}^*(\Gamma)\to \mathcal{B}(H')$ such that $\phi_{u'}(\delta_s) = u'(s)$ for all $s \in \Gamma$.
Set 
\begin{align}
\phi_u := \phi_{u'} \oplus 0_{H \ominus H'},
\end{align}
and we get
\begin{align}
\phi_u(\delta_s) = u'(s) \oplus 0_{H \ominus H'} = u(s),
\end{align}
for all $s \in \Gamma$.

\smallskip
\noindent
{\em Case 4.} Suppose that $H$ is an arbitrary Hilbert space and that $u(e) \neq \mathbf{1}_{\mathcal{B}(H)}$.
Let
\begin{align}
\mathcal{F}(H) := \{F \leq H \mid \dim F < \infty\}.
\end{align}
For every $F \in \mathcal{F}(H)$ define the $\mathcal{B}(F)$-valued function
\begin{align}
u_F : \Gamma \to \mathcal{B}(F); \quad u_F(s) = P_F u(s)|_F.
\end{align}
By Case 3, there exists a completely positive map
\begin{align}
\phi_{u_{F}} : \mathrm{C}^*(\Gamma) \to \mathcal{B}(F); \quad \phi_{u_F}(\delta_s) = u_F(s).
\end{align}
Since $\mathrm{C}^*(\Gamma)$ is an operator system we have
\begin{align}
\|\phi_{u_F}\|_{\rm cb} = \|\phi_{u_F}(\delta_e)\| = \|u_F(e)\| \leq \|u(e)\|,
\end{align}
for all $F \in \mathcal {F}(H)$.
We make the identification of $F$ as the subspace of $H$ so that every $\phi_{u_F}$ is $\mathcal{B}(H)$-valued.
The set
\begin{align}
\mathfrak{C} := \{\phi : \mathrm{C}^*(\Gamma) \to \mathcal{B}(H) \mid \text{ $\phi$ is completely positive and } \|\phi\|_{\rm cb} \leq \|u(e)\| \}
\end{align}
is compact with respect to Arveson's BW topology (see \cite{Arv69}), that is, the topology of point-weak* convergence; see \cite[Theorem 7.4]{Pau02}.
Therefore there is a subnet which converges to some $\phi \in \mathfrak{C}$.
Then we can follow the proof of \cite[Theorem 7.5]{Pau02} to conclude that $\phi$ is the requried extension.
That is, for $s \in \Gamma$ and $x,y \in H$ we let $F$ be the finite dimensional space spanned by $x, y \in H$.
Then for every $F' \supseteq F$ we get
\begin{align}
u_F(s) = P_F u(s)|_F = P_{F} u_{F'}(s) |_F,
\end{align}
and so
\begin{align}
\langle u(s)x,y \rangle 
=
\langle u_F(s)x,y \rangle 
=
\langle u_{F'}(s)x,y \rangle 
= 
\langle \phi_{u_{F'}}(\delta_s)x,y \rangle.
\end{align}
Since the set $\{F' \in \mathcal{F} : F' \supseteq F\}$ is cofinal, we get
\begin{align}
\langle u(s)x,y \rangle = \langle \phi(\delta_s)x,y \rangle,
\end{align}
and thus $\phi_u$ is the required extension of $u$.

\smallskip
\noindent        
(ii) 
Let $u_{\phi} : \Gamma\to \cl B(H)$ be given by $u_{\phi}(s) := \phi(\delta_s)$. 
Given $s_i\in \Gamma$, $i = 1,\dots,k$, we have that $(\delta_{s_i s_j^{-1}})_{i,j}\in \mathrm{M}_k(\mathrm{C}^*(\Gamma))^+$; 
thus, $(u_{\phi}(s_i s_j^{-1}))_{i,j}\in \mathrm{M}_k(\cl B(H))^+$, that is, $u_{\phi}$ is positive semi-definite. 
Since $\mathrm{C}^*(\Gamma)$ is densely spanned by $\{\delta_s\}_{s \in \Gamma}$ and the extension $\phi_{u_\phi}$ of $u_\phi$ is unique by (i), we get $\phi_{u_{\phi}} = \phi$. 
\end{proof}
	
We note that 
the uniqueness clauses in Proposition \ref{prop:cpmaps} 
imply that the assignments 
\begin{align}
u \mapsto \phi_u
\ \ \mbox{ and } \ \
\phi \mapsto u_\phi
\end{align}
are inverse to each other.
\autoref{prop:cpmaps} gives a canonical identification of the cones
\begin{align}\label{eqn:Pos-def-fct-cp-maps-Gamma}
\mathrm{B}(\Gamma,H)^+ \cong \mathcal{CP}(\mathrm{C}^*(\Gamma),\mathcal{B}(H))
\end{align} 
implying that the family of cones $(\mathrm{B}(\Gamma,\mathbb{C}^n)^+)_{n\in\mathbb{N}}$ is compatible;
in particular $(\mathrm{B}(\Gamma),(\mathrm{B}(\Gamma,\mathbb{C}^n)^+)_{n\in\mathbb{N}})$ becomes a matrix ordered vector space through this identification.

We define a norm on $\mathrm{B}(\Gamma,H)$ by 
\begin{align}
\lVert u \rVert_{\mathrm{B}(\Gamma,H)} := \lVert \phi_u \rVert_{\rm cb}.
\end{align}
In the case that $H = \mathbb{C}^n$, this can be viewed as a norm on $\mathrm{M}_n(\mathrm{B}(\Gamma))$ and in fact the Fourier--Stieltjes algebra with the matrix norms $\lVert\cdot\rVert_{\mathrm{M}_n(\mathrm{B}(\Gamma))} := \lVert\cdot\rVert_{\mathrm{B}(\Gamma,\mathbb{C}^n)}$ is the operator space dual of $\mathrm{C}^*(\Gamma)$, see \cite[Subsection 1.2.20]{BlM04}.
Together with the above established matrix order structure this turns $\mathrm{B}(\Gamma)$ into a dual matrix ordered operator space, as studied e.g.\ in \cite{HKM23}.
	
\begin{definition}\cite{llt}
A non-empty subset $\Delta \subseteq \Gamma$ is called a \emph{positivity domain} if it is symmetric, i.e.\ $\Delta^{-1} := \{s^{-1} \mid s\in\Delta \}= \Delta$, and unital, i.e.\ $e \in \Delta$.
\end{definition}

For a subset $\Delta \subseteq \Gamma$, we write 
\begin{align}
\mathrm{S}[\Delta] := \mathrm{span}\{\delta_s \mid s\in\Delta\}.
\end{align}
If $\Delta \subseteq \Gamma$ is a positivity domain, the vector space $\mathrm{S}[\Delta]$  will be equipped with the involution, given by $\delta_s^* := \delta_{s^{-1}}$, $s\in \Delta$.

\begin{example}
If $\Sigma \subseteq \Gamma$ is any subset, then the subset $\Sigma\Sigma^{-1} := \{st^{-1}\mid s,t\in\Sigma\}$ of $\Gamma$ is a positivity domain.
We note that not all positivity domains 
have the latter form; see \autoref{fig:Positivity-domain-5pts}. 
\end{example}

\begin{definition}
Let $\Delta \subseteq \Gamma$ be a positivity domain.
The \emph{Fourier system} $\mathrm{C}^*(\Gamma)_{(\Delta)}$ associated with $\Delta$ has underlying $^*$-vector space $\mathrm{S}[\Delta]$ and operator system structure, arising from the inclusion 
$\mathrm{S}[\Delta]\subseteq \mathrm{C}^*(\Gamma)$.
\end{definition}

By definition, the positive cones of the Fourier system $\mathrm{C}^*(\Gamma)_{(\Delta)}$ are given by 
\begin{align}
\mathrm{M}_n(\mathrm{C}^*(\Gamma)_{(\Delta)})^+ = \mathrm{M}_n(\mathrm{C}^*(\Gamma))^+ \cap \mathrm{M}_n(\mathrm{S}[\Delta]),
\end{align}
and the matrix order unit $e_{\mathrm{C}^*(\Gamma)_{(\Delta)}}$ is the unit $\delta_e$ of $\mathrm{C}^*(\Gamma)$.
For a positivity domain $\Delta$, we set 
\begin{align}
\mathrm{B}(\Gamma,H)^+|_\Delta := \{u|_{\Delta} \mid u\in \mathrm{B}(\Gamma,H)^+\};
\end{align}
it is clear that $\mathrm{B}(\Gamma,H)^+|_\Delta$ is a cone.
Similarly, we write 
\begin{align}
\mathrm{B}(\Gamma,H)|_\Delta = \{u|_{\Delta} \mid u\in \mathrm{B}(\Gamma,H)\};
\end{align}
we note that $\mathrm{B}(\Gamma,H)|_\Delta = \mathrm{span}(\mathrm{B}(\Gamma,H)^+|_\Delta)$ in virtue of the identification (\ref{eqn:Pos-def-fct-cp-maps-Gamma}), and we
set 
\begin{align}
\mathrm{B}(\Gamma)|_\Delta := \mathrm{B}(\Gamma,\mathbb{C})|_\Delta.
\end{align}

We provide a generalisation of \autoref{prop:cpmaps} to positivity domains that will be needed in the sequel. 

\begin{proposition}\label{prop:Duality-Fourier-system-restrictions-psd-functions}
Let $\Gamma$ be a discrete group, $\Delta \subseteq \Gamma$ be a positivity domain and $H$ be a Hilbert space. 
\begin{itemize}
\item[(i)]  If $u\in \mathrm{B}(\Gamma,H)^+|_\Delta$ then there exists a (necessarily) unique completely positive map $\phi_u : \mathrm{C}^*(\Gamma)_{(\Delta)} \to \cl B(H)$ such that $\phi_u(\delta_s) = u(s)$, for all $s\in \Delta$. 
			
\item[(ii)]	If  $\phi : \mathrm{C}^*(\Gamma)_{(\Delta)} \to \cl B(H)$ is a  completely positive map then there exists a (necessarily) unique function $u_{\phi}\in \mathrm{B}(\Gamma,H)^+|_\Delta$ such that $\phi = \phi_{u_{\phi}}$. 
\end{itemize}
\end{proposition}
	
\begin{proof}
(i) 
Let $u \in \mathrm{B}(\Gamma,H)^+|_\Delta$ and let $\tilde{u} \in \mathrm{B}(\Gamma,H)^+$ be such that $\tilde{u}|_\Delta = u$.
By \autoref{prop:cpmaps} (i), there exists a unique completely positive map $\phi_{\tilde{u}} : \mathrm{C}^*(\Gamma) \to \mathcal{B}(H)$ such that $\phi_{\tilde{u}}(\delta_s) = \tilde{u}(s)$, for all $s \in \Gamma$.
Restricting $\phi_{\tilde{u}}$ to the Fourier system $\mathrm{C}^*(\Gamma)_{(\Delta)}$ yields the desired completely positive map $\phi_u$.
If $\tilde{u}_1, \tilde{u}_2 \in \mathrm{B}(\Gamma,H)^+$ are extensions of $u$, then $\phi_{\tilde{u}_1}(\delta_s) = \phi_{\tilde{u}_2}(\delta_s)$, for all $s \in \Delta$, which implies uniqueness of $\phi_u$.

\smallskip
\noindent
(ii) Let $\phi : \mathrm{C}^*(\Gamma)_{(\Delta)} \rightarrow \mathcal{B}(H)$ be a completely positive map.
By Arveson's extension theorem, we may extend $\phi$ to a completely positive map $\tilde{\phi} : \mathrm{C}^*(\Gamma) \to \mathcal{B}(H)$.
Using \autoref{prop:cpmaps} (ii) we obtain a positive semi-definite function $u_{\tilde{\phi}} : \Gamma \rightarrow \mathcal{B}(H)$, such that $u_{\tilde{\phi}}(s) = \tilde{\phi}(\delta_s)$, for all $s \in \Gamma$.
Restricting $u_{\tilde{\phi}}$ to $\Delta$ gives the desired function $u_\phi : \Delta \rightarrow \mathcal{B}(H)$.
If $\tilde{\phi}_1, \tilde{\phi}_2 \in \mathcal{CP}(\mathrm{C}^*(\Gamma),\mathcal{B}(H))$ are extensions of $\phi$, then $u_{\tilde{\phi}_1}(s) = u_{\tilde{\phi}_2}(s)$, for all $s \in \Delta$, which implies uniqueness of $u_\phi$.
\end{proof}
	
Similarly as in (\ref{eqn:Pos-def-fct-cp-maps-Gamma}), \autoref{prop:Duality-Fourier-system-restrictions-psd-functions} gives a canonical identification of the cones
\begin{align}\label{eqn:Pos-def-fct-cp-maps-Restr}
\mathrm{B}(\Gamma,H)^+|_\Delta \cong \mathcal{CP}(\mathrm{C}^*(\Gamma)_{(\Delta)},\mathcal{B}(H)),
\end{align}
implying the following corollary.

\begin{corollary}\label{cor:B-Gamma-Restr-MOVS}
The pair $(\mathrm{B}(\Gamma)|_\Delta, (\mathrm{B}(\Gamma,\mathbb{C}^n)^+|_\Delta)_{n\in\mathbb{N}})$ is a matrix ordered vector space.
\end{corollary}

We note that, further, the identification 
\begin{align}
\mathrm{B}(\Gamma)|_\Delta 
\cong \mathrm{span}(\mathcal{CP}(\mathrm{C}^*(\Gamma)_{(\Delta)},\mathcal{B}(H))) 
= \mathcal{CB}(\mathrm{C}^*(\Gamma)_{(\Delta)},\mathcal{B}(H))
\end{align}  
turns $\mathrm{B}(\Gamma)|_\Delta$ into a (matrix ordered) operator space.
	
\section{Sums of squares and Toeplitz matrices}\label{sec:SOS-Toep}
    
In this section we examine Toeplitz matrices associated to finite subsets $\Sigma \subseteq \Gamma$ and provide an explicit identification of the dual of the corresponding Toeplitz operator system, leading to an operator system characterisation of a positive extension property.
Throughout, $\Gamma$ will be a discrete group and $\Sigma \subseteq \Gamma$ a finite subset.
We begin by defining an operator system spanned by the unitaries $\delta_{st^{-1}}$ in $\mathrm{C}^*(\Gamma)$, for $s,t\in\Sigma$, but whose positive matrix cones consist of those elements which admit a decomposition into sums of squares, rather than the (possibly different) positive matrix cones inherited from $\mathrm{C}^*(\Gamma)$.
The resulting operator system will be denoted by $\mathrm{SOS}(\Sigma)$ and one may think of it as the operator system structure on the *-vector space generated by $\{\delta_{st^{-1}} \mid s,t\in\Sigma\}$, for which a Fej\'{e}r--Riesz type property is enforced.

\subsection{The sums of squares operator system and the Fej\'{e}r--Riesz property}
	
The following sets will constitute the positive matrix cones for our operator system $\mathrm{SOS}(\Sigma)$.
Recall our notation
\begin{align}
\mathrm{S}[\Sigma] := \mathrm{span}\{\delta_{s} \mid s\in\Sigma\}
\ \ \mbox{ and } \ \
\mathrm{S}[\Sigma\Sigma^{-1}] := \mathrm{span}\{\delta_{st^{-1}} \mid s,t\in\Sigma\},
\end{align}
and the fact that $\mathrm{S}[\Sigma\Sigma^{-1}]$ is a *-vector subspace of $\mathrm{C}^*(\Gamma)$.

\begin{definition}
Let $\Gamma$ be a discrete group and $\Sigma\subseteq \Gamma$ be a finite subset.
For $n \in \mathbb{N}$, set
\begin{align}
\mathcal{Q}_n(\Sigma) 
&:= 
\left\{ \sum_{i=1}^r y_iy_i^* \mid r \in \mathbb{N}, y_i = \sum_{s \in \Sigma} A_s^i \otimes \delta_s, A_s^i \in \mathrm{M}_{n} \right\}.
\end{align}
By definition we have $\mathcal{Q}_n(\Sigma) \subseteq \mathrm{M}_n(\mathrm{S}[\Sigma\Sigma^{-1}])_\mathrm{h}$.
\end{definition}

\begin{lemma}\label{lem:SOS-cones-positive-matrix-version}
Let $\Gamma$ be a discrete group and $\Sigma\subseteq \Gamma$ be a finite subset.
The following identities hold:
\begin{align}
\mathcal{Q}_n(\Sigma)
&= \left\{ yy^* \mid y = \sum_{s \in \Sigma} A_s \otimes \delta_s, A_s \in \mathrm{M}_{n,k}, k \in \mathbb{N} \right\} \label{eqn:SOS-cone-arbitrary-rank} \\
&= \left\{ \sum_{s,t \in \Sigma} A_{s,t} \otimes \delta_{st^{-1}} \mid (A_{s,t})_{s,t \in \Sigma} \in (\mathrm{M}_\Sigma(\mathrm{M}_n))^+ \right\} \label{eqn:SOS-cone-positive-matrix}.
\end{align}
\end{lemma}
	
 \begin{proof}
Let $y_i = \sum_{s \in \Sigma} A_s^i \otimes \delta_s \in \mathrm{M}_n \otimes \mathrm{S}[\Sigma]$, $i = 1,\dots,r$;
then $\sum_{i=1}^r y_iy_i^* = yy^*$, where
\begin{align}
y = (y_1,\dots,y_r) \in \mathrm{M}_{1,r}(\mathrm{M}_n \otimes \mathrm{S}[\Sigma]) \cong \mathrm{M}_{n,nr} \otimes \mathrm{S}[\Sigma],
\end{align}
showing that $\cl{Q}_n(\Sigma)$ is contained in the right hand side of (\ref{eqn:SOS-cone-arbitrary-rank}). 
		
To show the inclusion of the right-hand side of (\ref{eqn:SOS-cone-arbitrary-rank}) in the set displayed in (\ref{eqn:SOS-cone-positive-matrix}), let $yy^* \in \mathcal{Q}_n(\Sigma)$ with $y = \sum_{s \in \Sigma} A_s \otimes \delta_s$, for $A_s \in \mathrm{M}_{n,k}$ with $k \in \mathbb{N}$, as in (\ref{eqn:SOS-cone-arbitrary-rank}).
Then 
\begin{align} 
yy^* = \sum_{s,t \in \Sigma} A_sA_t^* \otimes \delta_{st^{-1}}
\ \ \mbox{ and } \ \
(A_sA_t^*)_{s,t \in \Sigma} \in (\mathrm{M}_\Sigma(\mathrm{M}_n))^+,
\end{align}
as required.
		
To show that the set displayed in (\ref{eqn:SOS-cone-positive-matrix}) is contained in $\mathcal{Q}_n(\Sigma)$, consider an element of the form $x = \sum_{s,t\in\Sigma} A_{s,t} \otimes \delta_{st^{-1}}$ with $A := (A_{s,t})_{s,t\in\Sigma} \in (\mathrm{M}_\Sigma(\mathrm{M}_n))^+$. 
Then there is a matrix $B = (B_{s,t})_{s,t\in\Sigma} \in \mathrm{M}_\Sigma(\mathrm{M}_n)$ such that 
\begin{align} 
A = BB^* = \left(\sum_{r\in\Sigma} B_{s,r} (B^*)_{r,t}\right)_{s,t\in\Sigma}.
\end{align}
Note that $(B^*)_{r,t} = B_{t,r}^* \in \mathrm{M}_n$, for all $r,t \in \Sigma$.
It follows that
\begin{align}\label{eqn:Decomposition-Elements-E-cone}
\begin{split}
x
&= \sum_{s,t \in \Sigma} A_{s,t} \otimes \delta_{st^{-1}}
= \sum_{s,t \in \Sigma} \sum_{r\in\Sigma} B_{s,r} (B^*)_{r,t} \otimes \delta_{st^{-1}} \\
&= \sum_{r\in\Sigma} \left(\sum_{s\in\Sigma} B_{s,r} \otimes \delta_{s}\right) \left(\sum_{t\in\Sigma} B_{t,r} \otimes \delta_t\right)^*
\in \mathcal{Q}_n(\Sigma).
\end{split}
\end{align}
This completes the proof.
\end{proof}

\begin{lemma}\label{lem:SOS-cones-MOU}
Let $\Gamma$ be a discrete group and $\Sigma\subseteq \Gamma$ be a finite subset. 
The following statements hold:
\begin{itemize}
\item[(i)] $\mathcal{Q}_n(\Sigma) \subseteq \mathrm{M}_n(\mathrm{C}^*(\Gamma))^+$, $n\in \bb{N}$.

\item[(ii)] $(\mathcal{Q}_n(\Sigma))_{n \in \mathbb{N}}$ is a compatible family of proper convex cones on the *-vector space $\mathrm{S}[\Sigma\Sigma^{-1}]$.

\item[(iii)] If $(\mathcal{C}_n)_{n \in \mathbb{N}}$ is a compatible family of convex cones on the *-vector space $\mathrm{S}[\Sigma\Sigma^{-1}]$ such that 
\begin{equation}\label{eq_dij}
\left\{(\delta_{s_is_j^{-1}})_{i,j=1}^n \mid s_1,\dots,s_n \in \Sigma\right\}\subseteq \cl C_n, \ \ \ n\in \bb{N},
\end{equation}
then $\mathcal{Q}_n(\Sigma)\subseteq \cl C_n$, $n\in \bb{N}$.

\item[(iv)] The element $\delta_e$ is a matrix order unit for the family $(\mathcal{Q}_n(\Sigma))_{n \in \mathbb{N}}$.
\end{itemize}
In particular, the triple $(\mathrm{S}[\Sigma\Sigma^{-1}], (\mathcal{Q}_n(\Sigma))_{n\in\mathbb{N}}, \delta_e)$ is a matrix order unit space and its canonical inclusion into $\mathrm{C}^*(\Gamma)$ is ucp.
\end{lemma}
	
\begin{proof}
Item (i) follows from the fact that sums of squares are positive, while
item (ii) follows from (\ref{eqn:SOS-cone-positive-matrix}).

\smallskip
\noindent
(iii) We first note that 
\begin{align} 
\left(\delta_{s_i s_j^{-1}}\right)_{i,j=1}^n 
= 
\left(\sum_{i=1}^n E_i \otimes \delta_{s_i}\right)\left(\sum_{i=1}^n E_i \otimes \delta_{s_i}\right)^* \in \mathcal{Q}_n(\Sigma),
\end{align}
where the $E_i$ are the matrix units in $\mathrm{M}_{n,1}$ and $s_1,\dots,s_n \in \Sigma$; thus the matrix $(\delta_{s_is_j^{-1}})_{i,j=1}^n$ is an element of $\mathcal{Q}_n(\Sigma)$.
Next, suppose that $(\mathcal{C}_n)_{n \in \mathbb{N}}$ is a compatible family of convex proper cones satisfying (\ref{eq_dij}) and let $A := (A_{i,j})_{i,j}\in {\rm M}_k(\mathrm{M}_n)^+$. 
We have that $A \otimes (\delta_{s_is_j^{-1}})_{i,j=1}^k \in \mathcal{C}_{k^2n}$; again by compatibility, 
\begin{align}
\sum_{i,j=1}^k A_{i,j} \otimes \delta_{s_is_j^{-1}}\in 
\mathcal{C}_n,    
\end{align}
implying, by Lemma \ref{lem:SOS-cones-positive-matrix-version}, that $\mathcal{C}_n \supseteq \mathcal{Q}_n(\Sigma)$. 

\smallskip
\noindent
(iv) 
Let $s\in \Sigma$ and $A\in \mathrm{M}_n$, and set 
\begin{align}\label{eqn:x(s,A)}
x := x(s,A) := A\otimes \delta_s + A^* \otimes \delta_{s^{-1}};
\end{align}
clearly, $x \in \mathrm{M}_n(\mathrm{S}[\Sigma\Sigma^{-1}])_\mathrm{h}$. 
Write $A = V|A|$ in its polar decomposition and set $A_1 := V|A|^{\frac{1}{2}}$, $A_2 := |A|^{\frac{1}{2}}$. 
Then
\begin{align}
&2\lVert A\rVert \cdot (\mathbf{1}_n \otimes \delta_e) - x \\
&\geq (V|A|V^*+|A|) \otimes \delta_e - (V|A|^{\frac{1}{2}}|A|^{\frac{1}{2}} \otimes \delta_s + |A|^{\frac{1}{2}}|A|^{\frac{1}{2}}V^* \otimes \delta_{s^{-1}}) \\
&= (A_1A_1^* + A_2A_2^*) \otimes \delta_e - A_1A_2^* \otimes \delta_s - A_2A_1^* \otimes \delta_{s^{-1}}  \\
&= \left(A_1 \otimes \delta_s - A_2 \otimes \delta_e\right) \left(A_1 \otimes \delta_s - A_2 \otimes \delta_e\right)^* 
\in  \mathcal{Q}_n(\Sigma).
\end{align}
Note that a general hermitian element $y \in \mathrm{M}_n(\mathrm{S}[\Sigma\Sigma^{-1}])_\mathrm{h}$ is of the form $y = \sum_{s \in \Sigma} x(s,A_s)$, with $x(s,A_s)$ as in (\ref{eqn:x(s,A)}).
Setting $r := 2\sum_{s \in \Sigma} \lVert A_s \rVert$, we thus obtain $r \cdot (\mathbf{1}_n \otimes \delta_e) - y \in \mathcal{Q}_n(\Sigma)$.
This shows that $(\delta_e)_n$ is an order unit for the ordered vector space $(\mathrm{M}_n(\mathrm{S}[\Sigma\Sigma^{-1}]),\mathcal{Q}_n(\Sigma))$.
\end{proof}
	
The assumption that $\Sigma$ is finite allows us to show that $\delta_e$ is in fact an Archimedean matrix order unit for the matrix ordered vector space $(\mathrm{S}[\Sigma\Sigma^{-1}],(\mathcal{Q}_n(\Sigma))_{n\in\mathbb{N}})$; 
we show this in the next proposition, whose proof is a noncommutative and matrix version of the argument in the proof of \cite[Lemma 1.3]{Rud63}.
	
\begin{proposition}\label{prop:Finite-SOS-archimedean}
The matrix order unit space $(\mathrm{S}[\Sigma\Sigma^{-1}], (\mathcal{Q}_n(\Sigma))_{n\in\mathbb{N}}, \delta_e)$ is an operator system.
\end{proposition}
	
\begin{proof}
We show that the cones $\mathcal{Q}_n(\Sigma)$ are norm-closed in $\mathrm{M}_n(\mathrm{C}^*(\Gamma))$.
Let $d = |\Sigma\Sigma^{-1}|$.
We claim that every $x \in \mathcal{Q}_n(\Sigma)$ is a sum of at most $n^2d$ squares in $\mathrm{M}_n \otimes \mathrm{S}[\Sigma\Sigma^{-1}]$.
To see this, let $x = \sum_{i=1}^r y_iy_i^* \in \mathcal{Q}_n(\Sigma)$ and assume that $r > n^2d$.
Since 
\begin{align} 
\mathcal{Q}_n(\Sigma) - \mathcal{Q}_n(\Sigma) 
\subseteq \mathrm{M}_n(\mathrm{S}[\Sigma^{-1}\Sigma])_\mathrm{h},
\end{align}
we have
\begin{align} 
\dim_\mathbb{R}(\mathcal{Q}_n(\Sigma) - \mathcal{Q}_n(\Sigma)) 
\leq \mathrm{dim}_\mathbb{R}(\mathrm{M}_n(\mathrm{S}[\Sigma^{-1}\Sigma]))_\mathrm{h} 
= n^2d; 
\end{align}
thus there are $\alpha_1, \dots, \alpha_r \in \mathbb{R}\setminus\{0\}$ such that $\sum_{i=1}^r \alpha_i y_iy_i^* = 0$.
We may assume that $\alpha_i \leq \alpha_{i+1}$, for all $i = 1,\dots,r-1$.
It follows that 
\begin{align} 
x = \sum_{i=1}^{r-1} (1-\frac{\alpha_i}{\alpha_r}) y_iy_i^*,
\end{align}
which completes the proof of the claim.
		
Now, let $x^k = \sum_{i=1}^{n^2d} y_i^k (y_i^k)^*  \in \mathcal{Q}_n(\Sigma)$, $k\in \bb{N}$, with 
\begin{align} 
\lim_{k\to \infty} x^k = x \in \mathrm{M}_n \otimes \mathrm{S}[\Sigma\Sigma^{-1}].
\end{align}
For every $i=1,\dots,n^2d$, we have that 
\begin{align} 
\|y_i^k\|^2 \leq \|x^k\| \leq \sup_{k\in \bb{N}} \|x^k\| < \infty.
\end{align}
Hence the sequences $(y_i^k)_k$ are bounded and, since $\dim(\mathrm{M}_n \otimes \mathrm{S}[\Sigma^{-1}\Sigma]) < \infty$, we may choose convergent subsequences $(y_i^{k_j})_j$ and set $\lim_j y_i^{k_j} =: y_i$.
It follows that 
\begin{align}
x = \lim_{k\to \infty} x^k = \lim_{k\to \infty} \sum_{i=1}^{n^2d} y_i^k (y_i^k)^* = \lim_{j\to \infty} \sum_{i=1}^{n^2d} y_i^{k_j} (y_i^{k_j})^* = \sum_{i=1}^{n^2d} y_i y_i^*.
\end{align}
This shows that $x \in \mathcal{Q}_n(\Sigma)$, so $\mathcal{Q}_n(\Sigma)$ is norm closed in $\mathrm{M}_n(\mathrm{C}^*(\Gamma))$.		
		
To see that $\delta_e$ is an Archimedean matrix order unit, recall from \cite{PT09} that the order topology is a locally convex topology, whence equivalent to the norm topology, since $\mathrm{S}[\Sigma\Sigma^{-1}]$ is finite dimensional. 
Thus the cones $\mathcal{Q}_n(\Sigma)$ are closed in the order topology, and, by \cite[Theorem 2.30, Remark 3.4]{PT09}, this implies that $(\delta_e)_n$ is an Archimedean order unit for $\mathrm{M}_n(\mathrm{S}[\Sigma\Sigma^{-1})$.
\end{proof}

\begin{remark}\label{rem:maxKPTT}
\rm \autoref{lem:SOS-cones-MOU} and \autoref{prop:Finite-SOS-archimedean} show that $V := (\mathrm{S}[\Sigma\Sigma^{-1}], \mathcal{Q}_1(\Sigma), \delta_e)$ is an Archimedean order unit space.
By the minimality of the compatible family of cones $(\mathcal{Q}_n(\Sigma))_{n\in\mathbb{N}}$ (\autoref{lem:SOS-cones-MOU}), $\mathrm{SOS}(\Sigma)$ is the maximal operator system structure $\mathrm{OMAX}(V)$ on $V$ in the sense of \cite{PTT11}.
\end{remark}

\begin{definition}
We call the operator system $(\mathrm{S}[\Sigma\Sigma^{-1}], (\mathcal{Q}_n(\Sigma))_{n\in\mathbb{N}}, \delta_e)$ the \emph{sums-of-squares system} and denote it by $\mathrm{SOS}(\Sigma)$.
\end{definition}

Given two discrete groups $\Gamma_1$, $\Gamma_2$ and finite subsets $\Sigma_1 \subseteq \Gamma_1$, $\Sigma_2 \subseteq \Gamma_2$, we make the canonical identification
\begin{align}\label{eq_S1S2}
\begin{split}
\mathrm{S}[\Sigma_1\Sigma_1^{-1} \times \Sigma_2\Sigma_2^{-1}] &\cong \mathrm{S}[\Sigma_1\Sigma_1^{-1}] \otimes \mathrm{S}[\Sigma_2\Sigma_2^{-1}] \\
\delta_{(s_1t_1^{-1},s_2t_2^{-1})} &\leftrightarrow \delta_{s_1t_1^{-1}} \otimes \delta_{s_2t_2^{-1}},
\end{split}
\end{align}
as finite dimensional *-vector spaces.

\begin{proposition}\label{prop:SOS-max-tensor-product}
Let $\Gamma_1$, $\Gamma_2$ be discrete groups and $\Sigma_1 \subseteq \Gamma_1$, $\Sigma_2 \subseteq \Gamma_2$ be finite subsets.
Then $\mathrm{SOS}(\Sigma_1 \times \Sigma_2) \cong \mathrm{SOS}(\Sigma_1) \otimes_\mathrm{max} \mathrm{SOS}(\Sigma_2)$. 
\end{proposition}

\begin{proof}
One checks that, via the identification (\ref{eq_S1S2}), the cone  $\mathcal{Q}_1(\Sigma_1 \times \Sigma_2)$ 
corresponds to the cone 
\begin{align} 
\mathcal{D}_1^\mathrm{max} := \left\{\sum_{i=1}^r A_i \otimes B_i \mid A_i \in \mathcal{Q}_1(\Sigma_1), B_i \in \mathcal{Q}_1(\Sigma_2), r\in\mathbb{N}\right\}.
\end{align}
It follows that 
\begin{align} 
V_1 \otimes V_2 := (\mathrm{S}[\Sigma_1\Sigma_1^{-1}] \otimes \mathrm{S}[\Sigma_2\Sigma_2^{-1}], \mathcal{D}_1^\mathrm{max}, \delta_{e_1} \otimes \delta_{e_2})
\end{align}
is an Archimedean order unit space.
By \cite[Proposition 5.13]{KPTT11} we have 
\begin{align} 
\mathrm{OMAX}(V_1 \otimes V_2) = \mathrm{OMAX}(V_1) \otimes_\mathrm{max} \mathrm{OMAX}(V_2),
\end{align}
and the proof is complete in view of \autoref{rem:maxKPTT}.
\end{proof}

If $\Sigma\subseteq \Gamma$ is a finite subset, we say that an element $x\in \mathrm{M}_n(\mathrm{C}^*(\Gamma))$ is \emph{supported on} $\Sigma$ if $y = \sum_{s\in\Sigma} B_s \otimes \delta_s$, where $B_s \in \mathrm{M}_n$, $s\in \Sigma$. 
We say that an element $x\in \mathrm{M}_n(\mathrm{C}^*(\Gamma))$ \emph{stretches over} $\Sigma$ if $x = \sum_{s,t \in \Sigma} A_{s,t} \otimes \delta_{st^{-1}}$, where $A_{s,t} \in \mathrm{M}_n$, $s,t\in\Sigma$. 

\begin{definition}\label{def:Admitting-a-Fejer--Riesz-lemma}
Let $\Gamma$ be a discrete group and $\Sigma \subseteq \Gamma$ be a finite set.
We say that $\Sigma$ possesses the \emph{complete Fej\'{e}r--Riesz property} if, for every $n \in \mathbb{N}$ and every element $x \in \mathrm{M}_n(\mathrm{C}^*(\Gamma))^+$ that stretches over $\Sigma$, there exist $r\in \bb{N}$ and elements $y_i \in \mathrm{M}_n(\mathrm{C}^*(\Gamma))$, $i\in [r]$, supported on $\Sigma$, such that $x = \sum_{i=1}^r y_iy_i^*$.
\end{definition}

\begin{remark}
\rm 
Recall from \autoref{lem:SOS-cones-MOU} that $\mathcal{Q}_n(\Sigma) \subseteq \mathrm{M}_n(\mathrm{C}^*(\Gamma))^+$; in other words, the canonical inclusion map $\iota : \mathrm{SOS}(\Sigma) \rightarrow \mathrm{C}^*(\Gamma)$ is ucp.
Thus, by the definition of the operator system structure of $\mathrm{SOS}(\Sigma)$, the following are equivalent:
\begin{itemize}
\item[(i)]
$\Sigma$ possesses the complete Fej\'{e}r--Riesz property;

\item[(ii)]
the canonical map $\iota : \mathrm{SOS}(\Sigma) \rightarrow \mathrm{C}^*(\Gamma)$ is a complete order embedding;

\item[(iii)]
$\mathrm{M}_n(\mathrm{C}^*(\Gamma))^+ \cap \mathrm{M}_n(\mathrm{S}[\Sigma\Sigma^{-1}]) \subseteq \mathcal{Q}_n(\Sigma)$ for every $n\in \bb{N}$.
\end{itemize}
\end{remark}

\begin{example}\label{ex:Fejer--Riesz-circle}
By the operator-valued Fej\'{e}r--Riesz lemma for the circle, the subset $\Sigma_N := \{0,\dots,N\} \subseteq \mathbb{Z}$ has the complete Fej\'{e}r--Riesz property.
In this case, every positive matrix $x \in \mathrm{M}_n(\mathrm{C}^*(\mathbb{Z})_{(\Sigma_N\Sigma_N^{-1})})^+$ is in fact a \emph{single} square, that is, we have that $x = yy^*$ for some $y \in \mathrm{M}_n(\mathrm{S}[\Sigma])$. 
\end{example}

As the following example demonstrates, the (complete) Fej\'{e}r--Riesz property depends on the group in which a subset is considered.
While intervals in the group $\mathbb{Z}$ have the complete Fej\'{e}r--Riesz property (\autoref{ex:Fejer--Riesz-circle}), their one-to-one images under the quotient map $\mathbb{Z} \rightarrow C_m$ onto the cyclic group, for large enough $m$, fails to have this property.

\begin{example}\label{ex:Counterexample-FR-Cyclic}
In contrast with \autoref{ex:Fejer--Riesz-circle}, the subsets $\Sigma_N$ do not have the (complete) Fej\'{e}r-Riesz property when considered in the finite cyclic groups $C_m$, for $m > 2N+1$. 
Indeed, consider the operator system $\mathrm{C}^*(C_m)_{(\{-N,\ldots,N\})}$ which corresponds to the subset $\Sigma = \{0,\ldots, N\} \subseteq C_m$. 
Via finite Fourier transform, this operator system can be identified with the operator subsystem of $\mathrm{C}(C_m)$ consisting of functions of the form 
\begin{equation}
f = \left(C_m \ni l \mapsto \sum_{k=-N}^N \widehat f (k) e^{2 \pi i l k} \right).
\label{eq:finite-fourier}
\end{equation}
If the (complete) Fej\'{e}r-Riesz property did hold, it would imply that all such functions $f$ that are positive (as functions in $\mathrm{C}(C_m)$) allow for a decomposition $\widehat f = \sum_{i=1}^r \widehat g_i * \widehat g_i^*$, with $\widehat g \in \mathrm{C}^*(C_m)_{(\{ 0,\ldots, N\})}$. But that would imply that the function
\begin{align} 
\theta \mapsto \sum_{k=-N}^N \widehat f (k) e^ {i k \theta} =\sum_{k=-N}^N \sum_{i=1}^r (\widehat g_i * \widehat g_i^*)(k) e^ {i k \theta} = \sum_{i=1}^r \left\lvert \sum_{k=0}^N \widehat g_i(k) e^ {ik \theta} \right\rvert^ 2,  \quad (\theta \in [0,2\pi))
\end{align}
is positive on the whole of the circle group $\mathbb{T}$. 
This cannot be true as there may exist functions in $\mathrm{C}(C_m)$ with support in Fourier contained in $\{-N, \ldots, N\}$ which are positive as functions on $C_m$, while the corresponding trigonometric polynomial fails to be positive as a function on $\mathbb{T}$; an example is given in Figure \ref{fig:FRcyclic}.
	
We note that Fourier systems in the group $\mathrm{C}^*$-algebras of finite cyclic groups were investigated in \cite{Mui25}, including the remarkable observation that the minimal $\mathrm{C}^*$-covers of their dual operator systems are not necessarily commutative.
\end{example}

\begin{figure}[h]
	\centering
	\includegraphics[scale=.7]{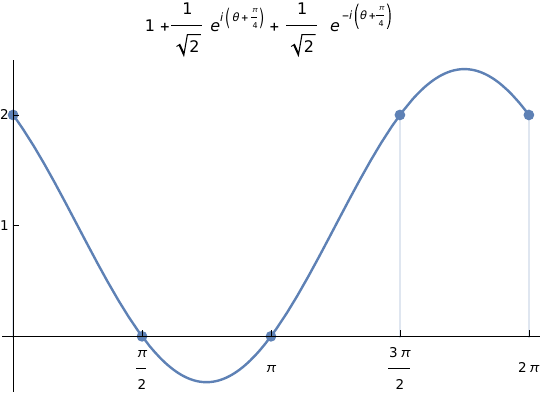}
	\caption{A function defined by the Fourier coefficients $\widehat f_{1} = r e^{i \pi/4}, \widehat f_0 = 1, \widehat f_{-1} =r e^{-i \pi/4}$ and a suitable $r=\sqrt{2}^{-1}$: the function $f$ on $C_4$ as given in \eqref{eq:finite-fourier} is positive but the corresponding trigonometric polynomial $1+2r  \cos (\theta + \pi/4)$ on $\mathbb{T}$ is not. 
    So $\widehat{f}$ is a positive element of the Fourier system $\mathrm{C}^*(C_4)_{(\{-1,0,1\})}$, but does not admit a decomposition $\widehat{f} = \sum_{i=1}^r\widehat{g}_i \ast \widehat{g}_i^*$ with $\widehat{g}_i \in \mathrm{S}[\{\delta_0,\delta_1\}] \subseteq \mathrm{C}^*(C_4)$. 
    }
    \label{fig:FRcyclic}
\end{figure}

\begin{example}\label{ex:Counterexample-Fejer--Riesz}
By \cite{Rud63,Sak80}, for $N \geq 2$, the set $\{0,\dots,N\}^2 \subseteq \mathbb{Z}^2$ does not have the complete Fej\'{e}r--Riesz property.
\end{example}
    	
\subsection{The Toeplitz system}

We now study the operator system of Toeplitz matrices associated to a finite subset $\Sigma$ of a discrete group $\Gamma$.

\begin{definition}
Let $\Gamma$ be a discrete group, $\Sigma\subseteq \Gamma$ be a finite set and $H$ be a Hilbert space. 
\begin{itemize}
\item[(i)] A \emph{$\mathcal{B}(H)$-valued Toeplitz matrix} is an operator $T = (T_{s,t})_{s,t \in \Sigma} \in \mathcal{B}(\ell^2(\Sigma) \otimes H)$ with $T_{s,t} \in \mathcal{B}(H)$, for all $s, t \in \Sigma$, such that $T_{s_1,t_1} = T_{s_2, t_2}$ whenever $s_1t_1^{-1} = s_2t_2^{-1}$.

\item[(ii)] A $\mathcal{B}(H)$-valued Toeplitz matrix $T = (T_{s,t})_{s,t\in\Sigma}$ is called \emph{positive semi-definite} if $T \in \mathcal{B}(\ell^2(\Sigma) \otimes H)^+$.
\end{itemize}
\end{definition}

The \emph{Toeplitz system} $\mathrm{C}^*(\Gamma)^{(\Sigma)}$ is defined as the operator system consisting of all ($\mathbb{C}$-valued) Toeplitz matrices; 
by its definition, $\mathrm{C}^*(\Gamma)^{(\Sigma)}$ is an operator subsystem of the C$^*$-algebra $\mathrm{M}_\Sigma$ of matrices indexed by $\Sigma$. 
Further, the operator system of $\mathcal{B}(H)$-valued Toeplitz matrices is the operator subsystem $\mathrm{C}^*(\Gamma)^{(\Sigma)} \otimes_\mathrm{min} \mathcal{B}(H)$ of $\mathcal{B}(\ell^2(\Sigma,H))$.
The $\mathcal{B}(H)$-valued Toeplitz matrices $T$ correspond bijectively to functions $u : \Sigma\Sigma^{-1} \to \mathcal{B}(H)$ via $T_{s,t} = u(st^{-1})$, $s,t \in \Sigma$.
Therefore we can define the restriction map 
\[
\rho : \mathrm{B}(\Gamma) \rightarrow \mathrm{C}^*(\Gamma)^{(\Sigma)}, u \mapsto (u(st^{-1}))_{s,t\in\Sigma}.
\]
By the definitions of the positive matrix cones over $\mathrm{B}(\Gamma)$ and $\mathrm{C}^*(\Gamma)^{(\Sigma)}$, the map $\rho$ is completely positive.

Our interest in the Toeplitz system (and the notation $\mathrm{C}^*(\Gamma)^{(\Sigma)}$) comes from its appearance in spectral truncations in noncommutative geometry \cite{CvS21,LvS24,Lei24} and its analysis in \cite{Far21}.
It is straightforward to check that 
\begin{align} 
\mathrm{C}^*(\Gamma)^{(\Sigma)} = P_{\Sigma} \mathrm{C}^*_\mathrm{r}(\Gamma) P_{\Sigma},
\end{align}
where $P_{\Sigma} : \ell^2(\Gamma) \rightarrow \ell^2(\Sigma)$ is the canonical orthogonal projection.

\subsection{Duality between the SOS and the Toeplitz systems}

We next show that Toeplitz operator systems are compatible with the sums-of-squares cone structure.

\begin{proposition}\label{prop:PD-functions-Toeplitz-CP-maps-SOS}
Let $\Gamma$ be a discrete group, $\Sigma\subseteq \Gamma$ be a finite set and $H$ be a Hilbert space.  
\begin{itemize}
\item[(i)] If $T = (T_{s,t})_{s,t\in\Sigma} \in \mathrm{C}^*(\Gamma)^{(\Sigma)} \otimes_\mathrm{min} \mathcal{B}(H)$ is a positive semi-definite $\mathcal{B}(H)$-valued Toeplitz matrix, then the map $\phi_T : \mathrm{SOS}(\Sigma) \to \cl B(H)$, given by $\phi_T(\delta_{st^{-1}}) = T_{s,t}$, $s,t\in \Sigma$, is completely positive. 
	
\item[(ii)] If  $\phi : \mathrm{SOS}(\Sigma) \to \cl B(H)$ is a completely positive map then there exists a unique positive semi-definite $\mathcal{B}(H)$-valued Toeplitz matrix $T_\phi$ such that $\phi = \phi_{T_{\phi}}$. 
\end{itemize}
\end{proposition}
	
\begin{proof}
(i) Let $T = (T_{s,t})_{s,t\in\Sigma} \in \mathrm{C}^*(\Gamma)^{(\Sigma)} \otimes_\mathrm{min} \mathcal{B}(H)$ be a positive semi-definite $\mathcal{B}(H)$-valued Toeplitz matrix.
By the proof of \autoref{lem:SOS-cones-MOU} (iii), it suffices to check that the block operator matrix $\phi_T^{(n)}((\delta_{s_is_j^{-1}})_{i,j=1}^n)$ with entries in $\mathcal{B}(H)$ is positive for any $s_1, \dots, s_n \in \Sigma$.
A direct computation gives
\begin{align}\label{eqn:T-minors-amplified-phi_u}
\phi_T^{(n)}((\delta_{s_is_j^{-1}})_{i,j=1}^n) 
= (\phi_T(\delta_{s_is_j^{-1}}))_{i,j=1}^n
= (T_{s_i, s_j})_{i,j=1}^n.
\end{align} 
If the $s_i$'s are pairwise distinct then $(T_{s_i, s_j})_{i,j=1}^n$ is positive semi-definite as a principal minor of the positive semi-definite matrix $T$.
Suppose that $s_1 = s_2$ and the $s_2, \dots, s_n$ are distinct; then we can write
\begin{align}
(T_{s_i, s_j})_{i,j=1}^n = \begin{bmatrix} T_{s_2, s_2} & R \\ C & (T_{s_i,s_j})_{i,j=2}^n \end{bmatrix}
\end{align}
where $R$ and $C$ are the first row and the first column of $(T_{s_i,s_j})_{i,j=2}^n$, respectively.
The matrix $(T_{s_i,s_j})_{i,j=2}^n$ is positive semi-definite as a principal minor of $T$, and thus $(T_{s_i, s_j})_{i,j=1}^n$ is positive semi-definite (see for example \cite[Remark 5.1.1]{EKT25}).
Inductively we have that $(T_{s_i, s_j})_{i,j=1}^n$ is positive semi-definite for any collection $s_1, \dots, s_n \in \Sigma$.

\smallskip
\noindent
(ii) Let $\phi \in \mathcal{CP}(\mathrm{SOS}(\Sigma),\mathcal{B}(H))$.
Replacing $\phi_T$ by $\phi$, $T$ by $T_\phi$ in (\ref{eqn:T-minors-amplified-phi_u}) and reading backwards for all $s_i \in \Sigma$, we see that $T_\phi$ is positive semi-definite.

The uniqueness clause follows from the fact  that the assignments $T \mapsto \phi_T$ and $\phi \mapsto T_\phi$ are inverse to each other.
\end{proof}

\begin{corollary}\label{cor:Toeplitz-SOS-dual-op-sys}
Let $\Gamma$ be a discrete group and $\Sigma \subseteq \Gamma$ be a finite subset.
The map
\begin{align}
\varphi : \mathrm{C}^*(\Gamma)^{(\Sigma)} &\rightarrow \mathrm{SOS}(\Sigma)^\mathrm{d}, \ \ \ 
T \mapsto \phi_T,
\end{align}
is a complete order isomorphism, and hence its dual map 
\begin{align} 
\varphi^\mathrm{d} : \mathrm{SOS}(\Sigma) \rightarrow (\mathrm{C}^*(\Gamma)^{(\Sigma)})^\mathrm{d}
\end{align}
is a complete order isomorphism.
Furthermore, 
\begin{align} 
\varphi(\mathbf{1}) = \left(\delta_{st^{-1}} \mapsto \begin{cases}
1, \text{if } s=t \\
0, \text{otherwise}
\end{cases}\right) 
\ \mbox{ and } \ 
\varphi^\mathrm{d}(\delta_e) = \left(T \mapsto \frac{1}{\lvert\Sigma\rvert} \mathrm{Tr}(T)\right).
\end{align}
\end{corollary}
	
\begin{proof}
The fact that the map $\varphi$ is a complete order isomorphism and the computation of $\varphi$ evaluated at the unit matrix follow directly from \autoref{prop:PD-functions-Toeplitz-CP-maps-SOS}.
Since $\mathrm{SOS}(\Sigma)$ is finite dimensional it coincides with its double dual and thus $\varphi^\mathrm{d}$ is also a complete order isomorphism.
To see that $\varphi^\mathrm{d}(\delta_e)$ is the normalised trace on the Toeplitz system, note that $\varphi^\mathrm{d}(\delta_{st^{-1}})(T) = \phi_T(\delta_{st^{-1}}) = T_{s,t}$, and the claim follows.
\end{proof}

\subsection{The operator extension properties}

Consider the canonical injective ucp map 
\begin{align} 
\iota : \mathrm{SOS}(\Sigma) \rightarrow \mathrm{C}^*(\Gamma), 
\end{align}
and the canonical surjective cp map
\begin{align} 
\rho : \mathrm{B}(\Gamma) \rightarrow \mathrm{C}^*(\Gamma)^{(\Sigma)}.
\end{align}
By virtue of \autoref{cor:Toeplitz-SOS-dual-op-sys} the restriction map $\rho$ is the dual map of the inclusion map $\iota$, and hence completely contractive.
Here we will examine when this map is an $\mathbf{MVS}$-/$\mathbf{OSp}$-quotient map.

\begin{definition}
Let $\Gamma$ be a discrete group.
A finite set $\Sigma \subseteq \Gamma$ is said to possess the \emph{positive operator extension property} if, for every Hilbert space $H$, the following condition holds:
\begin{quote}
Condition (PE): \quad For every positive semi-definite $\mathcal{B}(H)$-valued Toeplitz matrix $T = (T_{s,t})_{s,t\in\Sigma} \in \mathrm{C}^*(\Gamma)^{(\Sigma)} \otimes_\mathrm{min} \mathcal{B}(H)$, there exists a positive semi-definite function $u \in \mathrm{B}(\Gamma,H)^+$ such that $T_{s,t} = u(st^{-1})$, for all $s,t \in \Sigma$.
\end{quote}

Similarly, we say that $\Sigma$ possesses the \emph{cp extension property} if condition (PE) holds for every finite dimensional Hilbert space $H$.
\end{definition}

\begin{remark}\label{rmk:BW-argument-positive-operator-extension}
Note that if a finite set $\Sigma \subseteq \Gamma$ has the positive operator extension property, it trivially has the cp extension property.
Conversely, assume that $\Sigma$ has the cp extension property and let $H$ be a Hilbert space.
Let $\cl F(H)$ be the set of all finite rank projections in $\cl B(H)$, ordered by inclusion; clearly, $\cl F(H)$ is a directed set. 
Assume that $T = (T_{s,t})_{s,t\in\Sigma}$ is a positive semi-definite $\mathcal{B}(H)$-valued Toeplitz matrix.
Then, for every $P\in \cl F(H)$, we have
\begin{align} 
(I\otimes P)T(I\otimes P)\in \left(\mathrm{C}^*(\Gamma)^{(\Sigma)} \otimes_\mathrm{min} \mathcal{B}(PH)\right)^+.
\end{align}
By assumption, there exists a positive semi-definite function $u_P \in \mathrm{B}(\Gamma,PH)^+$ such that 
\begin{align} 
PT_{s,t}P = u_P(st^{-1}), \text{ for all } s,t \in \Sigma.
\end{align}
As $PH\subseteq H$, we may consider $u_P$ as elements of $\mathrm{B}(\Gamma,H)^+$.
By virtue of (\ref{eqn:Pos-def-fct-cp-maps-Gamma}), we consider further $u_P\in \cl{CP}(C^*(\Gamma),\cl B(H))$, and equip the latter space with Arveson's BW topology (see \cite{Arv69}), that is, the topology of point-weak* convergence. 
By \cite{Arv69}, the BW toplogy is compact on bounded sets; thus, the net $(u_P)_{P\in \cl F(H)}$ has a BW cluster point, say $u\in \mathrm{B}(\Gamma,H)^+$.
If $s,t\in \Sigma$ then $u(st^{-1})$ is a weak* cluster point of $((I\otimes P)T_{s,t}(I\otimes P))_{P\in \cl F(H)}$, and hence coincides with $T_{s,t}$, in other words, $u$ extends $T$. 
\end{remark}

\begin{proposition}\label{prop:Sigma-ext-property-quotient-map}
Let $\Gamma$ be a discrete group and $\Sigma \subseteq \Gamma$ be a finite set.
The set $\Sigma$ has the positive operator extension property if and only if the restriction map 
\begin{align} 
\rho : \mathrm{B}(\Gamma) \rightarrow \mathrm{C}^*(\Gamma)^{(\Sigma)}, 
u \mapsto (u(st^{-1}))_{s,t\in\Sigma},
\end{align}
is an $\mathbf{MVS}$-quotient map.
\end{proposition}

\begin{proof}
Recall from \autoref{subsec:Quotients} that $\rho$ is an $\mathbf{MVS}$-quotient map if and only if the induced map 
\begin{align} 
\tilde{\rho} : \setfrac{\mathrm{B}(\Gamma)}{\ker(\rho)} \rightarrow \mathrm{C}^*(\Gamma)^{(\Sigma)}
\end{align}
is a complete order isomorphism.

Assume that $\Sigma$ has the positive operator and hence the cp extension property.
Then for all $n \in \mathbb{N}$ and for every positive semi-definite Toeplitz matrix 
\begin{align}
T = (T_{s,t})_{s,t\in\Sigma}\in \mathrm{M}_n(\mathrm{C}^*(\Gamma)^{(\Sigma)})^+,
\end{align}
there is a positive semi-definite function $u : \Gamma \rightarrow \mathrm{M}_n$ with $T_{s,t} = u(st^{-1})$, for all $s,t\in\Sigma$;
in other words, we have that $$(\tilde{\rho}^{-1})^{(n)}(T) = u|_{\Sigma\Sigma^{-1}} \in \mathrm{B}(\Gamma,\mathbb{C}^n)^+|_{\Sigma\Sigma^{-1}} = \mathrm{M}_n(\setfrac{\mathrm{B}(\Gamma)}{\ker(\rho)})^+,$$
showing that the cp map $\tilde{\rho}$ is a complete order isomorphism.

Conversely, if $\rho$ is an $\mathbf{MVS}$-quotient map, then given a positive semi-definite Toeplitz matrix 
\begin{align} 
T = (T_{s,t})_{s,t\in\Sigma}\in \mathrm{M}_n(\mathrm{C}^*(\Gamma)^{(\Sigma)})^+,
\end{align}
the pre-image $u := (\tilde{\rho}^{-1})^{(n)}(T)$ is an element of 
\begin{align} 
\mathrm{M}_n(\setfrac{\mathrm{B}(\Gamma)}{\ker(\rho)})^+ = \mathrm{B}(\Gamma,\mathbb{C}^n)^+|_{\Sigma\Sigma^{-1}}.
\end{align}
That is, $u$ is the restriction of a positive semi-definite function $v : \Gamma \rightarrow \mathrm{M}_n$ to $\Sigma\Sigma^{-1}$, so $v(st^{-1}) = u(st^{-1}) = T_{s,t}$, for all $s,t\in\Sigma$. 
Hence, $\Sigma$ has the cp extension property.
By \autoref{rmk:BW-argument-positive-operator-extension} it follows that $\Sigma$ has the positive operator extension property.
\end{proof}

\begin{definition}
Let $\Gamma$ be a discrete group.
A finite set $\Sigma \subseteq \Gamma$ is said to possess the \emph{bounded operator extension property} if, for every Hilbert space $H$, the following condition holds:
\begin{quote}
    Condition (BE): \quad For every Toeplitz matrix $T = (T_{s,t})_{s,t\in\Sigma} \in \mathrm{C}^*(\Gamma)^{(\Sigma)} \otimes_\mathrm{min} \mathcal{B}(H)$, there exists a function $u \in \mathrm{B}(\Gamma,H)$ such that $T_{s,t} = u(st^{-1})$, for all $s,t \in \Sigma$, and $\lVert T \rVert = \lVert \phi_{u} \rVert_\mathrm{cb}$.
\end{quote}

Similarly, we say that $\Sigma$ possesses the \emph{cb extension property} if condition (BE) holds for every finite dimensional Hilbert space $H$.
\end{definition}

An analogue of \autoref{rmk:BW-argument-positive-operator-extension} applies:

\begin{remark}\label{rmk:BW-argument-bounded-operator-extension}
If a finite set $\Sigma \subseteq \Gamma$ has the bounded operator extension property, it trivially has the cb extension property.
Conversely, a BW-argument analogous to \autoref{rmk:BW-argument-positive-operator-extension} shows that the cb extension property implies the bounded operator extension property.
\end{remark}

\begin{proposition}\label{prop:Sigma-ext-property-quotient-map-2}
Let $\Gamma$ be a discrete group and $\Sigma \subseteq \Gamma$ be a finite set.
The set $\Sigma$ has the bounded operator extension property if and only if the restriction map 
\begin{align} 
\rho : \mathrm{B}(\Gamma) \rightarrow \mathrm{C}^*(\Gamma)^{(\Sigma)},
u \mapsto (u(st^{-1}))_{s,t\in\Sigma},
\end{align}
is an $\mathbf{OSp}$-quotient map.
\end{proposition}

\begin{proof}
Recall that $\rho$ is an $\mathbf{OSp}$-quotient map if and only if the induced cc map
\begin{align} 
\tilde{\rho} : \setfrac{\mathrm{B}(\Gamma)}{\ker(\rho)} \rightarrow \mathrm{C}^*(\Gamma)^{(\Sigma)}
\end{align}
is completely isometric.

Assume that $\Sigma$ has the bounded operator, and hence the cb, extension property.
For $n \in \mathbb{N}$ and $T \in \mathrm{M}_n(\mathrm{C}^*(\Gamma)^{(\Sigma)})$ let $u \in \mathrm{B}(\Gamma, \mathbb{C}^n)$ be an element of the $\mathrm{M}_n$-valued Fourier--Stieltjes algebra viewed as function $\Gamma \rightarrow \mathrm{M}_n$ such that
\begin{align} 
T = (u(st^{-1}))_{s, t \in \Sigma} = \rho^{(n)}(u)
\ \ \mbox{ and } \ \
\lVert T \rVert = \lVert \phi_{u} \rVert_\mathrm{cb}.
\end{align}
Then we have
\begin{align}
\lVert T\rVert 
& = \lVert \rho^{(n)}(u) \rVert
\leq \|u + \ker \rho^{(n)} \rVert
\leq \lVert u \rVert 
= \lVert \phi_u \rVert_{\mathrm{cb}}
= \lVert T \rVert,
\end{align}
showing that $\tilde{\rho}^{(n)}$ is isometric, as required.

Conversely, suppose that $\tilde{\rho}$ is completely isometric and fix $n \in \mathbb{N}$.
By surjectivity and the assumption there exists a $u \in \mathrm{B}(\Gamma, \mathbb{C}^n)$ such that
\begin{align} 
T = \rho^{(n)}(u)
\ \ \mbox{ and } \ \
\lVert T \rVert = \| u + \ker \rho^{(n)} \|.
\end{align}
For $m \in \mathbb{N}$ choose $v_m \in \ker \rho^{(n)}$ such that
\begin{align}
\lVert T \rVert \leq \lVert u+ v_m \rVert < \lVert T \rVert + \frac{1}{m}.
\end{align}
We see that
\begin{align}
\lVert v_m \rVert \leq \lVert u \rVert + \lVert u + v_m \rVert \leq \lVert u \rVert + \lVert T \rVert + 1.
\end{align}
Hence the sequence $(v_m)_{m \in \mathbb{N}}$ is uniformly bounded, and thus so is the induced sequence of cb maps $(\phi_{v_m})_{m \in \mathbb{N}}$ with respect to the cb-norm.
Since the set
\begin{align} 
\mathcal{CB}_{\lVert u \rVert + \lVert T \rVert+ 1}(\mathrm{C}^*(\Gamma), \mathrm{M}_n) := \{ \phi \in \mathrm{B}( \mathrm{C}^*(\Gamma), \mathrm{M}_n) \mid \lVert\phi\rVert_{\mathrm{cb}} \leq \lVert u \rVert + \lVert T \rVert + 1 \}
\end{align}
is compact in the BW-topology, we may take $v_\phi$ induced be a BW-limit $\phi$ of a subnet of $(\phi_{v_m})_{m \in \mathbb{N}}$ (and rename accordingly the net $(v_m)$).
We see that
\begin{align}
\limsup_m \lVert \phi_{u + v_m} \rVert_{\mathrm{cb}} = \limsup_m \lVert u + v_m \rVert = \lVert T \rVert.
\end{align}
Since the maps take values in $\mathrm{M}_n$ it follows that the BW-convergence on this bounded set coincides with the point-norm convergence, and hence
\begin{align}
\phi(\delta_{st^{-1}}) = \lim_m \phi_{v_m}(\delta_{st^{-1}}),
\end{align}
for all $s, t \in \Sigma$.
By the definition of $\rho$ we obtain
\begin{align}
\rho^{(n)}(v_{\phi})
& = (v_{\phi}(st^{-1}))_{s,t \in \Sigma}
= (\phi(\delta_{st^{-1}}))_{s,t \in \Sigma} \\
& = \lim_m (\phi_{v_m}(\delta_{st^{-1}}))_{s,t \in \Sigma}
= \lim_m \rho^{(n)}(v_m)
= 0,
\end{align}
where the limit above is taken in the entrywise-topology, and thus in the norm-topology, of $\mathrm{M}_{\Sigma}(\mathrm{M}_n(\mathbb{C}))$ as $\Sigma$ is finite.
Therefore $v_{\phi} \in \ker \rho^{(n)}$, from which we get 
\begin{align} 
T = \rho^{(n)}(u) = \rho^{(n)}(u+v_{\phi})
\ \ \mbox{ and } \ \
\lVert T \rVert = \| u + \ker \rho^{(n)} \| \leq \lVert u + v_{\phi} \rVert.
\end{align}
On the other hand, for $k \in \mathbb{N}$, $x \in \mathrm{M}_k(\mathrm{C}^{*}(\Gamma))$ and $\xi, \eta \in \mathbb{C}^{kn}$ we have
\begin{align}
| \langle \phi_{u+v_{\phi}}^{(k)}(x) \xi, \eta \rangle|
& = \lim_m | \langle \phi_{u+v_m}^{(k)}(x) \xi, \eta \rangle| \\
& \leq \limsup_m \lVert \phi_{u + v_m} \rVert_{\mathrm{cb}} \cdot \lVert x \rVert \cdot \lVert\xi \rVert \cdot \lVert \eta \rVert \\
& = \lVert T \rVert \cdot \lVert x \rVert \cdot \lVert \xi \rVert \cdot \lVert \eta \rVert.
\end{align}
From this we get
\begin{align}
\lVert u + v_{\phi} \rVert = \lVert \phi_{u+v_{\phi}} \rVert_{\mathrm{cb}} = \sup_{k \in \mathbb{N}} \lVert \phi_{u+v_{\phi}}^{(k)} \rVert \leq \lVert T \rVert,
\end{align}
showing that $\Sigma$ has the cb extension property.
By \autoref{rmk:BW-argument-bounded-operator-extension} it follows that $\Sigma$ has the bounded operator extension property.
\end{proof}

\subsection{The characterisation of the complete Fej\'{e}r--Riesz property}

We next show that the positive/bounded operator extension property is equivalent to the complete Fej\'{e}r--Riesz property.
We refer to \autoref{fig:CP-maps} as an illustration of the operator systems and matrix ordered vector spaces with the canonical maps which are relevant in the proof of the next result. 

\begin{theorem}\label{thm:Characterizations-CPEP}
Let $\Gamma$ be a discrete group and $\Sigma \subseteq \Gamma$ be a finite subset. 
The following statements are equivalent:
\begin{enumerate}
\item[(i)] The set $\Sigma$ has the positive operator extension property.

\item[(ii)] The set $\Sigma$ has the bounded operator extension property. 

\item[(iii)] The set $\Sigma$ has the complete Fej\'{e}r--Riesz property.

\item[(iv)] The Toeplitz system $\mathrm{C}^*(\Gamma)^{(\Sigma)}$ is completely order isomorphic to the dual operator system of the Fourier system $\mathrm{C}^*(\Gamma)_{(\Sigma\Sigma^{-1})}$.
\end{enumerate}
\end{theorem}

\begin{proof}
By \autoref{cor:Toeplitz-SOS-dual-op-sys}, the dual map of the canonical inclusion map 
\begin{align} 
\iota : \mathrm{SOS}(\Sigma) \rightarrow \mathrm{C}^*(\Gamma)
\end{align}
is the restriction map 
\begin{align} 
\rho : \mathrm{B}(\Gamma) \rightarrow \mathrm{C}^*(\Gamma)^{(\Sigma)}, u \mapsto (u(st^{-1}))_{s,t\in\Sigma},
\end{align}
where $\mathrm{B}(\Gamma)$ and $\mathrm{C}^*(\Gamma)^{(\Sigma)}$ are at the same time equipped with the structures of the matrix ordered vector space and operator space duals of $\mathrm{C}^*(\Gamma)$ and $\mathrm{SOS}(\Sigma)$.
Recall from \autoref{prop:Sigma-ext-property-quotient-map} and \autoref{prop:Sigma-ext-property-quotient-map-2} that the set $\Sigma$ has the bounded/positive operator extension property if and only if the restriction map $\rho$ is an $\mathbf{MVS}$-/$\mathbf{OSp}$-quotient map.
By \autoref{prop:Equiv-embeddings-quotients}, both of these statements  are equivalent to the statement that the map $\iota$ is a complete order embedding, i.e.\ that $\Sigma$ has the Fej\'{e}r--Riesz property.
This shows the equivalences [(i)$\Leftrightarrow$(ii)$\Leftrightarrow$(iii)].

\smallskip
\noindent
[(iii)$\Leftrightarrow$(iv)]. First note that, by finite-dimensionality, the dual $\mathrm{C}^*(\Gamma)_{(\Sigma\Sigma^{-1})}^\mathrm{d}$ of the Fourier system is an operator system, and by the Fej\'{e}r--Riesz property we have $\mathrm{C}^*(\Gamma)_{(\Sigma\Sigma^{-1})} \cong \mathrm{SOS}(\Sigma)$, which implies the claim by \autoref{cor:Toeplitz-SOS-dual-op-sys}.
Conversely, if the operator systems $\mathrm{C}^*(\Gamma)^{(\Sigma)}$ and $\mathrm{C}^*(\Gamma)_{(\Sigma\Sigma^{-1})}^\mathrm{d}$ are completely order isomorphic, we obtain complete order isomorphisms 
\begin{align}
\mathrm{C}^*(\Gamma)_{(\Sigma\Sigma^{-1})} 
\cong \mathrm{C}^*(\Gamma)_{(\Sigma\Sigma^{-1})}^\mathrm{dd} 
\cong (\mathrm{C}^*(\Gamma)^{(\Sigma)})^\mathrm{d} \cong \mathrm{SOS}(\Sigma),
\end{align}
so the one-to-one ucp map 
\begin{align} 
\iota : \mathrm{SOS}(\Sigma) \rightarrow \mathrm{C}^*(\Gamma)_{(\Sigma\Sigma^{-1})} \subseteq \mathrm{C}^*(\Gamma)
\end{align}
must be a complete order embedding, showing the implication [(iv)$\Rightarrow$(iii)].
\end{proof}

In view of Theorem \ref{thm:Characterizations-CPEP}, we may only refer to the operator extension property for a subset satisfying its equivalent conditions.

\subsection{$(\mathrm{min},\mathrm{max})$-nuclearity}

We close the section with an application to tensor products.
Recall the spatial implementation
\begin{align} 
\mathrm{C}^*(\Gamma)^{(\Sigma)} = P_{\Sigma} \mathrm{C}^*_\mathrm{r}(\Gamma) P_{\Sigma}
\subseteq \mathcal{B}(\ell^2(\Gamma)),
\end{align}
where $P_{\Sigma} : \ell^2(\Gamma) \rightarrow \ell^2(\Sigma)$ is the canonical orthogonal projection.
It follows that if $\Gamma_1$ and $\Gamma_2$ are discrete groups, and $\Sigma_i \subseteq \Gamma_i$ is a finite subset, $i = 1,2$, then
\begin{align} 
\mathrm{C}^*(\Gamma_1 \times \Gamma_2)^{(\Sigma_1 \times \Sigma_2)} 
\cong \mathrm{C}^*(\Gamma_1)^{(\Sigma_1)} \otimes_\mathrm{min} \mathrm{C}^*(\Gamma_2)^{(\Sigma_2)},
\end{align}
where we used that $\mathrm{C}^*_{\mathrm{r}}(\Gamma_1 \times \Gamma_2)$ is canonically *-isomorphic to $\mathrm{C}^*_{\mathrm{r}}(\Gamma_1) \otimes_\mathrm{min} \mathrm{C}^*_{\mathrm{r}}(\Gamma_2)$.
	
\begin{corollary}\label{cor:Min-max-nuclearity}
Let $\Gamma_1$ and $\Gamma_2$ be discrete groups and assume that $\Gamma_1$ or $\Gamma_2$ is amenable.
For $i=1,2$ let $\Sigma_i \subseteq \Gamma_i$ be a finite subset possessing the cp $\Sigma_i$-extension property. 
Then the following statements are equivalent:
\begin{enumerate}
\item[(i)] the subset $\Sigma_1 \times \Sigma_2 \subseteq \Gamma_1 \times \Gamma_2$ possesses the  operator extension property;
    
\item[(ii)] the pair of Toeplitz systems $(\mathrm{C}^*(\Gamma_1)^{(\Sigma_1)},\mathrm{C}^*(\Gamma_2)^{(\Sigma_2)})$ is $(\mathrm{min},\mathrm{max})$-nuclear.
\end{enumerate}
\end{corollary}
	
\begin{proof}
By universality and the amenability assumption we have a canonical *-iso\-morphism
\begin{align}\label{eqn:Nuclearity-grp-C*-alg}
\mathrm{C}^*(\Gamma_1 \times \Gamma_2) 
\cong \mathrm{C}^*(\Gamma_1) \otimes_\mathrm{max} \mathrm{C}^*(\Gamma_2) 
\cong \mathrm{C}^*(\Gamma_1) \otimes_\mathrm{min} \mathrm{C}^*(\Gamma_2).
\end{align}
Its restriction induces a complete order isomorphism
\begin{equation}\label{eqn:Nuclearity-grp-C*-alg-2}
\mathrm{C}^*(\Gamma_1 \times \Gamma_2)_{((\Sigma_1 \times \Sigma_2)(\Sigma_1 \times \Sigma_2)^{-1})}
\cong \mathrm{C}^*(\Gamma_1)_{(\Sigma_1 \Sigma_1^{-1})} \otimes_\mathrm{min} \mathrm{C}^*(\Gamma_2)_{(\Sigma_2 \Sigma_2^{-1})}.
\end{equation}

\smallskip
\noindent
[(i)$\Rightarrow$(ii)]. 
Assume that $\Sigma_1 \times \Sigma_2$ has the operator and hence in particular the cp extension property.
Then, by \autoref{thm:Characterizations-CPEP}, the identity map on generators 
\begin{align} 
\iota : \mathrm{SOS}(\Sigma_1 \times \Sigma_2) \rightarrow \mathrm{C}^*(\Gamma_1 \times \Gamma_2)
\end{align}
is a complete order embedding.
By the assumption that $\Sigma_i$ has the operator, and hence the cp, extension property it follows from \autoref{thm:Characterizations-CPEP} that $\Sigma_i$ has the complete Fej\'{e}r--Riesz property, i.e.\ $\mathrm{C}^*(\Gamma_i)_{(\Sigma_i\Sigma_i^{-1})}$ is completely order isomorphic to $\mathrm{SOS}(\Sigma_i)$, for $i=1,2$.
By \autoref{prop:SOS-max-tensor-product} and by (\ref{eqn:Nuclearity-grp-C*-alg-2}), we obtain 
\begin{align}
\mathrm{SOS}(\Sigma_1) \otimes_\mathrm{max} \mathrm{SOS}(\Sigma_2) 
&\cong \mathrm{SOS}(\Sigma_1 \times \Sigma_2) \\
&\cong \mathrm{C}^*(\Gamma_1 \times \Gamma_2)_{((\Sigma_1 \times \Sigma_2)(\Sigma_1 \times \Sigma_2)^{-1})} \\
&\cong \mathrm{C}^*(\Gamma_1)_{(\Sigma_1\Sigma_1^{-1})} \otimes_\mathrm{min} \mathrm{C}^*(\Gamma_2)_{(\Sigma_2\Sigma_2^{-1})} \\
&\cong \mathrm{SOS}(\Sigma_1) \otimes_\mathrm{min} \mathrm{SOS}(\Sigma_2),
\end{align}
that is, we have that the pair $(\mathrm{SOS}(\Sigma_1),\mathrm{SOS}(\Sigma_2))$ is
$(\mathrm{min},\mathrm{max})$-nuclear. 
Duality of the Toeplitz and sums of squares systems (\autoref{prop:PD-functions-Toeplitz-CP-maps-SOS}) together with the fact that $(X \otimes_\mathrm{min} Y)^\mathrm{d} \cong X^\mathrm{d} \otimes_\mathrm{max} Y^\mathrm{d}$ and $(X \otimes_\mathrm{max} Y)^\mathrm{d} \cong X^\mathrm{d} \otimes_\mathrm{min} Y^\mathrm{d}$, for finite-dimensional operator systems $X$ and $Y$ (see \cite{FP12}), implies $(\mathrm{min},\mathrm{max})$-nuclearity of the pair of Toeplitz systems $(\mathrm{C}^*(\Gamma_1)^{(\Sigma_1)},\mathrm{C}^*(\Gamma_2)^{(\Sigma_2)})$.

\smallskip
\noindent
[(ii)$\Rightarrow$(i)].
Assume that 
\begin{align} 
\mathrm{C}^*(\Gamma_1)^{(\Sigma_1)} \otimes_\mathrm{max} \mathrm{C}^*(\Gamma_2)^{(\Sigma_2)} \cong \mathrm{C}^*(\Gamma_1)^{(\Sigma_1)} \otimes_\mathrm{min} \mathrm{C}^*(\Gamma_2)^{(\Sigma_2)}.
\end{align}
Since $\Sigma_1$ and $\Sigma_2$ have the complete Fej\'{e}r--Riesz property and the minimal operator system tensor product is injective, we have that the tensor product identity map on generators 
\begin{align} 
\iota_1 \otimes \iota_2 : \mathrm{SOS}(\Sigma_1) \otimes_\mathrm{min} \mathrm{SOS}(\Sigma_2) \rightarrow \mathrm{C}^*(\Gamma_1) \otimes_\mathrm{min} \mathrm{C}^*(\Gamma_2)
\end{align}
is a complete order embedding.
By \autoref{cor:Toeplitz-SOS-dual-op-sys}, the assumption that the pair $(\mathrm{C}^*(\Gamma_1)^{(\Sigma_1)},\mathrm{C}^*(\Gamma_2)^{(\Sigma_2)})$ is $(\mathrm{min},\mathrm{max})$-nuclear,  
and the duality of the minimal and maximal operator system tensor product for finite-dimensional operator systems, we have 
\begin{align}
\mathrm{SOS}(\Sigma_1 \times \Sigma_2)
&\cong \mathrm{SOS}(\Sigma_1) \otimes_\mathrm{max} \mathrm{SOS}(\Sigma_2) \\
&\cong \mathrm{SOS}(\Sigma_1) \otimes_\mathrm{min} \mathrm{SOS}(\Sigma_2) \\
&\cong \mathrm{C}^*(\Gamma_1)_{(\Sigma_1 \Sigma_1{^-1})} \otimes_\mathrm{min} \mathrm{C}^*(\Gamma_2)_{(\Sigma_2 \Sigma_2^{-1})} \\
&\cong \mathrm{C}^*(\Gamma_1 \times \Gamma_2)_{((\Sigma_1 \times \Sigma_2)(\Sigma_1 \times \Sigma_2)^{-1})}.
\end{align}
By \autoref{thm:Characterizations-CPEP}, the set $\Sigma_1 \times \Sigma_2$ possesses the complete Fej\'{e}r--Riesz property, whence the operator extension property.
\end{proof}

A scalar-valued version of \autoref{thm:Characterizations-CPEP} was shown by Rudin \cite{Rud63} and (according to \cite{Sak80}) independently by Calderón--Pepinsky \cite{CP52}. 
From a result of Hilbert's \cite{Hil88}, they then inferred that there is a positive semi-definite trigonometric polynomial in two variables of respective degree $3$ which cannot be expressed as a sum of squares.
Equivalently, the subset $\{0,1,2,3\}^2 \subseteq \mathbb{Z}^2$ does not have the (scalar-valued) extension property.
This result was strengthened in \cite{Sak80} to the subset $\{0,1,2\}^2 \subseteq \mathbb{Z}^2$, which yields the following special case of \cite[Theorem 6.5]{Far21}.

\begin{corollary}\label{cor:Nonnuclearity-Toeplitz-system}
Let $n \geq 3$ and $\mathrm{T}_n \subseteq \mathrm{M}_n$ be the operator system of $n \times n$ Toeplitz matrices.
Then 
\begin{align}
\mathrm{T}_n \otimes_\mathrm{min} \mathrm{T}_n \ncong \mathrm{T}_n \otimes_\mathrm{max} \mathrm{T}_n.
\end{align}
\end{corollary}
	
\begin{proof}
By \cite{Rud63,Sak80}, for $\Sigma_n = \{0,\dots,n-1\}^2 \subseteq \mathbb{Z}^2$, the cone $\mathcal{Q}_1(\Sigma_n)$ is \emph{strictly} contained in the cone $\mathrm{C}^*(\mathbb{Z}^2)^+ \cap \mathrm{S}[\Sigma_n\Sigma_n^{-1}]$, i.e.\@ the set $\Sigma_n$ does not have the Fej\'{e}r--Riesz property.
The claim now follows from \autoref{thm:Characterizations-CPEP} and \autoref{cor:Min-max-nuclearity}.
\end{proof}

\section{Positive semi-definite functions on a positivity domain} \label{sec:Psd-fct-pos-domain}
	
In this section we discuss positive semi-definite operator-valued functions defined on positivity domains, addressing the question of when they admit  extensions to positive semi-definite functions defined on the whole group.
We thus drop the assumption that the domain of the given function has the form $\Sigma\Sigma^{-1}$ for some finite set, allowing simultaneously infinite domains. 
Our analysis follows a similar path as in the previous section by first identifying the operator system predual of the matrix ordered vector space of \black positive semi-definite functions;
this operator system can be constructed as a colimit from sums-of-squares operator systems.
In the case of the positivity domain consisting of five points in $\mathbb{Z}^2$ as in \autoref{fig:Positivity-domain-5pts} this allows us to realise the associated operator system as a coproduct, giving us access to its maximal and minimal $\mathrm{C}^*$-cover and allowing us to infer that there must be positive semi-definite operator-valued functions on this positivity domain which do not admit positive semi-definite extensions to the whole group $\mathbb{Z}^2$.

\subsection{The operator system $\mathrm{UOS}(\Delta)$ and the complete factorisation property}
	
Let $\Gamma$ be a discrete group and $\Delta \subseteq \Gamma$ be a positivity domain, i.e. a subset such that $\Delta^{-1} = \Delta$ and $e\in \Delta$.
We begin by defining the positive matricial cones which we will use to define an operator system structure on the $^*$-vector space 
\begin{align}
\mathrm{S}[\Delta] := \mathrm{span}\{\delta_s \mid s \in \Delta\}; \quad \delta_s^* := \delta_{s^{-1}}.
\end{align}
For a $^*$-vector space $V$ and a subset $\mathcal{P} \subseteq V_\mathrm{h}$, we denote by $\mathrm{cone}(\mathcal{P})$ the \emph{cone generated} by $\mathcal{P}$; thus, $\mathrm{cone}(\mathcal{P}) := [0,\infty) \cdot \mathcal{P} + \mathcal{P}$.
For $n\in \bb{N}$, set 
\begin{align}\label{eqn:E-cone-union}
\mathcal{U}_n(\Delta) 
:= \mathrm{cone} \left( \bigcup \big\{ \mathcal{Q}_n(\Sigma) \mid \Sigma \subseteq \Gamma, \Sigma\Sigma^{-1} \subseteq \Delta, \lvert\Sigma\rvert < \infty \big\} \right)
\subseteq \mathrm{M}_n(\mathrm{S}[\Delta])_\mathrm{h}.
\end{align}
Note that, since $\mathcal{Q}_n(\Sigma)$ is a cone for any $\Sigma \subseteq \Gamma$, in order to form $\mathcal{U}_n(\Delta)$ it suffices to take the convex hull of the union of the cones $\mathcal{Q}_n(\Sigma)$.
Since the cone generated by a union of cones is their sum, we have that 
\begin{align}\label{eqn:E-cone-sum}
\mathcal{U}_n(\Delta)
= \sum \big\{ \mathcal{Q}_n(\Sigma) \mid \Sigma \subseteq \Gamma, \Sigma\Sigma^{-1} \subseteq \Delta, \lvert\Sigma\rvert < \infty \big\}.
\end{align}

\begin{remark}\label{rem:Delta-cones-MOU}
It is straightforward to check that $(\mathcal{U}_n(\Delta))_{n \in \mathbb{N}}$ is a compatible family of convex proper cones, each respectively contained in $\mathrm{M}_n(\mathrm{C}^*(\Gamma))^+$, $n\in \bb{N}$.
It follows from \autoref{lem:SOS-cones-MOU} that, in fact, $(\mathcal{U}_n(\Delta))_{n \in \mathbb{N}}$ is the smallest such family whose $n$-th level contains the sets 
\begin{align}
\left\{(\delta_{st^{-1}})_{s,t\in\Sigma} \mid \Sigma \subseteq \Gamma, \Sigma\Sigma^{-1} \subseteq \Delta, \lvert\Sigma\rvert < \infty\right\}.
\end{align}
The element $\delta_e$ is a matrix order unit for the family $(\mathcal{U}_n(\Delta))_{n \in \mathbb{N}}$.
In particular, the triple $(\mathrm{S}[\Delta], (\mathcal{U}_n(\Delta))_{n\in\mathbb{N}}, \delta_e)$ is a matrix order unit space and the canonical inclusion into $\mathrm{C}^*(\Gamma)$ is ucp.
\end{remark}

We write $\mathrm{UOS}_0(\Delta)$ for the matrix order unit space $(\mathrm{S}[\Delta],(\mathcal{U}_n(\Delta))_{n \in \mathbb{N}},\delta_e)$ and $\mathrm{UOS}(\Delta) := \mathrm{Arch}(\mathrm{UOS}_0(\Delta))$ for its Archimedeanisation.
Thus $\mathrm{UOS}(\Delta)$ is an operator system;
its notation refers to its matricial cones being generated by unions of sums-of-squares cones.
In the next proposition, we point out that 
when $\Delta$ is finite, under an additional technical assumption, the matrix ordered operator space $\mathrm{UOS}_0(\Delta)$ is in fact already an operator system (without Archimedeanizing).

\begin{proposition}\label{prop:Finite-dim-P-Delta-quotient}
Let $\Gamma$ be a discrete group, $\Delta \subseteq \Gamma$ be a finite positivity domain and let $\{\Sigma_i\}_{i=1}^n$ be a family of subsets of $\Gamma$ such that $\bigcup_{i=1}^n \Sigma_i\Sigma_i^{-1} = \Delta$.
Assume that for every finite set $\Sigma \subseteq \Gamma$ with $\Sigma\Sigma^{-1} \subseteq \Delta$ there exists $i\in \{1,\dots,n\}$ such that $\mathcal{Q}_k(\Sigma) \subseteq \mathcal{Q}_k(\Sigma_i)$, for all $k\in\mathbb{N}$.
Then $\mathrm{UOS}_0(\Delta) = \mathrm{UOS}(\Delta)$.
\end{proposition}

\begin{proof}
The assumptions on the family $\{\Sigma_i\}_{i=1}^n$ ensure that we have
\begin{align}
\mathcal{U}_k(\Delta) = \mathrm{cone} \left( \bigcup_{i=1}^n \mathcal{Q}_k(\Sigma_i) \right),
\end{align}
for all $k \in \mathbb{N}$.
Since each cone $\mathcal{Q}_k(\Sigma_i)$ is closed by \autoref{prop:Finite-SOS-archimedean}, so is $\mathcal{U}_k(\Delta)$, whence $\mathrm{UOS}_0(\Delta)$ is an operator system and we have $\mathrm{UOS}_0(\Delta) = \mathrm{UOS}(\Delta)$.
\end{proof}

\begin{remark}
Note that the assumption in \autoref{prop:Finite-dim-P-Delta-quotient} ensures in particular that each $\Sigma_i$ is maximal in the sense that $(\Sigma_i \cup \gamma) (\Sigma_i \cup \gamma)^{-1} \nsubseteq \Delta$ for any $\gamma \in \Gamma \setminus \Sigma$.
\end{remark}
	
Before we move to the analysis of the operator systems of the form $\mathrm{UOS}(\Delta)$, we provide a motivating example; 
although this example may be known to the experts, we provide  sufficient operator-algebraic details in order to depict the differences with the sums-of-squares operator systems. 

\begin{example}\label{ex:5-pt-pos-dom-Z2}
Consider the positivity domain 
\begin{align}\label{eqn:Positivity-domain-5pts}
\Delta := \{(0,0), (0,\pm 1), (\pm 1,0)\} \subseteq \mathbb{Z}^2,
\end{align}
as in \autoref{fig:Positivity-domain-5pts}.
For notational convenience, set 
\begin{align}
\Sigma_\mathrm{h} := \{(0,0), (1,0)\}, 
\ &
\mathrm{SOS}_\mathrm{h} := \mathrm{SOS}(\Sigma_\mathrm{h}), \\
\Sigma_\mathrm{v} := \{(0,0), (0,1)\}, 
\ &
\mathrm{SOS}_\mathrm{v} := \mathrm{SOS}(\Sigma_\mathrm{v}). \end{align}
Our aim is to show that
\begin{align} 
\mathrm{UOS}(\Delta) 
= \mathrm{SOS}_\mathrm{h} \oplus_\mathbf{1} \mathrm{SOS}_\mathrm{v}.
\end{align}

\begin{figure}[h]
	\centering
	\begin{tikzpicture}[scale=1]
		\foreach \x in {-2,...,2} {
			\foreach \y in {-2,...,2} {
				\filldraw[black] (\x,\y) circle (.5pt);
			}
		}
				
		\filldraw[black] (-1,0) circle (2pt);
		\filldraw[black] (0,-1) circle (2pt);
		\filldraw[black] (0,0) circle (2pt);
		\filldraw[black] (1,0) circle (2pt);
		\filldraw[black] (0,1) circle (2pt);
    	\end{tikzpicture}
		\caption{The positivity domain $\{(0,0),\allowbreak (0,\pm 1), (\pm 1,0)\} \subseteq \mathbb{Z}^2$.}
		\label{fig:Positivity-domain-5pts}
\end{figure}
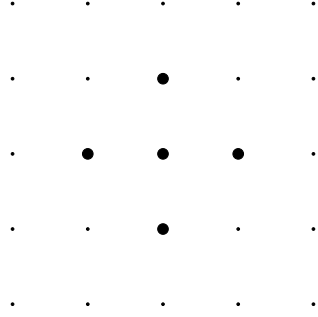
		
We first show that if $\Sigma \subseteq \Gamma$ with $|\Sigma|< \infty$ and $\Sigma - \Sigma \subseteq \Delta$, then $|\Sigma| \leq 2$.
Indeed, if there are distinct $s_1, s_2, s_3 \in \Sigma$ with their pairwise differences in $\Delta$, then these differences have to be distinct elements in $\Delta$ and not $(0,0)$.
Set $e_1 := (1,0), e_2:=(0,1) \in \mathbb{Z}^2$.
Without loss of generality assume that $s_1 = s_2 + e_1$.
If $s_1 = s_3 + e_2$, or if $s_1 = s_3 - e_2$, then we get
\begin{align}
e_1 \pm e_2 = (s_1 - s_2) + (s_3 - s_1) \in \Sigma - \Sigma \subseteq \Delta,
\end{align}
which is a contradiction.
If $s_2 = s_3 + e_2$, or if $s_2 = s_3 - e_2$, then we get
\begin{align}
e_1 \mp e_2 = (s_1 - s_2) + (s_2 - s_3) \in \Sigma - \Sigma \subseteq \Delta,
\end{align}
which is again a contradiction.
Hence $\Sigma = \{s\}$, or $\Sigma = \{s_1, s_2\}$.

If $\Sigma = \{s_1, s_2\}$, the difference $s_1 - s_2 \in \{\pm e_1, \pm e_2\}$ implies that either $\Sigma \subseteq \Sigma_{\mathrm{h}} + s$ for some $s \in \mathbb{Z}^2$, or $\Sigma \subseteq \Sigma_{\mathrm{v}} + s$ for some $s \in \mathbb{Z}^2$.
These inclusions trivially hold for $\Sigma = \{s\}$.
Hence
\begin{align} 
\mathcal{Q}_n(\Sigma) \subseteq \mathcal{Q}_n(\Sigma_{\mathrm{h}}) \cup \mathcal{Q}_n(\Sigma_{\mathrm{v}}), \ \ \ n \in \mathbb{N},
\end{align}
for every $\Sigma \subseteq \Gamma$ with $|\Sigma|< \infty$ and $\Sigma - \Sigma \subseteq \Delta$.
We claim that
\begin{equation}\label{eq_clUclQ}
\mathcal{U}_n(\Delta) = \mathcal{Q}_n(\Sigma_\mathrm{h}) + \mathcal{Q}_n(\Sigma_\mathrm{v}), \ \ \ n \in \mathbb{N}.
\end{equation}
Indeed, for $\Sigma = \Sigma_{\mathrm{h}} + s$ and $A,B \in \mathrm{M}_{n,k}$ we have that
\begin{align}
(A \otimes \delta_s + B \otimes \delta_{e_1 + s}) (A \otimes \delta_s + B \otimes \delta_{e_1 + s})^*
& = \\
& \hspace{-4cm} =
(A \otimes \delta_{e} + B \otimes \delta_{e_1}) (I \otimes \delta_s)(I \otimes \delta_s)^* (A \otimes \delta_{e} + B \otimes \delta_{e_1})^*, \\
& \hspace{-4cm} =
(A \otimes \delta_{e} + B \otimes \delta_{e_1}) (A \otimes \delta_{e} + B \otimes \delta_{e_1})^*,
\end{align}
and \autoref{lem:SOS-cones-positive-matrix-version} gives that $\mathcal{Q}_n(\Sigma) = \mathcal{Q}_n(\Sigma_{\mathrm{h}})$.
Likewise, if $\Sigma = \Sigma_{\mathrm{v}} + s$ then we get $\mathcal{Q}_n(\Sigma) = \mathcal{Q}_n(\Sigma_{\mathrm{v}})$.

By Proposition \ref{prop:Finite-dim-P-Delta-quotient} below, we have $\mathrm{UOS}(\Delta) = \mathrm{UOS}_0(\Delta)$; alternatively, we have that each $\mathcal{U}_n(\Delta) = \mathcal{Q}_n(\Sigma_\mathrm{h}) + \mathcal{Q}_n(\Sigma_\mathrm{v})$ is a closed cone and hence $\delta_e$ is an Archimedean matrix order unit for $\mathrm{UOS}_0(\Delta)$.
We further claim that the canonical ucp inclusion maps 
\begin{align} 
\phi_\mathrm{h} : \mathrm{SOS}_\mathrm{h} \rightarrow \mathrm{UOS}(\Delta) 
\ \ \mbox{ and } \ \ 
\phi_\mathrm{v} : \mathrm{SOS}_\mathrm{v} \rightarrow \mathrm{UOS}(\Delta)
\end{align}
are complete order embeddings.
Indeed, we have 
\begin{align} 
\mathrm{M}_n(\mathrm{S}[\{(-1,0),(0,0),(1,0)\}]) \cap \mathcal{U}_n(\Delta) = \mathcal{Q}_n(\Sigma_\mathrm{h}),
\end{align}
so $\phi_\mathrm{h}$ is a complete order embedding; analogous arguments apply for $\phi_\mathrm{v}$.
		
We now show that the operator system $\mathrm{UOS}(\Delta)$ is the direct sum of two copies of the Fourier system $\mathrm{C}^*(\mathbb{Z})_{(\{-1,0,1\})} \subseteq \mathrm{C}^*(\mathbb{Z})$ amalgamated over the unit, i.e.
\begin{align}\label{eqn:P-Delta-amalgamated-direct-sum}
\mathrm{UOS}(\Delta) 
= \mathrm{C}^*(\mathbb{Z})_{(\{-1,0,1\})} \oplus_\mathbf{1} \mathrm{C}^*(\mathbb{Z})_{(\{-1,0,1\})};
\end{align}
in fact, we have 
\begin{align} 
\mathrm{UOS}(\Delta) 
= \mathrm{SOS}_\mathrm{h} \oplus_\mathbf{1} \mathrm{SOS}_\mathrm{v},
\end{align}
which we show by checking that $\mathrm{UOS}(\Delta)$ satisfies the universal property of the amalgamated direct sum.
To this end, assume that $X$ is an operator system, and let
\begin{align}
\psi_\mathrm{h} : \mathrm{SOS}_\mathrm{h} \rightarrow X
\ \ \mbox{ and } \ \
\psi_\mathrm{v} : \mathrm{SOS}_\mathrm{v} \rightarrow X
\end{align}
be ucp maps.
Define a map $\psi : \mathrm{UOS}(\Delta) \rightarrow X$ by 
\begin{align} 
\psi(\delta_{(0,0)}) := e_X, \ \
\psi(\delta_{(\pm 1,0)}) := \psi_\mathrm{h}(\delta_{\pm 1})
\ \ \mbox{ and } \ \
\psi(\delta_{(0,\pm 1)}) := \psi_\mathrm{v}(\delta_{\pm 1}).
\end{align}
Clearly, $\psi$ is the unique map such that $\psi \circ \phi_\mathrm{v} = \psi_\mathrm{v}$ and $\psi \circ \phi_\mathrm{h} = \psi_\mathrm{h}$.
If $x \in \mathcal{U}_n(\Delta)$, by (\ref{eq_clUclQ}), there exist $x_\mathrm{h} \in \mathcal{Q}_n(\Sigma_\mathrm{h})$, $x_\mathrm{v} \in \mathcal{Q}_n(\Sigma_\mathrm{v})$ such that $x = x_\mathrm{h} + x_\mathrm{v}$.
Then 
\begin{align} 
\psi^{(n)}(x) = \psi_\mathrm{h}^{(n)}(x_\mathrm{h}) + \psi_\mathrm{v}^{(n)}(x_\mathrm{v}) \in \mathrm{M}_n(X)^+,
\end{align}
so $\psi$ is ucp; this completes the proof that $\mathrm{UOS}(\Delta)$ has the universal property of the coproduct.
By the operator-valued Fej\'{e}r--Riesz lemma, we have 
\begin{align} 
\mathrm{SOS}_\mathrm{h} \cong \mathrm{SOS}_\mathrm{v} \cong \mathrm{C}^*(\mathbb{Z})_{(\{-1,0,1\})},
\end{align}
which shows (\ref{eqn:P-Delta-amalgamated-direct-sum}).
\end{example}

Realizing the operator system $\mathrm{UOS}(\Delta)$ as a coproduct as in (\ref{eqn:P-Delta-amalgamated-direct-sum}) allows us to compute its maximal and minimal $\mathrm{C}^*$-covers.
Recall that the universal unital $\mathrm{C}^*$-algebra $\mathrm{C}^*_\mathrm{u}(\mathrm{con})$ of a contraction is defined as the  universal unital C*-algebra generated by a unit $\mathbf{1}$ and a contraction $x \neq \mathbf{1}$ such that for any contraction $y \in \mathcal{B}(H)$ there is a canonical unital $^*$-homomorphism
\begin{align}
\Phi \colon \mathrm{C}^*(\mathbf{1}, x) \to \mathrm{C}^*(\mathbf{1}_{\mathcal{B}(H)}, y), \ \
x \mapsto y.
\end{align}
Proposition 5.1 from \cite{FKPT14} implies that 
\begin{align}\label{L:univ}
\mathrm{C}^*_{\max}\left(\mathrm{C}^*(\mathbb{Z})_{(\{-1,0,1\})}\right) = \mathrm{C}^*_\mathrm{u}(\mathrm{con}),
\end{align}
up to a canonical *-isomorphism. 

The universal C*-alebra $\mathrm{C}^*_\mathrm{u}(\mathrm{con}_1, \dots, \mathrm{con}_m)$ of $m$ contractions is defined in an analogous way.
By universality we have a canonical *-isomorphism with the $m$-times free product of $\mathrm{C}^*_\mathrm{u}(\mathrm{con})$ with itself over the unit.

We are now in position to compute the maximal and minimal $\mathrm{C}^*$-cover of the operator system $\mathrm{UOS}(\Delta)$ for the positivity domain $\Delta$ from \autoref{ex:5-pt-pos-dom-Z2}.
	
\begin{proposition}\label{prop:C*-covers-P-Delta}
Let
\begin{align}
\Delta = \{(0,0), (0,\pm 1), (\pm 1,0)\} \subseteq \mathbb{Z}^2
\end{align}
(see \autoref{fig:Positivity-domain-5pts}).
The following identities hold:
\begin{align}
\label{eqn:C*-max-5pt}
\mathrm{C}^*_\mathrm{max}(\mathrm{UOS}(\Delta)) 
&\cong 
\mathrm{C}^*_\mathrm{u}(\mathrm{con}_1,\mathrm{con}_2),
\end{align} 
and
\begin{align}\label{eqn:C*-min-5pt}
\mathrm{C}^*_\mathrm{min}(\mathrm{UOS}(\Delta)) 
&\cong \mathrm{C}^*(\mathbb{F}_2).
\end{align}
\end{proposition}
	
\begin{proof}
Recall that, by (\ref{eqn:P-Delta-amalgamated-direct-sum}), 
\begin{align}
\mathrm{UOS}(\Delta) 
= \mathrm{C}^*(\mathbb{Z})_{(\{-1,0,1\})} \oplus_\mathbf{1} \mathrm{C}^*(\mathbb{Z})_{(\{-1,0,1\})}.
\end{align}
The identity for the maximal $\mathrm{C}^*$-cover (\ref{eqn:C*-max-5pt}) follows by the compatibility of the coproduct with the amalgamated free product of the universal C*-algebra of the Fourier system $\mathrm{C}^*(\mathbb{Z})_{(\{-1,0,1\})}$.
That is, we have 
\begin{align}
\mathrm{C}^*_\mathrm{max}(\mathrm{UOS}(\Delta))
& = 
\mathrm{C}^*_\mathrm{max}\left(\mathrm{C}^*(\mathbb{Z})_{(\{-1,0,1\})} \oplus_\mathbf{1} \mathrm{C}^*(\mathbb{Z})_{(\{-1,0,1\})}\right) \\
& = 
\mathrm{C}^*_\mathrm{max}\left(\mathrm{C}^*(\mathbb{Z})_{(\{-1,0,1\})}\right) \ast_\mathbf{1} \mathrm{C}^*_\mathrm{max}\left(\mathrm{C}^*(\mathbb{Z})_{(\{-1,0,1\})}\right) \\
& = 
\mathrm{C}^*_\mathrm{u}(\mathrm{con}) \ast_\mathbf{1} \mathrm{C}^*_\mathrm{u}(\mathrm{con}),
\end{align}
where we applied (\ref{L:univ}) for the equality in the last line and \cite[Proposition 4.2]{Cha23} for the equality in the second line. 
Since the universal unital $\mathrm{C}^*$-algebra of two contractions equals $\mathrm{C}^*_\mathrm{u}(\mathrm{con}) \ast_\mathbf{1} \mathrm{C}^*_\mathrm{u}(\mathrm{con})$, we obtain (\ref{eqn:C*-max-5pt}).
		
For the identity involving the minimal $\mathrm{C}^*$-cover (\ref{eqn:C*-min-5pt}), note that, since the Fourier system
\begin{align}
\mathrm{C}^*(\mathbb{Z})_{(\{-1,0,1\})} = \mathrm{span} \{\delta_1^*, \delta_0, \delta_1\}
\subseteq \mathrm{C}^*(\mathbb{Z})
\end{align}
contains the unitary generator of $\mathrm{C}^*(\mathbb{Z})$, we have $\mathrm{C}^*_\mathrm{min}(\mathrm{C}^*(\mathbb{Z})_{(\{-1,0,1\})}) = \mathrm{C}^*(\mathbb{Z})$ \cite[Proposition 5.6]{Kav14}.
As noticed in \cite[Theorem 5.2]{FKPT18}, if $X \subseteq A$, $Y \subseteq B$ are operator systems which respectively contain the unitary generators of the $\mathrm{C}^*$-algebras $A$, $B$, then the coproduct $X \oplus_\mathbf{1} Y$ contains the unitary generators of $A \ast_\mathbf{1} B$, and hence, by \cite[Proposition 5.6]{Kav14}, we have $\mathrm{C}^*_\mathrm{min}(X) = A$, $\mathrm{C}^*_\mathrm{min}(Y) = B$ and 
\begin{align} 
\mathrm{C}^*_\mathrm{min}(X \oplus_\mathbf{1} Y) = A \ast_\mathbf{1} B.
\end{align}
In our case, this gives
\begin{align}
\mathrm{C}^*_\mathrm{min}(\mathrm{UOS}(\Delta))
& \cong 
\mathrm{C}^*_\mathrm{min}\left(\mathrm{C}^*(\mathbb{Z})_{(\{-1,0,1\})} \oplus_\mathbf{1} \mathrm{C}^*(\mathbb{Z})_{(\{-1,0,1\})}\right) \\
& \cong 
\mathrm{C}^*_\mathrm{min}\left(\mathrm{C}^*(\mathbb{Z})_{(\{-1,0,1\})}\right) \ast_\mathbf{1} \mathrm{C}^*_\mathrm{min}\left(\mathrm{C}^*(\mathbb{Z})_{(\{-1,0,1\})}\right) \\
& \cong 
\mathrm{C}^*(\mathbb{Z}) \ast_\mathbf{1} \mathrm{C}^*(\mathbb{Z})
\cong 
\mathrm{C}^*(\mathbb{F}_2),
\end{align}
and the proof is complete.
\end{proof}

\begin{remark}\label{rem:hyperi}
We point out that \cite[Theorem 5.2]{FKPT18} can be strengthened to operator systems which are hyperrigid in their minimal $\mathrm{C}^*$-cover, see \cite[Theorem 1]{Dun08}, \cite[Theorem 5.3.21]{DFK17}, \cite[Theorem 4.11]{Cha23}.
Note that an operator system which contains the unitary generators of its ambient $\mathrm{C}^*$-algebra is hyperrigid in it \cite[Proposition 4.8]{Cha23}.
In particular, the Fourier systems $\mathrm{C}^*(\mathbb{Z})_{(\{-1,0,1\})}$ and $\mathrm{C}^*(\mathbb{Z}^2)_{(\Delta)}$, for $\Delta$ as in \autoref{fig:Positivity-domain-5pts}, are hyperrigid in $\mathrm{C}^*(\mathbb{Z})$ and $\mathrm{C}^*(\mathbb{Z}^2)$, respectively;
as a consequence of \autoref{prop:C*-covers-P-Delta}, the operator system $\mathrm{UOS}(\Delta)$ is hyperrigid in $\mathrm{C}^*(\mathbb{F}_2)$.
\end{remark}
    
For the next definition recall the canonical map \begin{align}
\iota_0 : \mathrm{UOS}_0(\Delta) \rightarrow \mathrm{C}^*(\Gamma)_{(\Delta)}, \ \ \delta_s \mapsto \delta_s,
\end{align}
which is ucp by the definition of the positive cones $\mathcal{U}_n(\Delta)$ of the matrix order unit space $\mathrm{UOS}_0(\Delta)$.
In the diagram \autoref{fig:Factorization-properties} we furthermore include the quotient map $q : \mathrm{UOS}_0(\Delta) \rightarrow \mathrm{UOS}(\Delta)$ from the Archimedeanisation as well as the map $\iota : \mathrm{UOS}(\Delta) \rightarrow \mathrm{C}^*(\Gamma)_{(\Delta)}$ which is canonically induced by the universal property of the Archimedeanisation.
	
\begin{figure}[h]
\begin{tikzcd}
\mathrm{UOS}_0(\Delta) & \mathrm{C}^*(\Gamma)_{(\Delta)} \subseteq \mathrm{C}^*(\Gamma)\\
\mathrm{UOS}(\Delta) &
%
\arrow["\iota_0"', from=1-1, to=1-2]
\arrow["q", from=1-1, to=2-1]
\arrow["\iota", from=2-1, to=1-2]
\end{tikzcd}
\caption{The canonical maps $\iota_0$ and $\iota$.}
\label{fig:Factorization-properties}
\end{figure}

We will denote by $\mathcal{S}(\Delta)$ the family of finite subsets $\Sigma \subseteq \Gamma$ which satisfy $\Sigma\Sigma^{-1} \subseteq \Delta$.

\begin{definition}\label{def:Complete-factorization-property}
Let $\Gamma$ be a discrete group and $\Delta \subseteq \Gamma$ be a positivity domain.
We say that $\Delta$ possesses:
\begin{itemize}
\item[(i)] 
the \emph{complete strict factorisation property} if, for every $n \in \mathbb{N}$ and every element 
$x \in \mathrm{M}_n(\mathrm{C}^*(\Gamma))^+$ with 
finite support inside $\Delta$, there exist $r\in \bb{N}$, finite subsets $\Sigma_i\in \cl S(\Delta)$ and elements $y_i \in \mathrm{M}_n(\mathrm{C}^*(\Gamma))$, $i\in [r]$, supported on $\Sigma_i$, such that 
\begin{equation}\label{eq_precise}
x = \sum_{i=1}^r y_iy_i^*;   
\end{equation}

\item[(ii)] 
the \emph{complete factorisation property} if, for every $n \in \mathbb{N}$, every element $z \in \mathrm{M}_n(\mathrm{C}^*(\Gamma))^+$ with finite support inside $\Delta$, and every $\epsilon > 0$, the element $z + \epsilon \delta_e \otimes \mathbf{1}_n$ admits a decomposition of the form 
(\ref{eq_precise}).
\end{itemize}
\end{definition}

\begin{remark}\label{rmk:FP-equiv-embedding}
It follows from the definition of the operator system $\mathrm{UOS}(\Delta)$ that a positivity domain $\Delta \subseteq \Gamma$ has the complete factorisation property if the canonical map $\iota : \mathrm{UOS}(\Delta) \rightarrow \mathrm{C}^*(\Gamma)$ is a complete order embedding.
\end{remark}

\begin{corollary}\label{cor:Failure-complete-factorization-property-free-Delta}
The positivity domain $\Delta = \{(0,0), (0,\pm 1), (\pm 1,0)\} \subseteq \mathbb{Z}^2$, as in \autoref{fig:Positivity-domain-5pts}, does not possess the complete factorisation property.
\end{corollary}
	
\begin{proof}
Since $\mathrm{C}^*(\mathbb{Z}^2)$ is the C$^*$-algebra generated by the unitary generators of $\mathrm{C}^*(\mathbb{Z}^2)_{(\Delta)}$, it follows from \cite[Proposition 5.6]{Kav14} that $\mathrm{C}^*_\mathrm{min}(\mathrm{C}^*(\mathbb{Z}^2)_{(\Delta)}) = \mathrm{C}^*(\mathbb{Z}^2)$.
If the canonical isomorphism $\iota : \mathrm{UOS}(\Delta) \rightarrow \mathrm{C}^*(\mathbb{Z}^2)_{(\Delta)}$ was a complete order isomorphism, we would have $$\mathrm{C}^*(\mathbb{F}_2) = \mathrm{C}^*_\mathrm{min}(\mathrm{UOS}(\Delta)) \cong \mathrm{C}^*_\mathrm{min}(\mathrm{C}^*(\mathbb{Z}^2)_{(\Delta)}) = \mathrm{C}^*(\mathbb{Z}^2),$$ 
a contradiction. 
\end{proof}

\subsection{The operator system $\mathrm{UOS}(\Delta)$ as a colimit}

In this subsection, we express $\mathrm{UOS}(\Delta)$ as a colimit of sums-of-squares operator systems (see \autoref{prop:Colimit}).
We begin with some useful observations for the sums-of-squares operator systems.
We note that if $\Sigma \subseteq \Gamma$ and $\gamma \in \Gamma$ then $(\Sigma \gamma) (\Sigma \gamma)^{-1} = \Sigma \Sigma^{-1}$, and thus the identity map
\begin{align}
\iota : \mathrm{SOS}(\Sigma \gamma) \rightarrow \mathrm{SOS}(\Sigma)
\end{align}
is well defined.

\begin{lemma}\label{lem:sos-same}
Let $\Gamma$ be a discrete group, $\Sigma \subseteq \Gamma$ be a finite set and $\gamma\in \Gamma$.
\begin{itemize}
\item[(i)] The identity map
\begin{align}
\iota : \mathrm{SOS}(\Sigma \gamma) \rightarrow \mathrm{SOS}(\Sigma)
\end{align}
is unital and completely positive, and thus a complete order isomorphism.

\item[(ii)] If $\Sigma' \subseteq \Sigma$ then the identity map
\begin{align}
\iota : \mathrm{SOS}(\Sigma') \rightarrow \mathrm{SOS}(\Sigma)
\end{align}
is unital and completely positive.

\item[(iii)] If $\gamma = st^{-1}$ for some $s,t \in \Sigma$, then 
the identity map
\begin{align}
\iota : \mathrm{SOS}(\{e, \gamma\}) \rightarrow \mathrm{SOS}(\Sigma)
\end{align}
is unital and completely positive.
\end{itemize}
\end{lemma}

\begin{proof}
(i) 
For $n,k \in \mathbb{N}$ let $A_{s} \in \mathrm{M}_{n,k}$, $s \in \Sigma$.
We compute directly
\begin{align}
&\left( \sum_{s \in \Sigma} A_{s} \otimes \delta_{s\gamma} \right) \left( \sum_{s \in \Sigma} A_{s} \otimes \delta_{s\gamma} \right)^* \\
& =
\left( \sum_{s \in \Sigma} A_{s} \otimes \delta_{s} \right) (I \otimes \delta_{\gamma}) (I \otimes \delta_{\gamma})^* \left( \sum_{s \in \Sigma} A_{s} \otimes \delta_{s} \right)^* \\
& =
\left( \sum_{s \in \Sigma} A_{s} \otimes \delta_{s} \right) \left( \sum_{s \in \Sigma} A_{s} \otimes \delta_{s} \right)^*.
\end{align}
This shows that $Q_n(\Sigma \gamma) = Q_n(\Sigma)$ and thus the unital map is completely positive.
Applying in reverse for $\Sigma\gamma$ and $\gamma^{-1}$ we get that the inverse map is completely positive, as required.

\smallskip
\noindent
(ii) 
This follows directly from the definition of the cones $\mathcal{Q}_n(\Sigma)$.

\smallskip
\noindent
(iii) 
This follows by combining items (i) and (ii) since $\{e,\gamma\} = \{t, s\} t^{-1}$ and $\{t,s\} \subseteq \Sigma$.
\end{proof}

We now aim to describe the operator system $\mathrm{UOS}(\Delta)$ as a colimit in the category $\mathbf{OSy}$.
We introduce some notation that will be used in \autoref{prop:Colimit} below.
Let $\Gamma$ be a discrete group and $\Delta \subseteq \Gamma$ be a positivity domain. 
For $\Sigma_1, \Sigma_2 \in \mathcal{S}(\Delta)$ we write $\Sigma_1 \leq_\Delta \Sigma_2$ if:
\begin{enumerate}
\item[(i)] $\Sigma_1\Sigma_1^{-1} \subseteq \Sigma_2\Sigma_2^{-1}$, and 

\item[(ii)] the canonical inclusion map 
\begin{align} 
\iota_{2,1} : \mathrm{SOS}(\Sigma_1) \rightarrow \mathrm{SOS}(\Sigma_2), \ \
\delta_{st^{-1}} \mapsto \delta_{st^{-1}}
\end{align}
is unital and completely positive.
\end{enumerate} 
We call $\leq_\Delta$ the \emph{$\Delta$-partial order} on $\mathcal{S}(\Delta)$.
We write $\Sigma_1\sim_\Delta \Sigma_2$ if $\Sigma_1\leq_\Delta\Sigma_2$ and $\Sigma_2\leq_\Delta \Sigma_1$.

\begin{theorem}\label{prop:Colimit}
Let $\Gamma$ be a discrete group and $\Delta \subseteq \Gamma$ be a positivity domain.
Then the colimit of the diagram $((\mathrm{SOS}(\Sigma))_{\Sigma \in \mathcal{S}(\Delta)}, (\iota_{2,1})_{\Sigma_1 \leq_{\Delta} \Sigma_2})$ in the operator system category exists and coincides with $\mathrm{UOS}(\Delta)$.
\end{theorem}
	
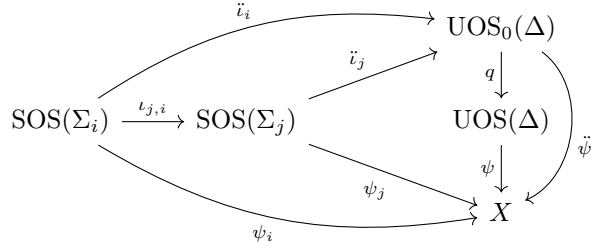
\begin{figure}[h]
	\centering
	\begin{tikzcd}
		&&& \mathrm{UOS}_0(\Delta)  \\
		\mathrm{SOS}(\Sigma_i) & \mathrm{SOS}(\Sigma_j) & &  \mathrm{UOS}(\Delta) \\
		&&& X 
		\arrow["\ddot{\iota}_{i}", bend left=20, from=2-1, to=1-4]
		\arrow["\ddot{\iota}_{j}", from=2-2, to=1-4]
		\arrow["q"', from=1-4, to=2-4]
		\arrow["\ddot{\psi}", bend left=60, from=1-4, to=3-4]
		\arrow["\iota_{j,i}", from=2-1, to=2-2]
		\arrow["\psi_{i}"', bend right=20, from=2-1, to=3-4]
		\arrow["\psi_{j}"', from=2-2, to=3-4]
		\arrow["\psi"', from=2-4, to=3-4]
	\end{tikzcd}
	\caption{The universal property of the colimit for the diagram of sums-of-squares system.}
\end{figure}
	
\begin{proof}
We check the universal property of the colimit, by first working in the matrix order unit space category and then Archimedeanizing.
To this end, let $X$ be an operator system and assume that, for every $\Sigma \in \mathcal{S}(\Delta)$, there is a ucp map $\psi_{\Sigma} : \mathrm{SOS}(\Sigma) \rightarrow X$ such that 
\begin{align}\label{eqn:Commutativity-colimit}
\psi_{\Sigma_2} \circ \iota_{2,1} = \psi_{\Sigma_1},
\end{align}
whenever $\Sigma_1 \leq_\Delta \Sigma_2$.
Define the map $\ddot{\psi} : \mathrm{UOS}_0(\Delta) \rightarrow X$ by letting $\ddot{\psi}(\delta_\gamma) := \psi_{\Sigma}(\delta_\gamma)$, for all $\gamma \in \Delta$ and any $\Sigma \in \mathcal{S}(\Delta)$ with $\gamma \in \Sigma\Sigma^{-1}$.

We first show that $\ddot{\psi}$ is a well-defined map.
To this end for $\gamma \in \Gamma$ set $\Sigma_\gamma = \{e, \gamma\}$.
For any $\Sigma$ containing $\gamma$ we have that $\{e,\gamma\} \leq_{\Delta} \Sigma$ by item (iii) of \autoref{lem:sos-same}.
If $\iota : \mathrm{SOS}(\{e,\gamma\}) \rightarrow \mathrm{SOS}(\Sigma)$ is the ucp inclusion map we have 
\begin{align}
\psi_{\Sigma_\gamma}(\delta_\gamma) = \psi_{\Sigma} \circ \iota(\delta_\gamma) = \psi_{\Sigma}(\delta_{\gamma}).
\end{align}
Therefore $\ddot{\psi}(\delta_\gamma)$ is independent of the choice of $\Sigma$ in which $\gamma$ is contained.
Further, the map $\ddot{\psi}$ is unital by the unitality of the maps $\psi_{\Sigma}$.

We next show that the map $\ddot{\psi}$ is completely positive.
If $x \in \mathrm{M}_n(\mathrm{UOS}_0(\Delta))^+ = \mathcal{U}_n(\Delta)$ is a positive matrix over $\mathrm{UOS}_0(\Delta)$, we have that $x = \sum_{i=1}^k x_i$, with $x_i \in \mathrm{M}_n(\mathrm{SOS}(\Sigma_i))^+ = \mathcal{Q}_n(\Sigma_i)$, for some $\Sigma_i \in \mathcal{S}(\Delta)$, $i=1,\dots,k$.
By the complete positivity of $\psi_{\Sigma_i}$, we have $\ddot{\psi}^{(n)}(x) = \sum_{i=1}^n \psi_{i}^{(n)}(x_i) \in \mathrm{M}_n(X)^+$, as required.

Therefore the matrix order unit space $\mathrm{UOS}_0(\Delta)$ possesses the universal property of the colimit in the matrix order unit space category.
Since passing to colimits commutes with Archimedeanisation  (see \autoref{lem:Arch-colimit-commute}), we obtain that $\mathrm{UOS}(\Delta)$ is the colimit of our diagram $((\mathrm{SOS}(\Sigma))_{\Sigma \in \mathcal{S}(\Delta)}, (\iota_{2,1})_{\Sigma_1 \leq_{\Delta} \Sigma_2})$ in the operator system category
\begin{align}
\mathrm{UOS}(\Delta)
= \mathrm{Arch}\left(\mathrm{UOS}_0(\Delta)\right)
= \mathrm{Arch}\left(\underset{\longrightarrow}{\lim}_\mathbf{MOU} \mathrm{SOS}(\Sigma)\right)
= \underset{\longrightarrow}{\lim}_\mathbf{OSy} \mathrm{SOS}(\Sigma),
\end{align}
and the proof is complete.
\end{proof}
	
To obtain a description of the diagram in \autoref{prop:Colimit} only in terms of subsets of the group $\Gamma$, it would be desirable to express the $\Delta$-partial order only in group theoretic terms.
We formulate this in the following question.

\begin{question}\label{qu:Delta-partial-order}
Let $\Gamma$ be a discrete group and $\Sigma_1, \Sigma_2 \subseteq \Gamma$ finite subsets with $\Sigma_1\Sigma_1^{-1} \subseteq \Sigma_2\Sigma_2^{-1}$.
When is the canonical map 
\begin{align} 
\mathrm{SOS}(\Sigma_1) \rightarrow \mathrm{SOS}(\Sigma_2), \ \ \delta_{st^{-1}} \mapsto \delta_{st^{-1}},
\end{align}
completely positive?
\end{question}

We next show how the cartesian product of positivity domain compiles with the maximal tensor product.

\begin{theorem}\label{prop:Product-delta-systems}
Let $\Gamma$, $\Gamma'$ be discrete groups and $\Delta \subseteq \Gamma$, $\Delta' \subseteq \Gamma'$ positivity domains.
We have that $\mathrm{UOS}(\Delta \times \Delta') \cong \mathrm{UOS}(\Delta) \otimes_\mathrm{max} \mathrm{UOS}(\Delta')$.
\end{theorem}  

Our proof below goes in three steps:
First, we check that the cartesian product of two diagrams of sums-of-squares systems is a diagram where all the maps are jointly completely positive.
Second, we show that the cartesian product of the colimits is the colimit of this diagram, again for jointly completely positive maps.
This follows by construction of the cones $\mathrm{M}_n(\mathrm{UOS}_0(\Delta))_+$ as the cones generated by the $\mathrm{SOS}$-cones which they contain.
Third, we conclude that the maximal tensor product $\mathrm{UOS}(\Delta) \otimes_\mathrm{max} \mathrm{UOS}(\Delta')$ satisfies the universal property of the colimit of the according diagram of maximal tensor products of sums-of-squares systems, by using the universal property of the maximal tensor product;
recall that the universal property of the maximal tensor product \cite[Theorem 5.8]{KPTT11} gives a canonical identification of completely positive maps out of the maximal tensor product with jointly completely positive maps out of the cartesian product of two operator systems.  

\begin{proof}
\emph{Step 1.} 
As above, denote by $\mathcal{S}(\Delta)$ the family of finite subsets $\Sigma \subseteq \Gamma$ with $\Sigma\Sigma^{-1} \subseteq \Delta$, and analogously for $\mathcal{S}(\Delta')$.
We equip $\mathcal{S}(\Delta)$ and $\mathcal{S}(\Delta')$ respectively with the $\leq_{\Delta}$- and $\leq_{\Delta'}$-orders, which induce the diagrams 
\begin{align}
\left((\mathrm{SOS}(\Sigma))_{\Sigma \in \mathcal{S}(\Delta)}, (\alpha_{2,1})_{\Sigma_1\leq_{\Delta}\Sigma_2}\right) \
\ \mbox{and} \ \
\left((\mathrm{SOS}(\Sigma'))_{\Sigma' \in \mathcal{S}(\Delta')}, (\beta_{2,1})_{\Sigma'_1\leq_{\Delta'}\Sigma'_2}\right),
\end{align}
as in \autoref{prop:Colimit}.
By functoriality of the maximal tensor product, we obtain a diagram 
\begin{align}\label{eqn:Diagram-Max-Tensor-Product-SOS-full}
\left((\mathrm{SOS}(\Sigma) \otimes_\mathrm{max} \mathrm{SOS}(\Sigma'))_{\Sigma \in \mathcal{S}(\Delta), \Sigma' \in \mathcal{S}(\Delta')}, (\alpha_{2,1} \otimes \beta_{2,1})_{\Sigma_1\leq_{\Delta}\Sigma_2, \Sigma'_1\leq_{\Delta'}\Sigma'_2}\right).
\end{align}
		
Let $X$ be an operator system and assume that, for all $\Sigma \in \mathcal{S}(\Delta)$ and $\Sigma' \in \mathcal{S}(\Delta')$, there are ucp maps $\psi_{\Sigma,\Sigma'} : \mathrm{SOS}(\Sigma) \otimes_\mathrm{max} \mathrm{SOS}(\Sigma') \rightarrow X$ such that 
\begin{align}\label{eqn:Diagram-Max-Tensor-Product-SOS}
\psi_{\Sigma_2,\Sigma'_2} \circ (\alpha_{2,1} \otimes \beta_{2,1}) 
= \psi_{\Sigma_1,\Sigma'_1},
\end{align}
for all $\Sigma_1, \Sigma_2 \in \mathcal{S}(\Delta)$ and $\Sigma'_1, \Sigma'_2 \in \mathcal{S}(\Delta')$ with $\Sigma_1 \leq_{\Delta} \Sigma_2$ and $\Sigma'_1 \leq_{\Delta'} \Sigma'_2$.
By the universal property of the maximal tensor product \cite[Theorem 5.8]{KPTT11}, for all $\Sigma \in \mathcal{S}(\Delta)$ and $\Sigma' \in \mathcal{S}(\Delta')$, the map 
\begin{align}
\phi_{\Sigma,\Sigma'} : \mathrm{SOS}(\Sigma) \times \mathrm{SOS}(\Sigma') \rightarrow X; \ \
(\delta_{st^{-1}},\delta_{s't'^{-1}}) \mapsto \psi_{\Sigma,\Sigma'}(\delta_{st'^{-1}}\otimes\delta_{s't'^{-1}}),
\end{align}
is jointly completely positive.
Note that (\ref{eqn:Diagram-Max-Tensor-Product-SOS}) implies 
\begin{align}\label{eqn:Diagram-Cartesian-Product-SOS}
\phi_{\Sigma_2,\Sigma'_2} \circ (\alpha_{2,1} \times \beta_{2,1}) 
= \phi_{\Sigma_1,\Sigma'_1},
\end{align}
for all $\Sigma_1, \Sigma_2 \in \mathcal{S}(\Delta)$ and $\Sigma'_1, \Sigma'_2 \in \mathcal{S}(\Delta')$	with $\Sigma_1 \leq_{\Delta} \Sigma_2$ and $\Sigma'_1 \leq_{\Delta'} \Sigma'_2$.
Moreover, the canonical inclusion map of $\mathrm{SOS}(\Sigma) \times \mathrm{SOS}(\Sigma')$ into the cartesian product of matrix ordered vector spaces $\mathrm{UOS}_0(\Delta) \times \mathrm{UOS}_0(\Delta')$ is jointly completely positive.

\smallskip
\noindent
\emph{Step 2.}
For $\gamma \in \Delta$ and $\gamma' \in \Delta'$ let $\Sigma \in \mathcal{S}(\Delta)$ and $\Sigma' \in \mathcal{S}(\Delta')$ be such that $\gamma \in \Sigma\Sigma^{-1}$ and $\gamma' \in \Sigma'\Sigma'^{-1}$.
Define 
\begin{align}
\ddot{\phi}(\delta_{\gamma},\delta_{\gamma'}) := \phi_{\Sigma,\Sigma'}(\delta_{st^{-1}},\delta_{s't'^{-1}}),
\end{align}
where $s,t \in \Sigma$, $s',t' \in \Sigma'$ are any elements satisfying $st^{-1} = \gamma$ and $s't'^{-1} = \gamma'$.
By item (iii) of \autoref{lem:sos-same} and (\ref{eqn:Diagram-Cartesian-Product-SOS}), we have that the map $\ddot{\phi}$ is well-defined on $\mathrm{UOS}_0(\Delta) \times \mathrm{UOS}_0(\Delta')$.
We also fix the notations 
\begin{align} 
q : \mathrm{UOS}_0(\Delta) \rightarrow \mathrm{UOS}(\Delta) \ \ \text{and} \ \ q' : \mathrm{UOS}_0(\Delta') \rightarrow \mathrm{UOS}(\Delta')
\end{align}
for the canonical quotient maps and 
\begin{align} 
\iota_\Sigma : \mathrm{SOS}(\Sigma) \rightarrow \mathrm{UOS}_0(\Delta)
\ \ \mbox{and} \ \
\iota_\Sigma' : \mathrm{SOS}(\Sigma') \rightarrow \mathrm{UOS}_0(\Delta')
\end{align}
for the canonical inclusion maps.

Since the matricial cones of the matrix ordered vector spaces $\mathrm{UOS}_0(\Delta)$ and $\mathrm{UOS}_0(\Delta')$ are generated by the matricial cones of the operator systems $\mathrm{SOS}(\Sigma)$ and $\mathrm{SOS}(\Sigma')$, for $\Sigma \in \mathcal{S}(\Delta)$ and $\Sigma' \in \mathcal{S}(\Delta')$, and by joint complete positivity of the maps $\phi_{\Sigma,\Sigma'}$, the map $\ddot{\phi}$ is jointly completely positive by construction.
Indeed, let $x \in \mathrm{M}_n(\mathrm{UOS}_0(\Delta))^+ = \mathcal{U}_n(\Delta)$ be a positive matrix over $\mathrm{UOS}_0(\Delta)$; then we have $x = \sum_{j=1}^k x_{j}$, with $x_j \in \mathrm{M}_n(\mathrm{SOS}(\Sigma_j))^+ = \mathcal{Q}_n(\Sigma_j)$, for some $\Sigma_j \in \mathcal{S}(\Delta)$, $j=1,\dots,k$.
Likewise for $y \in \mathrm{M}_n(\mathrm{UOS}_0(\Delta'))^+ = \mathcal{U}_n(\Delta')$ a positive matrix over $\mathrm{UOS}_0(\Delta')$, we have $y = \sum_{j'=1}^m y_{j'}$, with $y_{j'} \in \mathrm{M}_n(\mathrm{SOS}(\Sigma'_{j'}))^+ = \mathcal{Q}_n(\Sigma'_{j'})$, for some $\Sigma'_{j'} \in \mathcal{S}(\Delta')$, $j'=1,\dots,m$.
Applying $\ddot{\phi}$ we get
\begin{align}
\ddot{\phi}^{(n)}(x,y)
& =
\sum_{j=1}^k \sum_{j'=1}^m \ddot{\phi}^{(n)}(x_j, y_{j'})
=
\sum_{j=1}^k \sum_{j'=1}^m \phi_{\Sigma_j, \Sigma'_{j'}}^{(n)}(x_j, y_{j'}),
\end{align}
which is a sum of positive elements in $\mathrm{M}_n(X)$ as every $\phi_{\Sigma_j, \Sigma'_{j'}}$ is jointly completely positive, and thus $\ddot{\phi}^{(n)}(x,y)$ is positive.
\black

It follows that $\ddot{\phi}$ induces a unique jointly completely positive map on the cartesian product of the Archimedeanisations
\begin{align}
\phi : \mathrm{UOS}(\Delta) \times \mathrm{UOS}(\Delta') \rightarrow X
\end{align}
which satisfies 
$\phi \circ (q_1 \times q_2) = \ddot{\phi}$ and
\begin{align}
\phi \circ (q \times q') \circ (\iota_\Sigma \times \iota_\Sigma') = \phi_{\Sigma,\Sigma'}.
\end{align}
Indeed, fixing $x\in \mathrm{UOS}_0(\Delta)^+$, 
we have that the map $\phi_x : \mathrm{UOS}_0(\Delta') \rightarrow X$, given by $\phi_x(y) = \ddot{\phi}(x,y)$, is completely positive and hence induces a canonical completely positive map $\tilde{\phi}_x : \mathrm{UOS}(\Delta') \rightarrow X$.
Let $\theta : \mathrm{UOS}_0(\Delta) \times \mathrm{UOS}(\Delta') \rightarrow X$ be the (well-defined) map, satisfying $\theta(x,y) = \tilde{\phi}_x(y)$, $x\in \mathrm{UOS}_0(\Delta)^+$,
$y\in \mathrm{UOS}(\Delta')$. 
By construction, $\theta$ is jointly completely positive. 
For $y\in \mathrm{UOS}(\Delta')^+$, let $\theta_y : \mathrm{UOS}_0(\Delta) \rightarrow X$ be the map, 
given by $\theta_y(x) = \theta(x,y)$; 
as before, the latter map gives rise to a canonical completely positive map $\tilde{\theta}_y : \mathrm{UOS}(\Delta) \rightarrow X$. 
Finally, let $\phi(x,y) = \tilde{\theta}_y(x)$, $x\in \mathrm{UOS}(\Delta)$, $y\in \mathrm{UOS}(\Delta')$.
		
\begin{figure}[h]
	\centering
	\begin{tikzcd}[scale=1, transform shape, column sep=1.75em]
		&&& \mathrm{UOS}_0(\Delta) \times \mathrm{UOS}_0(\Delta')  \\
		\mathrm{SOS}(\Sigma_1) \times \mathrm{SOS}(\Sigma'_1) & \mathrm{SOS}(\Sigma_2) \times \mathrm{SOS}(\Sigma'_2) & &  \mathrm{UOS}(\Delta) \times \mathrm{UOS}(\Delta') \\
		&&& X 
		\arrow["\iota_{\Sigma_1} \times \iota_{\Sigma'_1}", bend left=20, from=2-1, to=1-4]
		\arrow["{\iota}_{\Sigma_2} \times \iota_{\Sigma'_2}", from=2-2, to=1-4]
		\arrow["q \times q'"', from=1-4, to=2-4]
		\arrow["\ddot{\phi}", bend left=80, from=1-4, to=3-4]
		\arrow["\alpha_{2,1} \times \beta_{2,1}"{yshift=4pt}, from=2-1, to=2-2]
		\arrow["\phi_{\Sigma_1,\Sigma'_1}"', bend right=20, from=2-1, to=3-4]
		\arrow["\phi_{\Sigma_2,\Sigma'_2}"', from=2-2, to=3-4]
		\arrow["\phi"', from=2-4, to=3-4]
	\end{tikzcd}
	\caption{The universal property of the colimit for the cartesian product of two diagrams of sums-of-squares systems.
	All maps indicated in this diagram are jointly completely positive.}
\end{figure}
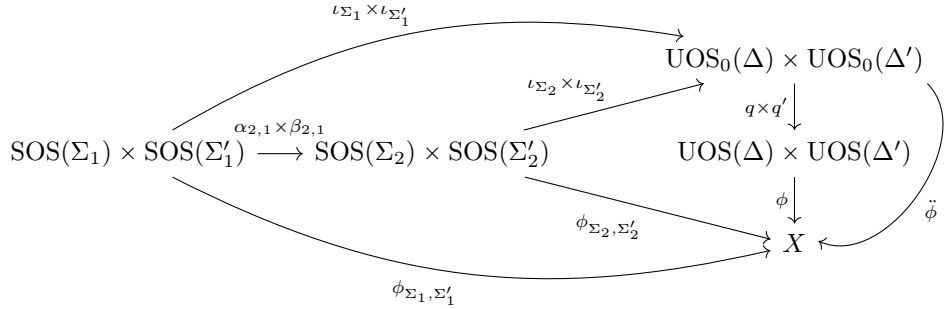

\smallskip
\noindent
\emph{Step 3.}
We now apply the universal property of the maximal tensor product again to see that the map 
\begin{align}
\psi : \mathrm{UOS}(\Delta) \otimes_\mathrm{max} \mathrm{UOS}(\Delta') \rightarrow X; \ \
p \otimes q \mapsto \phi(p,q),
\end{align}
is completely positive and satisfies the universal property of the colimit of the diagram (\ref{eqn:Diagram-Max-Tensor-Product-SOS-full}).
To summarise, we have just shown the first complete order isomorphism displayed below.
The second complete order isomorphism follows from \autoref{prop:SOS-max-tensor-product}, and the third one from \autoref{prop:Colimit}:
\begin{align}
\mathrm{UOS}(\Delta) \otimes_\mathrm{max} \mathrm{UOS}(\Delta')
&\cong \underset{\underset{\Sigma \in \mathcal{S}(\Delta), \Sigma' \in \mathcal{S}(\Delta')}{\longrightarrow}}{\lim} \mathrm{SOS}(\Sigma) \otimes_\mathrm{max} \mathrm{SOS}(\Sigma') \\
&\cong \underset{\underset{\Sigma \in \mathcal{S}(\Delta), \Sigma' \in \mathcal{S}(\Delta')}{\longrightarrow}}{\lim} \mathrm{SOS}(\Sigma \times \Sigma') 
\cong \mathrm{UOS}(\Delta \times \Delta')
\end{align}
This concludes the proof.
\end{proof}

\subsection{The operator extension properties for positivity domains}

We now show that the complete factorisation property is equivalent to a well-known extension property for partially defined positive semi-definite functions. 
	
\begin{definition}
Let $\Gamma$ be a discrete group, $\Delta\subseteq \Gamma$ be a positivity domain and $H$ a Hilbert space.
A function $u : \Delta \to \cl B(H)$ is called \emph{positive semi-definite} if $(u(st^{-1}))_{s,t\in\Sigma} \in \mathrm{M}_\Sigma(\mathcal{B}(H))^+ \subseteq \mathcal{B}(\ell^2(\Sigma) \otimes H)$ whenever $\Sigma \subseteq \Gamma$ is a finite subset with $\Sigma\Sigma^{-1} \subseteq \Delta$.
\end{definition}

We denote the cone of all positive semi-definite $\mathcal{B}(H)$-valued functions on $\Delta$ by $\mathrm{B}(\Delta,H)^+$ and its linear span by $\mathrm{B}(\Delta,H)$.
We call $\mathrm{B}(\Delta,H)$ the \emph{$\mathcal{B}(H)$-valued Fourier--Stieltjes space of (the positivity domain) $\Delta$}, and $\mathrm{B}(\Delta) := \mathrm{B}(\Delta,\mathbb{C})$ the \emph{Fourier--Stieltjes space of $\Delta$}.

\begin{proposition}\label{prop:funls}
Let $\Gamma$ be a discrete group and $\Delta \subseteq \Gamma$ a positivity domain. 
\begin{itemize}
\item[(i)] 
If $u\in \mathrm{B}(\Delta,H)^+$ is a positive semi-definite function on $\Delta$ then there exists a (necessarily) unique completely positive map $\phi_u : \mathrm{UOS}_0(\Delta) \to \cl B(H)$ such that $\phi_u(\delta_s) = u(s)$, for all $s\in \Delta$. 
			
\item[(ii)] If  $\phi : \mathrm{UOS}_0(\Delta) \to \cl B(H)$ is a  completely positive map then there exists a (necessarily) unique positive semi-definite function $u_\phi : \Delta \rightarrow \mathcal{B}(H)$ such that $u_\phi(s) = \phi(\delta_s)$, for all $s \in \Delta$. 
\end{itemize}

\end{proposition}

\begin{proof}
(i) 
Let $u\in \mathrm{B}(\Delta,H)^+$.
By the charaterisation of $\mathrm{M}_n(\mathrm{UOS}_0(\Delta))^+ = \mathcal{U}_n(\Delta)$ as the smallest compatible family of positive matrix-cones containing the set $$\{(\delta_{st^{-1}})_{s,t\in\Sigma} \mid \Sigma \subseteq \Gamma, \Sigma\Sigma^{-1} \subseteq \Delta, \lvert\Sigma\rvert < \infty\}$$ from \autoref{rem:Delta-cones-MOU} and by the same computation as in (\ref{eqn:T-minors-amplified-phi_u}), the induced map $\phi_u : \mathrm{UOS}_0(\Delta) \rightarrow \mathcal{B}(H)$ is completely positive.
		
\smallskip
\noindent
(ii) 
Conversely if $\phi : \mathrm{UOS}_0(\Delta) \to \cl B(H)$ is a  completely positive map, let $\Sigma \subseteq \Gamma$ be a finite subset with $\Sigma\Sigma^{-1} \subseteq \Delta$.
Then we have
\begin{align}
(u_\phi(st^{-1}))_{s,t\in\Sigma}
= (\phi(\delta_{st^{-1}}))_{s,t\in\Sigma}
= \phi^{(\lvert\Sigma\rvert)}((\delta_{st^{-1}})_{s,t\in\Sigma})
\in \mathcal{B}(H)^+,
\end{align}
so $u_\phi \in \mathrm{B}(\Delta,H)^+$.
		
The uniqueness statements follow since the assignments $u \mapsto \phi_u$ and $\phi \mapsto u_\phi$ are inverse to each other.
\end{proof}
	
Since the matrix ordered vector space dual of an operator system $X$ is by definition equipped with the positive matrix-cones $\mathrm{M}_n(X^{\mathrm{d}})^+ := \mathcal{CP}(X,\mathrm{M}_n)$, the following is immediate:

\begin{corollary}\label{cor:B-Delta-MOVS}
The pair $(\mathrm{B}(\Delta),(\mathrm{B}(\Delta,\mathbb{C}^n)^+)_n)$ is the matrix ordered vector space dual of the operator system $\mathrm{UOS}(\Delta)$.
\end{corollary}

The identification 
\begin{align}
\mathrm{B}(\Delta,H)
:= \mathrm{span} (\mathrm{B}(\Delta,H)^+) 
\cong \mathrm{span}(\mathcal{CP}(\mathrm{UOS}(\Delta),\mathcal{B}(H))) 
= \mathcal{CB}(\mathrm{UOS}(\Delta),\mathcal{B}(H))
\end{align}  
endows the Fourier--Stieltjes space $\mathrm{B}(\Delta)$ with (matrix ordered) operator space structure.

\begin{definition}
Let $\Gamma$ be a discrete group.
A positivity domain $\Delta \subseteq \Gamma$ is said to possess the \emph{positive operator extension property (as a positivity domain)} if, for every Hilbert space $H$, the following condition holds:
\begin{quote}
Condition (PEPD): \quad For every positive semi-definite function $u : \Delta \rightarrow \mathcal{B}(H)$, there exists a positive semi-definite function $\tilde{u} : \Gamma \rightarrow \mathcal{B}(H)$ such that $\tilde{u}|_\Delta = u$.
\end{quote}

Similarly, we say that $\Delta$ possesses the \emph{cp extension property (as a positivity domain)} if condition (PEPD) holds for every finite dimensional Hilbert space $H$.
\end{definition}

\begin{definition}
Let $\Gamma$ be a discrete group.
A positivity domain $\Delta \subseteq \Gamma$ is said to  the \emph{bounded operator extension property (as a positivity domain)} if, for every Hilbert space $H$, the following condition holds:
\begin{quote}
Condition (BEPD): \quad 
For every positive semi-definite function $u : \Delta \rightarrow \mathcal{B}(H)$, there exists a positive semi-definite function $\tilde{u} : \Gamma \rightarrow \mathcal{B}(H)$ such that $\tilde{u}|_\Delta = u$ and $\lVert \tilde{u} \rVert_{\mathrm{B}(\Gamma,H)} = \lVert u \rVert_{\mathrm{B}(\Delta,H)}$.
\end{quote}

Similarly, we say that $\Sigma$ possesses the \emph{cb extension property} if condition (BEPD) holds for every finite dimensional Hilbert space $H$.
\end{definition}

\begin{remark}\label{rmk:BW-argument-operator-extension-positivity-domain}
Analogues of \autoref{rmk:BW-argument-positive-operator-extension} and \autoref{rmk:BW-argument-bounded-operator-extension} apply, showing that the positive operator and cp extension properties, as well as the bounded operator and cb extension properties (all as positivity domains) are respectively equivalent.
\end{remark}

\begin{remark}\label{rmk:Extension-property-quotient-map}
By \autoref{prop:funls} the canonical ucp inclusion map $\iota : \mathrm{UOS}(\Delta) \to \mathrm{C}^*(\Gamma)$ induces the canonical surjective cp restriction map $\rho : \mathrm{B}(\Gamma) \rightarrow \mathrm{B}(\Delta)$ as its dual map.
In analogy to \autoref{prop:Sigma-ext-property-quotient-map} and \autoref{prop:Sigma-ext-property-quotient-map-2}, the positivity domain $\Delta$ has the cp/cb extension property (as a positivity domain) if and only if the restriction map $\rho : \mathrm{B}(\Gamma) \rightarrow \mathrm{B}(\Delta)$ is an $\mathbf{MVS}$-/$\mathbf{OSp}$-quotient map.
\end{remark}

\begin{theorem}\label{th_main}
Let $\Gamma$ be a discrete group and $\Delta \subseteq \Gamma$ a positivity domain. 
The following are equivalent:
\begin{enumerate}
\item[(i)] The positivity domain $\Delta$ has the positive operator extension property (as a positivity domain).
\item[(ii)] The positivity domain $\Delta$ has the bounded operator extension property (as a positivity domain).
\item[(iii)] The positivity domain $\Delta$ has the complete factorisation property.
\end{enumerate}
\end{theorem}
	
\begin{proof}
Similarly to the proof of \autoref{thm:Characterizations-CPEP}, note that the dual map of the canonical map $\iota : \mathrm{UOS}(\Delta) \rightarrow \mathrm{C}^*(\Gamma)$ is the restriction map $\rho : \mathrm{B}(\Gamma) \rightarrow \mathrm{B}(\Delta)$, where $\mathrm{B}(\Gamma)$ and $\mathrm{B}(\Delta)$ are, at the same time, the matrix ordered vector space and operator space duals of $\mathrm{C}^*(\Gamma)$ and $\mathrm{UOS}(\Delta)$ respectively.
Recall from \autoref{rmk:BW-argument-operator-extension-positivity-domain} and \autoref{rmk:Extension-property-quotient-map} that the positivity domain $\Delta$ has the positive/bounded operator extension property (as a positivity domain) if and only if the restriction map $\rho$ is an $\mathbf{MVS}$-/$\mathbf{OSp}$-quotient map.
By \autoref{prop:Equiv-embeddings-quotients}, (i) and (iii) are both equivalent to the statement that the map $\iota$ is a complete order embedding, that is, to (iii), in view of \autoref{rmk:FP-equiv-embedding}.
\end{proof}

Let $\Gamma$ be a discrete group and $\Sigma \subseteq \Gamma$ be a finite subset.
 It follows from the definitions that if $\Sigma$ has the operator extension/Fej\'er--Riesz property, then $\Sigma\Sigma^{-1}$ has the operator extension/factorisation property \emph{as a positivity domain}, see \autoref{fig:CP-maps}.
 The following example demonstrates that the converse is not true in general.

\begin{example}
Consider the subset $\Sigma := \{0,1,3\} \subseteq \mathbb{Z}$ of the integers, and set $\Delta := \Sigma-\Sigma = \{-3,\dots,3\}$.
Also set $\Sigma' := \{0,1,2,3\}$ and note that $\Sigma'-\Sigma' = \Delta$.
By \cite[Theorem 4.8]{Gab98}, a subset of $\mathbb{Z}$ has the extension property if and only if it is an arithmetic progression.
Hence $\Delta = \Sigma-\Sigma$ as above has the operator extension property (by the operator valued Fej\'er--Riesz lemma applied to $\Sigma'$), whereas $\Sigma$ does not.
\end{example}


\end{document}